\ifx\pdfoutput\undefined\else\pdfoutput=1\fi
\documentclass[9pt,a4paper,oneside,reqno]{amsart} 
\input{packages_and_notation.tex}
\author{Giorgio Cipolloni$^{\dagger\ddagger}$ \and L\'aszl\'o Erd\H{o}s$^{\dagger}$ \and Dominik Schr\"oder$^{\dagger\ast}$}
\address{IST Austria, Am Campus 1, A-3400 Klosterneuburg, Austria}
\email{dschroed@ist.ac.at}
\email{giorgio.cipolloni@ist.ac.at}
\email{lerdos@ist.ac.at} 
\thanks{$^\dagger$Partially supported by ERC Advanced Grant No.~338804}\thanks{$^\ddagger$This project has received funding from the European Union's Horizon 2020 research and innovation programme under the Marie Sk\l odowska-Curie Grant Agreement No. 665385.}\thanks{$^\ast$Current affiliation: ETH Institute for Theoretical Studies}
\title[Least singular value of the shifted Ginibre ensemble]{Optimal lower bound on the least singular value of the shifted Ginibre ensemble}
\date{\today}

\begin{document}
\thispagestyle{empty}
\maketitle
  
\begin{abstract}
    We consider the least singular value of a large random matrix with real or complex
    i.i.d.\ Gaussian entries shifted by a constant $z\in \C$. We prove
    an optimal  lower tail  estimate on this singular value in the  critical regime where
    $z$ is around the spectral edge thus improving the classical bound of Sankar, Spielman  and Teng~\cite{MR2255338} for the particular shift-perturbation in the edge regime. Lacking Br\'ezin-Hikami formulas in the real case, we rely on the  superbosonization formula~\cite{MR2430637}.
\end{abstract}

\section{Introduction}
    The effective numerical solvability of a large system of linear equations  \(Ax=b\) is determined by the condition number
    of the matrix \(A\). In many practical applications the norm of \(A\) is bounded and thus the condition number critically
    depends on the smallest singular value \(\sigma_1(A)\) of \(A\).  When the matrix elements of \(A\) come from  noisy measured data, 
    then the lower tail probability of \(\sigma_1(A)\) tends to exhibit a universal scaling behavior, depending on the 
    variance of the noise. In the simplest case \(A\) can be decomposed as 
    \begin{equation}\label{AX}
    A= A_0 + X,
    \end{equation}
     where \(A_0\) is a deterministic  square matrix and \(X\) is drawn from the 
    \emph{Ginibre ensemble}, 
    i.e.~\(X\) has i.i.d.~centred Gaussian matrix elements with variance \(\E \abs{x_{ij}}^2 = N^{-1}\),
    where \(N\) is the dimension. 

    The randomness in \(X\) smoothens out possible singular
    behavior of \(A^{-1}\). In particular Sankar, Spielman and Teng~\cite{MR2255338} showed that the smallest singular value \(\sigma_1(A)\),
     lives on a scale not smaller than \(N^{-1}\), equivalently, the smallest eigenvalue \(\lambda_1(AA^\ast)\) of \(AA^\ast\)  
     lives on a scale \(\le N^{-2}\), i.e.
    \begin{equation}\label{SST}
    \P (\lambda_1(AA^\ast)= [\sigma_1(A)]^2  \le x N^{-2}) \lesssim \sqrt{x}, \quad \text{for any}\quad x>0,
    \end{equation}
    up to logarithmic corrections, uniformly in \(A_0\). If \(X\) is a 
    complex Ginibre matrix, then the \(\sqrt{x}\) bound  improves to \(x\). 
      
The special case \(A_0=0\) shows that the bound~\eqref{SST} is essentially optimal. Indeed, the tail probability of \(\lambda_1(XX^*)\) of real and complex Ginibre ensembles has been explicitly computed by Edelman~\cite{MR964668} as 
    \begin{equation}\label{Gin}
    \lim_{N\to\infty}\P ( \lambda_1(XX^*) \le x N^{-2}) =  \begin{cases}
        1-e^{-x/2-\sqrt{x}}=\sqrt{x}+ \landauO{x} ,& \text{in the real case}\\
        1-e^{-x} = x +\landauO{x^2},& \text{in the complex case.}
    \end{cases} 
    \end{equation}
The complex Ginibre ensemble has a stronger smoothing effect in~\eqref{Gin} is due to the additional degrees of freedom. 
This observation is analogous to the different strength of the level repulsion in  real symmetric and complex Hermitian random matrices.
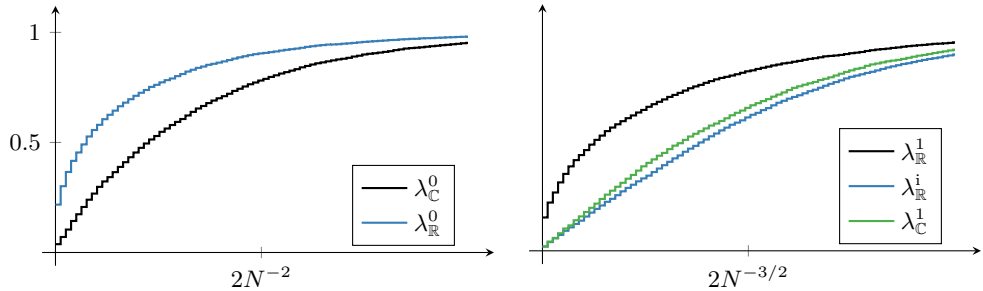
\begin{figure}
    \centering
    \begin{tikzpicture} 
        \pgfplotstableread[col sep = comma]{50.csv}\mydata
        \begin{axis}[
            xmin = -0,
            xmax = 4,
            ymin=0,
            ymax=1,
            width=7cm,height=4.5cm,
            tick align=outside,
            axis x line=middle,
            axis y line=middle,
            axis line style={shorten >=-10pt, shorten <=-5pt},
            legend pos = south east,
            xtick={2},
            xticklabels={$2N^{-2}$},
            ytick={0,0.5,1},
            cycle list/Set1-3,
        ]

        \path[name path=axis] (axis cs:0,0) -- (axis cs:4,0);
         
        \addplot+ [name path=A,const plot,thick,draw=black]  table[x index = {0}, y index = {1}]{\mydata};
        \addlegendentry{$\lambda_{\C}^{0}$};
        \addplot+ [name path=B,const plot,thick] table[x index = {0}, y index = {2}]{\mydata};
        \addlegendentry{$\lambda_{\R}^{0}$};
        \end{axis}
        \end{tikzpicture}\qquad\quad
        \begin{tikzpicture}
            \pgfplotstableread[col sep = comma]{50e.csv}\mydata
            \begin{axis}[
                xmin = -0,
                xmax = 4,
                ymin=0,
                ymax=1,
                width=7cm,height=4.5cm,
                tick align=outside,
                axis x line=middle,
                axis y line=middle,
                axis line style={shorten >=-10pt, shorten <=-5pt},
                legend pos = south east,
                xtick={2},
                xticklabels={$2N^{-3/2}$},
                ytick={0},
                cycle list/Set1-3,
            ]
    
            \path[name path=axis] (axis cs:0,0) -- (axis cs:4,0);
             
            \addplot+ [name path=A,const plot,thick,draw=black]  table[x index = {0}, y index = {1}]{\mydata};
            \addlegendentry{$\lambda_{\R}^{1}$};
            \addplot+ [name path=B,const plot,thick] table[x index = {0}, y index = {2}]{\mydata};
            \addlegendentry{$\lambda_{\R}^{\ii}$};
            \addplot+ [name path=C,const plot,thick] table[x index = {0}, y index = {3}]{\mydata};
            \addlegendentry{$\lambda_{\C}^{1}$};
                     
             \end{axis}
             \end{tikzpicture}
        \caption{Plots of the cumulative histograms of the smallest eigenvalue \(\lambda_{\R,\C}^{z}\) of the matrix \((X-z)(X-z)^\ast\), where \(\R,\C\) indicates whether \(X\) is distributed according to the real or complex Ginibre ensemble. 
                The data was generated by sampling \(5000\) matrices of size \(200\times 200\). The first plot confirms the difference between the \(x\)- and \(\sqrt{x}\)-scaling close to \(0\), see~\eqref{Gin}.  
                The second plot shows that this difference is also observable for shifted Ginibre matrices at the edge \(|z|=1\), 
                but only for real spectral parameters \(z=\pm 1\). When the complex parameter \(z\) is away from the real axis, then the real case behaves similarly
                to the complex case.}\label{fig}
\end{figure}     

The support of the spectrum of such \emph{information plus noise matrices} \(AA^*\) becomes deterministic 
as $N\to\infty$ and it can be computed from the solution of  a certain self-consistent equation~\cite{MR2322123}.
 Almost surely no eigenvalues lie outside the support of the limiting measure~\cite{MR2930382}. 
Thus  $\lambda_1(AA^*)$ has a simple $N$-independent positive lower bound if 0 is away from  this support.
However, when 0 is well inside the limiting spectrum, the smoothing mechanism becomes important 
yielding that $\lambda_1(AA^*)$ is of order $N^{-2}$
with a lower tail given in~\eqref{SST}.  The regime where $0$ is near the edge of  this support is yet unexplored.

The goal of this paper is to study this transitional regime for \(A=X-z\), i.e.~for the important special case where \(A_0=-z\1\) is a constant multiple of the idenity matrix, as the spectral parameter \(z\in\C\) is varied. The limiting density of states of \(Y^z\defeq (X-z)(X-z)^\ast\) is supported in the interval \([0,\ed_+]\) for \(\abs{z}\le 1\) and the interval \([\ed_-,\ed_+]\) with \(\ed_->0\) for \(\abs{z}>1\), where \(\ed_\pm\) are explicit functions of \(\abs{z}\)
given in~\eqref{ed pm}. As noted above, the problem is relatively simple if \(\abs{z}\ge 1+\epsilon\)  with 
some \(N\)-independent \(\epsilon\) as in this case~\cite{MR2930382} 
implies that almost surely \( \lambda_1(Y^z)\ge C(\epsilon)>0\) is bounded away from zero. In the opposite regime, when $\abs{z}\le 1-\epsilon$, then typically \(\lambda_1(Y^z)\sim N^{-2}\), and in fact~\eqref{SST} provides the correct corresponding upper bound (modulo logs).

Our main result on the tail probability of \(\lambda_1(Y^z)\) is that for \(\abs{z}\le 1+ C N^{-1/2}\)
\begin{subequations}
    \begin{equation}\label{our lambda}
        \P \Big(\lambda_1 (Y^z) \le x \cdot c(N,z)\Big)\lesssim \begin{cases}
        x+ \sqrt{x} e^{- \frac{1}{2} N (\Im z)^2},& \text{in the real case}\\
        x,& \text{in the complex case,}
        \end{cases} 
        \end{equation}            
\end{subequations}
where 
\begin{equation}\label{cN def}
c(N,z)\defeq \min \Bigl\{\frac{1}{N^{3/2}}, \frac{1}{N^2 \abs[0]{1-\abs{z}^2}}\Bigr\}.
\end{equation}
Our bound is sharp up to logarithmic corrections, see Corollary~\ref{cor:tailas} for the precise statement. Notice the transition between the \(x\) and \(\sqrt{x}\) behaviour in the real case of~\eqref{our lambda}: near the real axis, \(\abs{\Im z}\ll N^{-1/2}\), the result is analogous to the real case~\eqref{Gin} at \(z=0\), otherwise the complex behaviour~\eqref{Gin} dominates at the edges even for real \(X\), see Fig.~\ref{fig}. These results reveal how the robust bound~\eqref{SST} improves near the spectral edge in the transition regime  \(- CN^{-1/2} \le 1-\abs{z} \ll 1\) in both symmetry classes. The transition to the Tracy-Widom scaling in the regime well outside of the spectrum \(\abs{z}-1 \gg N^{-1/2}\) is deferred to our future work.

One motivation for studying \(X- z\) is the classical ODE model \(\diff u/\diff t = (X-z)u\)  on the stability of large biological networks by May~\cite{4559589}. 
For example, the matrix elements $x_{ij}$ may express  random connectivity rates between neurons and $z$ is the overall decay rate
of neuron activation~\cite{10039285}. As $\Re z$ crosses 1, there is a fine phase transition in the large time behavior of $u$ that depends on whether $X$ is real or complex Ginibre matrix, see~\cite{PhysRevLett.81.3367} and~\cite{MR3816180} for the recent mathematical results, as well as for further references.
Another important motivation is that an effective lower tail bound on the least singular value of $X-z$ is essential for 
the proof of the circular law via Girko's formula, see~\cite{MR2908617} for a detailed survey. In fact, this is the most delicate ingredient in any proof concerning eigenvalue distribution of large non-Hermitian matrices. In particular, relying on the main result of 
the current paper, we proved~\cite{1908.00969} that the local eigenvalue statistics for  random matrices with centered i.i.d.~entries near the spectral edge asymptotically coincide  with those for the corresponding Ginibre ensemble as \(N\to\infty\). This is the non-Hermitian analogue of the celebrated Tracy-Widom edge  universality for Wigner matrices~\cite{MR1727234,MR3253704}. Similarly, the singular value bound from the present paper is also an important ingredient for the recent CLTs for complex and real i.i.d.\ matrices~\cite{1912.04100,2002.02438}.

We now give a brief history of related results. In the \( z=0\) case tail estimates for \(\lambda_1(XX^*)\) beyond the Gaussian distribution 
 have been subject of intensive research~\cite{MR2434885,MR2480613} eventually obtaining~\eqref{Gin} with an additive \(\landauO{e^{-cN}}\) error term for any \(X\) with i.i.d.\ entries with subgaussian tails in~\cite{MR2407948}.
 The precise distribution of \(\lambda_1(XX^*)\) was shown in~\cite{MR2647142}  to coincide with the Gaussian case~\eqref{Gin} under a bounded high moment condition
 and with an  \(\landauO{N^{-c}}\) error  term, see also~\cite{MR3916329,1908.04060} for more general ensembles. In the case of general \(A_0\) lower bounds on \(\lambda_1(AA^*)\) in the non-Gaussian setting have been obtained in~\cite{MR2402448,MR2684367}, albeit not uniformly in \(A_0\),
  see also~\cite{MR3857860,1707.09656} beyond the i.i.d.\ case. We are not aware of any previous results improving~\eqref{SST} in the transitional regime~\eqref{our lambda}. 

Since we consider Ginibre (i.e.\ purely Gaussian) ensembles, one might think that everything is explicitly computable from the well understood spectrum of \(X\). The eigenvalue density of \(X\) converges to the uniform distribution on the unit disk and the spectral radius of X converges to 1 (these results have also been
established for the general non-Gaussian case, cf.~Girko's circular law~\cite{MR773436,MR1428519,MR2409368,MR866352,MR863545,MR3813992}). Also the joint probability density function of all Ginibre eigenvalues, as well as their local correlation functions are explicitly known; see~\cite{MR173726} and~\cite{MR0220494} for the relatively simple complex case, and~\cite{MR1121461,MR1437734,MR2530159,17930739} for the more involved real case, where the appearance of \(\sim N^{1/2}\) real eigenvalues causes a singularity in the
local density. However, eigenvalues of \(X\) give no direct information  on the singular values of \(X-z\)
and the extensive literature on the Ginibre spectrum is not applicable. Notice that any intuition based upon the eigenvalues of \(X\) is misleading: the nearest eigenvalue to \(z\) is at a distance
of order \(N^{-1/2}\) for any \(\abs{z}\le 1\). However, \(\norm{(X-z)}^{-1}\sim \max\{N^{3/4},N\abs[0]{1-\abs{z}^2}^{1/2}\}\) for \(\abs{z}\le  1+C N^{-1/2}\),  
as a consequence of our result~\eqref{our lambda}. This is an indication that typically \(X\) is highly non-normal (another indication
is that the largest singular value of \(X\) is \(2\), while its spectral radius is only \(1\)).  
    
Regarding our strategy, in this paper we 
use supersymmetric methods to express the resolvent of \(Y^z\). In particular, we use a multiple Grassmann integral formula for
\begin{equation}\label{res}
    \varrho_N^z(E)\defeq \frac{1}{N\pi } \Im \E \Tr \frac{1}{Y^z- E + \ii 0},
\end{equation}
the averaged density of states (or one-point function) of \(Y^z\) at energy \(E\in\R\). For  \(\abs{E} \lesssim c(N,z)\) a sizeable contribution to~\eqref{res} comes from the lowest eigenvalue \(\lambda_1(Y^z)\), hence a good upper estimate on~\eqref{res} translates into a lower tail bound on \(\lambda_1(Y^z)\).
    
With the help of the superbosonization formula by Littelmann, Sommers and Zirnbauer~\cite{MR2430637}, we can drastically reduce the number of integration variables: instead of \(N\) bosonic and \(N\) fermionic variables we will have an explicit expression for~\eqref{res} involving merely two contour integration  variables in complex case and  three in the real case. The remaining integrals are
still highly oscillatory, but contour deformation allows us to estimate them optimally.
In fact, saddle point analysis  identifies the leading term as long as \(\abs{E}\gg c(N, z)\). However, in the
critical  regime, \(\abs{E}\lesssim c(N, z)\), the saddle point analysis breaks down.
The leading term is extracted as a specific rescaling of a universal function given by a double integral.
We work out the precise answer for~\eqref{res} in the complex case and we provide optimal bounds in the real case, deferring the precise asymptotics to further work.
    
Lower tail estimates require delicate knowledge about individual eigenvalues, i.e.~about the density of states
below the scale of eigenvalue spacing,
and it is crucial to exploit  the Gaussianity of \(X\) via explicit formulas. There are essentially three methods: (i) orthogonal polynomials, (ii) Br\'ezin-Hikami contour integration formula~\cite{MR1662382} and (iii) supersymmetric formalism. We are not aware of any orthogonal polynomial approach to analyse  \(Y^z = (X-z)(X-z)^*\) in the real case
(see~\cite{MR2250019} in the complex case and~\cite{MR2969495}
for rank-1 perturbation of real \(X\)).
In the complex case, the ensemble \(Y^z\) has also
been extensively investigated by the Br\'ezin-Hikami formula in~\cite{MR2162782}, where even  the determinantal correlation kernel was computed as a double integral involving the Bessel kernel, see also~\cite{MR1419177,MR1413913} 
for a derivation via the supersymmetric version of the Itzykson-Zuber formula. Although the paper~\cite{MR2162782} did not analyse the resulting  
one point function, well known asymptotics for the
Bessel function may be used to rederive our bounds and asymptotics on~\eqref{res}, as well as~\eqref{our lambda}, from~\cite[Theorem 7.1]{MR2162782},
see Appendix~\ref{app} for more details. 
For the real case, however, there is no analogue of the Br\'ezin-Hikami formula. 
     
Therefore, in this paper we explore the last option, the
supersymmetric approach, that is available for both symmetry classes, albeit the real case is
considerably more involved. Our main tool is the powerful superbosonization formula~\cite{MR2430637}
followed by a delicate multivariable contour integral analysis.
We remark that, alternatively, one may also use  the Hubbard-Stratonovich transformation, e.g.~\cite[Proposition~1]{MR4001834} where correlation functions, i.e.\ expectations of \emph{products}  of characteristic polynomials of \(X\) were computed in this way. Note, however,  that the density of states~\eqref{res} requires to analyse \emph{ratios} of determinants, a technically much
more demanding task. While explicit formulas can be obtained with both methods (see~\cite{MR2975518} and especially~\cite{MR2512124} for an explicit comparison), the subsequent analysis seems to be more feasible with the formula obtained from the superbosonisation approach, as our work demonstrates. 
 
Supersymmetry is a compelling method originated in physics~\cite{MR1628498,MR2932627,MR820690} to produce surprising identities related to random matrices whose potential has not yet been fully exploited
in mathematics. It has been especially successful in deriving rigorous result on  Gaussian random band matrices~\cite{MR3627427,MR3665217,1810.13150,MR1942858,1905.08136,MR3196980,MR3623245,MR3541181,1910.02999}, sometimes  even beyond the Gaussian case~\cite{MR2842968,MR3055265,MR2932643}, as well as on overlaps of non-Hermitian Ginibre eigenvectors~\cite{MR3851824}. We also mention
 the  recent results in~\cite{MR3851824} and~\cite{MR3780342} as examples of a remarkable interplay between supersymmetric and orthogonal polynomial techniques in the theory of Ginibre and related matrices.

The main object of our work, the  Hermitian block random matrix
\begin{equation}\label{Hdef}
H=H^z\defeq \begin{pmatrix} 
    0 & X -z\\
    X^\ast-\bar{z} & 0
\end{pmatrix} 
\end{equation}
arose in the physics literature as a chiral random matrix model for massless Dirac operator,
introduced by Stephanov in~\cite{10061300}.
Typically,  instead of $z$ and $\bar z$, both shift parameters are chosen equal $z$ (interpreted as $\ii$-times the chemical potential)
so that the corresponding $H$  is not self-adjoint; this model has been extensively investigated by
both supersymmetric and orthogonal polynomial techniques, see e.g.~\cite{10053982,MR2128558,PhysRevD.57.6486,MR1488590,OSBORN1999317}.
 However, in the special case when $z$ is real, our $H^z$ as given in~\eqref{Hdef}
 coincides with Stephanov's model where $z$ 
 can be interpreted as temperature (or Matsubara frequency), see~\cite[Section~6.1]{doi:10.1146/annurev.nucl.50.1.343}.

\subsubsection*{Acknowledgement} The authors are grateful to Nicholas Cook and Patrick Lopatto for pointing out missing references, and to Ievgenii Afanasiev for useful remarks. We would also like to thank the anonymous referees for drawing our attention to additional references in the physics literature.

\section{Model and main results}
We consider the model \(Y=Y^z=(X-z)(X^\ast-\overline{z})\) with a fixed complex parameter \( z\in \C\) and  with
a random matrix  \(X\in\C^{N\times N}\) having independent real or complex Gaussian entries \(x_{ab}\sim\mathcal{N}(0,N^{-1})\), where in the complex case we additionally assume \(\E x_{ab}^2=0\). Note that \(Y\) is related to the block matrix~\eqref{Hdef}
through its resolvent via 
\begin{equation} \frac{\Tr (H-\sqrt{w})^{-1}}{2\sqrt{w}}=\Tr (Y-w)^{-1}, \quad \Re w>0, \; \Im w > 0,
\label{eq H Y cor}\end{equation}
where the branch of \(\sqrt{w}\) is chosen such that \(\Im \sqrt{w} >0\).
It is well known that in the large \(N\) limit the normalized trace of the resolvent of many  random matrix ensembles
 becomes deterministic and it satisfies an algebraic equation, the matrix Dyson equation (MDE)~\cite{MR3916109}.
  In the current case of i.i.d.~entries the MDE reduces to a simple cubic  scalar equation
\begin{equation} \frac{1}{m_{H^z}}+(w+m_{H^z})-\frac{\abs{z}^2}{w+m_{H^z}}=0,\quad\Im m_{H^z}(w)>0,\quad\Im w>0 
\label{eq MDE}\end{equation}
that has a unique solution, denoted by \(m_{H^z}\). The local law from~\cite{1907.13631} asserts that 
\begin{equation}\frac{1}{2N}\Tr (H^z-w)^{-1}=  m_{H^z}(w)+\mathcal{O}_\prec\bigl((N\Im w)^{-1}\bigr), \label{local law}
\end{equation} where \(\mathcal{O}_\prec\) denotes a suitable concept of high-probability error term. Together with~\eqref{eq H Y cor} it follows that the normalized trace of the
   resolvent \( (Y^z-w)^{-1}\) of \(Y^z\) is well approximated \[\frac{1}{N}\Tr (Y^z-w)^{-1}\approx m^z(w)\] 
   by the unique solution \(m=m^z=m_{Y^z}\) to the equation 
\begin{equation}  \frac{1}{m^z} + w(1+m^z) - \frac{\abs{z}^2}{1+m^z}=0, \quad \Im m^z(w)>0, \quad \Re w >0, \; \Im w>0, \label{MDE Y}\end{equation}
which is given by \( m^z(w)= m_{H^z}(\sqrt{w})/\sqrt{w}\). Since \(m\) approximates the trace of the resolvent, the density of states is obtained as the imaginary part of the continuous extension of \(m\)
to the real line, i.e.~\(\varrho_\#(E) = \pi^{-1}\lim_{\epsilon\to 0+}\Im m_\#(E+\ii \epsilon) \) for both choices \( \# = H^z, Y^z \).
 For \(\delta\defeq 1-\abs{z}^2\approx 0\) the Stieltjes transform \(m_{H^z}\) and its density of states exhibit
  a cusp formation at \(w=0\) as $\delta$ crosses the value 0. This cusp formation in 
\(H^z\) implies an analogous transition for \( m^z\); the corresponding density of states are depicted in Figure~\ref{fig densities}. 

\begin{figure}
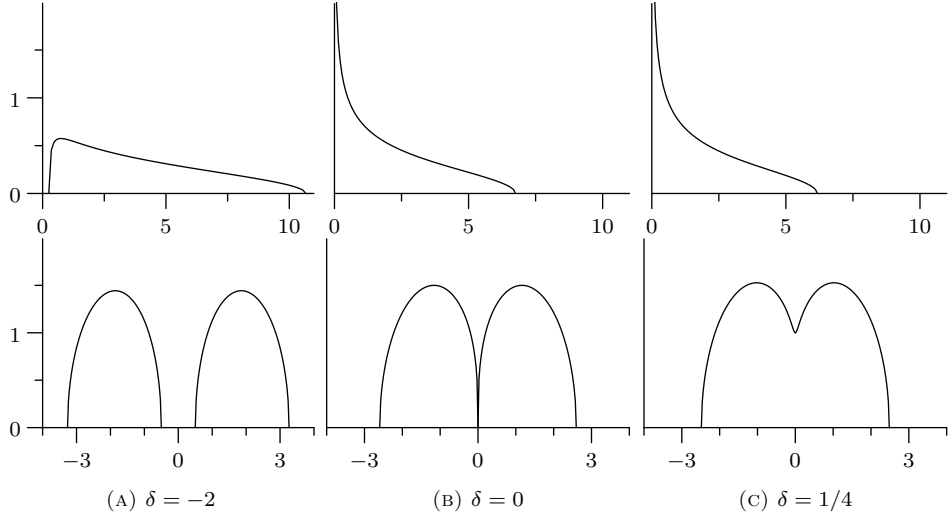
 
    \centering
    \begin{subfigure}[b]{.33\linewidth}
        \centering
        \begin{asy}
            size(4cm,3cm,IgnoreAspect);
            var delta = -2;
            real cr(real x) { return abs(x)^(1/3)*sgn(x); }
            real p(real e) { return delta/3/e - 1/9 ; }
            real q(real e) { return delta/3/e + 1/27 -1/2/e; }
            pair m(real e) { return -2/3 + exp(2*pi*I/3)*cr(q(e)+sqrt(q(e)^2+p(e)^3))+exp(-2*pi*I/3)*cr(q(e)-sqrt(q(e)^2+p(e)^3)); }
            real imm(real e) {return ypart(m(e));}
            draw(graph(imm,.25,32/3));
            xaxis(Bottom,xmin=0,xmax=11,RightTicks(Step=5,step=2.5));
            yaxis(Left,0,1.99,LeftTicks(Step=1,step=.5));
            xlimits(0,11);
            ylimits(-.6,2,Crop);
        \end{asy}  
    \end{subfigure}%
    \begin{subfigure}[b]{.33\linewidth}
        \centering
        \begin{asy}
            size(4cm,3cm,IgnoreAspect);
            var delta = 0;
            real cr(real x) { return abs(x)^(1/3)*sgn(x); }
            real p(real e) { return delta/3/e - 1/9 ; }
            real q(real e) { return delta/3/e + 1/27 -1/2/e; }
            pair m(real e) { return -2/3 + exp(2*pi*I/3)*cr(q(e)+sqrt(q(e)^2+p(e)^3))+exp(-2*pi*I/3)*cr(q(e)-sqrt(q(e)^2+p(e)^3)); }
            real imm(real e) {return ypart(m(e));}
            draw(graph(imm,.01,27/4));
            xaxis(Bottom,xmin=0,xmax=11,RightTicks(Step=5,step=2.5));
            yaxis(Left,0,1.99);
            xlimits(0,11);
            ylimits(-.6,2,Crop);
        \end{asy}  
    \end{subfigure}%
    \begin{subfigure}[b]{.33\linewidth}
        \centering
        \begin{asy}
            size(4cm,3cm,IgnoreAspect);
            var delta = .25;
            real cr(real x) { return abs(x)^(1/3)*sgn(x); }
            real p(real e) { return delta/3/e - 1/9 ; }
            real q(real e) { return delta/3/e + 1/27 -1/2/e; }
            pair m(real e) { return -2/3 + exp(2*pi*I/3)*cr(q(e)+sqrt(q(e)^2+p(e)^3))+exp(-2*pi*I/3)*cr(q(e)-sqrt(q(e)^2+p(e)^3)); }
            real imm(real e) {return ypart(m(e));}
            draw(graph(imm,.01,6.17004));
            xaxis(Bottom,xmin=0,xmax=11,RightTicks(Step=5,step=2.5));
            yaxis(Left,0,1.99);
            xlimits(0,11);
            ylimits(-.6,2,Crop);
        \end{asy}  
    \end{subfigure}\\
    \begin{subfigure}[b]{.33\linewidth}
        \centering
        \begin{asy}
            size(4cm,3cm,IgnoreAspect);
            var delta = -2;
            real cr(real x) { return abs(x)^(1/3)*sgn(x); }
            real p(real e) { return delta/3/e - 1/9 ; }
            real q(real e) { return delta/3/e + 1/27 -1/2/e; }
            pair m(real e) { return -2/3 + exp(2*pi*I/3)*cr(q(e)+sqrt(q(e)^2+p(e)^3))+exp(-2*pi*I/3)*cr(q(e)-sqrt(q(e)^2+p(e)^3)); }
            real imm(real e) {return ypart(m(e));}
            real imm2(real e) {return 2*imm(e^2)*abs(e); }
            draw(graph(imm2,sqrt(.25),sqrt(32/3)));
            draw(graph(imm2,-sqrt(32/3),-sqrt(.25)));
            xaxis(Bottom,xmin=-4,xmax=4,RightTicks(Step=3,step=1));
            yaxis(Left,0,1.99,LeftTicks(Step=1,step=.5));
            xlimits(-4,4);
            ylimits(-.6,2,Crop);
        \end{asy}  
    \caption{\(\delta=-2\)}
    \end{subfigure}%
    \begin{subfigure}[b]{.33\linewidth}
        \centering
        \begin{asy}
            size(4cm,3cm,IgnoreAspect);
            var delta = 0;
            real cr(real x) { return abs(x)^(1/3)*sgn(x); }
            real p(real e) { return delta/3/e - 1/9 ; }
            real q(real e) { return delta/3/e + 1/27 -1/2/e; }
            pair m(real e) { return -2/3 + exp(2*pi*I/3)*cr(q(e)+sqrt(q(e)^2+p(e)^3))+exp(-2*pi*I/3)*cr(q(e)-sqrt(q(e)^2+p(e)^3)); }
            real imm(real e) {return ypart(m(e));}
            real imm2(real e) {return 2*imm(e^2)*abs(e); }
            draw(graph(imm2,0.0000001,sqrt(27/4)));
            draw(graph(imm2,-sqrt(27/4),-0.0000001));
            xaxis(Bottom,xmin=-4,xmax=4,RightTicks(Step=3,step=1));
            yaxis(Left,0,1.99);
            xlimits(-4,4);
            ylimits(-.6,2,Crop);
        \end{asy}  
    \caption{\(\delta=0\)}
    \end{subfigure}%
    \begin{subfigure}[b]{.33\linewidth}
        \centering
        \begin{asy}
            size(4cm,3cm,IgnoreAspect);
            var delta = .25;
            real cr(real x) { return abs(x)^(1/3)*sgn(x); }
            real p(real e) { return delta/3/e - 1/9 ; }
            real q(real e) { return delta/3/e + 1/27 -1/2/e; }
            pair m(real e) { return -2/3 + exp(2*pi*I/3)*cr(q(e)+sqrt(q(e)^2+p(e)^3))+exp(-2*pi*I/3)*cr(q(e)-sqrt(q(e)^2+p(e)^3)); }
            real imm(real e) {return ypart(m(e));}
            real imm2(real e) {return 2*imm(e^2)*abs(e); }
            draw(graph(imm2,0.0000001,sqrt(6.17004)));
            draw(graph(imm2,-sqrt(6.17004),-0.0000001));
            xaxis(Bottom,xmin=-4,xmax=4,RightTicks(Step=3,step=1));
            yaxis(Left,0,1.99);
            xlimits(-4,4);
            ylimits(-.6,2,Crop);
        \end{asy}  
    \caption{\(\delta=1/4\)}
    \end{subfigure}%
    \caption{Density of states of \(Y^z\) and \(H^z\) around the cusp formation. The top and bottom figures show a plot of the boundary value of 
    \(\Im m^z= \Im m_{Y^z}\) and \( \Im m_{H^z}\), respectively on the real line.}\label{fig densities}
\end{figure}

\subsection*{Complex case}
Our main result of the present paper in the complex case is an asymptotic double-integral formula for \(\E \Tr (Y-w)^{-1}\) at \(w=E+\ii 0\), \( E\ge 0\).
 In the transitional regime it is convenient to introduce the rescaled variables 
\begin{equation}\label{rescale}
\lambda\defeq E/c(N), \qquad \widetilde\delta\defeq N^{1/2}\delta, \qquad \mbox{where} \quad \delta\defeq 1-\abs{z}^2,
\end{equation}
recalling that $ c(N)=c(N,z)$ was  defined in~\eqref{cN def}. 
For \(r\ge0\) let \(\Psi=\Psi(r)\) be the unique solution to the cubic equation \(1+r \Psi+\Psi^3=0\) with \(\Re\Psi,\Im\Psi>0\). It is easy to see that \(\Psi(r)\) satisfies \(\Psi(0)=e^{\ii\pi/3}\) and \(\Psi(r)\sim\ii\sqrt{r}\) for \(r\gg1\). We also introduce the notations \(a\wedge b\defeq \min\{ a, b\}\)
and  \(a\vee b\defeq \max\{ a, b\}\) for real numbers $a,b$.
\begin{theorem}[Asymptotic \(1\)-point function in the complex case]\label{thm cplx}
Uniformly in \(\widetilde\delta\ge - C\) and \( 0\le  \lambda\le C\) for some fixed large constant \(C>0\) we have   
\begin{subequations}
\begin{equation}
\label{eq:onepointest}
\begin{split}
\E\Tr(Y-\lambda\cdot c(N,\widetilde\delta)-\ii0)^{-1}&= \frac{1}{2\pi\ii} \frac{N^{3/2}}{\widetilde{z}_\ast}\int\diff x\oint\diff y\;  e^{h(y)-h(x)} \widetilde{H}(x,y) \\
&\qquad+ \landauO{ N(1\vee \widetilde{\delta}) \bigl(1+\abs{\log \lambda}\bigr)},
\end{split}
\end{equation}
where 
\begin{equation}
\label{eq:newfunone}
\begin{split}
\widetilde{H}(x,y)&\defeq{}\frac{1}{x^3}+\frac{1}{x^2 y}+\frac{1}{xy^2}+\frac{\widetilde{\delta}\widetilde{z}_*}{xy}+\frac{\widetilde{\delta}\widetilde{z}_*}{x^2},\quad h(x)\defeq -(1\wedge \widetilde\delta^{-1}) \lambda  \widetilde{z}_* x+ \frac{\widetilde{\delta}}{x \widetilde{z}_\ast}+\frac{1}{2x^2\widetilde{z}_*^2},\\ 
\widetilde{z}_\ast &\defeq \lambda^{-1/3} \bigl(1\vee\widetilde\delta^{1/3}\bigr) \abs{\Psi\bigl(\widetilde\delta \lambda^{-1/3} (1\vee\widetilde{\delta}^{1/3})\bigr)}, \qquad
c(N,\widetilde \delta) = N^{-3/2}\cdot (1\wedge \widetilde\delta^{-1}) ,
\end{split}
\end{equation}
and where the $x$-integration is over any contour from $0$ to $e^{3\ii \pi/4}\infty$, going out from $0$ in the direction of the positive real axis, and the $y$-integration is over any contour around $0$ in a counter-clockwise direction. Moreover, in the regime $\lambda\ll 1$ we have the bound
\begin{equation}
\label{eq:asymp}
\abs{\frac{1\wedge \widetilde{\delta}^{-1}}{\widetilde{z}_\ast}\int\diff x\oint\diff y\;  e^{h(y)-h(x)} \widetilde{H}(x,y)} \lesssim
\begin{cases}
\abs{\log \lambda} ,&   \lambda \ge \widetilde{\delta}^3, \\
\abs[0]{\log \lambda \widetilde{\delta}} ,&  \lambda < \widetilde{\delta}^3.
\end{cases}
\end{equation}
\end{subequations}
\end{theorem}

We stated \eqref{eq:onepointest}--\eqref{eq:asymp} only for $\lambda>0$, but in a very similar way we can prove that
\[
\E\Tr(Y+\lambda\cdot c(N,\widetilde\delta))^{-1}\lesssim N^{3/2}(1\vee |\widetilde{\delta}|)\cdot \begin{cases}
\abs{\log \lambda} ,&   \lambda \ge \widetilde{\delta}^3, \\
\abs[0]{\log \lambda \widetilde{\delta}} ,&  \lambda < \widetilde{\delta}^3,
\end{cases}
\]
uniformly in $0\le \lambda\le C$ and \(\widetilde\delta\ge - C\), for some fixed large constant $C>0$.

In the regime above the eigenvalue scaling, i.e.~for \(\lambda\gg 1\), the analogue of Theorem~\ref{thm cplx}  reduces essentially to the local law asymptotics~\eqref{local law}, albeit with a better error term due to the presence of the expectation. 
\begin{proposition}\label{prop local law} Let $Y^z=(X-z)(X-z)^*$ where $X$ is a complex Ginibre ensemble. Then, uniformly in \(\delta\defeq 1-\abs{z}^2\) and \(E\ge 0\), we have the asymptotic expansion in \(E_\pm
\defeq E-\ed_\pm\),
\begin{equation}
    \label{eq:fineqsaddle}\begin{split}
    &\E \Tr (Y^z-E -\ii 0)^{-1} = N m^z(E+\ii 0)\\
    &\qquad\times\Biggl(1+\landauO{\frac{1}{N\abs{E_+}^{3/2}}+\frac{1}{N E^{2/3}}\wedge\Bigl( \frac{\1_{\delta\ge 0}}{N E^{1/2}\delta^{1/2} }+ \frac{\1_{\delta<0}}{N\abs{E_-}^{3/2} \abs{\delta}^{5/2}} \Bigr)}\Biggr).\end{split}
    \end{equation}
where the edges \(\ed_\pm\) of \(\Im m^z\) are explicit functions of \(\delta\) given in~\eqref{ed pm}.
\end{proposition}

\subsection*{Real case}
In the real case our main result is the following optimal bound on \(\E \Tr (Y+E)^{-1}\) for \(E>0\).  Recall the notation $\delta\defeq 1-\abs{z}^2$.
\begin{theorem}[Optimal bound on the resolvent trace in the real case]\label{thm real}
Let \(\rho>0\) be any small constant. Then uniformly in \(E \ge 0\) and \(\delta\ge -CN^{-1/2}\) for some fixed large constant \(C>0\) we have that 
\begin{equation}
    \abs{\E \Tr (Y+E)^{-1} } \lesssim \frac{e^{-\frac{1}{2}N(\Im z)^2}[N^{3/4}\vee N\sqrt{\abs{\delta}}]}{\sqrt{E}}+(N^{3/2}\vee N^2\abs{\delta})\bigl[1+\abs[0]{\log(NE^{2/3})}\bigr]
    \end{equation}
\end{theorem}

Finally, we present our bound on the tail asymptotics for both real and complex cases;
for most applications, this  can be viewed as the main result of this paper. Since a sizeable contribution to \(\Im \Tr(Y-E+\ii 0)^{-1}\) and \(\Im \Tr(Y+E)^{-1}\) comes from the smallest eigenvalue \(\lambda_1(Y^z)\), 
by a straightforward Markov inequality we  immediately obtain the following corollary on the tail asymptotics of \(\lambda_1(Y^z)\) as an easy consequence of Theorems~\ref{thm cplx} and~\ref{thm real}. 
\begin{corollary}[Tail asymptotics of \(\lambda_1(Y^z)\)]
\label{cor:tailas}
For any $C>0$, uniformly in $x\in (0,C]$ and $1-|z|^2>-CN^{-1/2}$ we have the bound
\begin{equation}
\label{eq:complllamb}
\mathbf{P}\Big(\lambda_1(Y^z)\le c(N,z) x\Big)\lesssim \bigl(1+\abs[0]{\log x}\bigr) x
\end{equation}
in the complex case, and
\begin{equation}
\label{eq:reallamb}
\mathbf{P}\Big(\lambda_1(Y^z)\le c(N,z) x\Big) \lesssim e^{-\frac{1}{2}N(\Im z)^2} \sqrt{x}+ \bigl(1+\abs[0]{\log x}\bigr)x
\end{equation}
in the real case, where we recall the definition of the scaling factor $c(N,z)$ from \eqref{cN def}.
\end{corollary}

\subsection*{Properties of the asymptotic Stieltjes transform \texorpdfstring{\(m^z\)}{mz}}
We now record some information on the deterministic Stieltjes transform \(m^z\) which will be useful later. The endpoints of the support of the density of states \(\pi^{-1}\Im m^z\) are the zeros of the discriminant  of the cubic equation~\eqref{MDE Y}
since passing through these points with the real parameter \( E=\Re w\) creates solutions with nonzero imaginary part.
Elementary calculations show that the support of $\Im m_{Y^z}$ is
\([0,\ed_+]\) if \(0\le \delta\le 1\) and it is \([\ed_-,\ed_+]\) if \(\delta<0\), where 
\begin{subequations}
    \begin{equation}\label{ed pm}
        \ed_\pm \defeq  \frac{8\delta^2 \pm(9-8 \delta)^{3/2}-36\delta+27}{8(1-\delta) },
    \end{equation}
and \(\ed_-\) is only considered if \(\delta<0\). Note that while \(\ed_+\sim 1\), the edge \(\ed_-\) may be close to \(0\); more precisely \(0<\ed_-=-4\delta^3/27\bigl(1+\landauO{\abs{\delta}}\bigr)\). The slope coefficient of the square-root density at the edge in \(\ed_\pm\) is given by
    \begin{equation}\label{m asymp}
        \Im m(\ed_\pm\mp \lambda)= \begin{cases}
        \gamma_\pm \sqrt{\lambda}\Bigl(1+\landauO{\sqrt{\lambda}}\Bigr),& \lambda\ge 0,\\  0& \lambda\le 0,\end{cases}, \quad 
\gamma_\pm \defeq \frac{2 \sqrt{2} \left(\sqrt{9-8 \delta }\pm 1\right)^{3/2}}{\left(\sqrt{9-8 \delta }\pm 3\right)^{5/2} \sqrt[4]{9-8 \delta }}.\end{equation}\end{subequations}
Note that while the square-root edge at \(\ed_+\) is non-singular in the sense \(\gamma_+\sim 1\), the square-root edge in \(\ed_-\) becomes singular for small \(\abs{\delta}\) as 
\[\gamma_-=\frac{9}{4\abs{\delta}^{5/2}}\Bigl(1+\landauO{\abs{\delta}}\Bigr).\]

\section{Supersymmetric method} 
Let \(\chi_1,\overline{\chi_1},\dots,\chi_N,\overline{\chi_N}\) denote Grassmannian variables satisfying the commutation rules 
\[ \chi_i \chi_j = -\chi_j\chi_i, \quad \chi_i\overline{\chi_j}=-\overline{\chi_j}\chi_i, \quad \overline{\chi_i}\overline{\chi_j}=-\overline{\chi_j}\overline{\chi_i},\] 
from which it follows that \(\chi_i^2=\overline{\chi_i}^2=0\). As a convention we set \(\overline{\overline{\chi_i}}\defeq-\chi_i\).   The power series of any function of Grassmannian variables is multilinear and it suffices to define the integral in the sense of Berezin~\cite{MR0914369} over Grassmannian variables as the derivatives
\[ \partial_{\chi_k} \chi_k= \partial_{\overline{\chi_k}} \overline{\chi_k}=1, \quad \partial_{\chi_k}1=\partial_{\overline{\chi_k}}1=0,\quad \partial_\chi\defeq \partial_{\chi_1}\partial_{\overline{\chi_1}}\dots\partial_{\chi_N}\partial_{\overline{\chi_N}}\]
and extend them multilinearly to all finite  combinations of monomials in Grassmannians. 
We denote the column vectors with entries \(\chi_1,\dots,\chi_N\) and \(\overline{\chi_1},\dots,\overline{\chi_N}\) by \(\chi\) and \(\overline{\chi}\), respectively. The conjugate transposes of those vectors, i.e.~the row vectors with entries \(\overline{\chi_1},\dots,\overline{\chi_N}\) and \(-\chi_1,\dots,-\chi_N\) will be denoted by \(\chi^\ast\) and \(\overline{\chi}^\ast\), respectively. 
Note that \((\chi^*)^*= -\chi\), \([\overline{\chi}^*]^*=-\overline{\chi}\).
 We now define the inner product of Grassmannian vectors \(\chi,\phi\) by 
\[\braket{\chi,\phi}\defeq \sum_i\overline{\chi_i}\phi_i,\]
so that the quadratic form \(\sum_{i,j} \overline{\chi_i} A_{ij}\chi_j\) can be written as 
\[ \braket{\chi,A\chi}=\sum_{i,j} \overline{\chi_i} A_{ij}\chi_j,\]
where  the matrix-vector product is understood in its usual sense.  
Similarly,  \(s\) and  \(\overline{s}\) denote the column vectors with complex entries \(s_1, \ldots, s_N\) and
their complex conjugates \(\overline{s_1}, \ldots,  \overline{s_N} \), respectively, and for the conjugate transpose
we have \( (s^*)^*=s \) as usual. We have
\[\braket{s,\phi}\defeq \sum_i\overline{s_i}\phi_i,
\qquad \braket{\chi,s}\defeq \sum_i\overline{\chi_i}s_i,\]
 and similarly for quadratic forms. 
The commutation rules naturally also apply to linear functions of the Grassmannians, and therefore also, for example, \(\braket{s,\chi}^2=\braket{\chi,s}^2=0\) for any vector \(s\) of complex numbers. The complex numbers \(s_i\) and often called
\emph{bosonic} variables, while Grassmannians are called \emph{fermions}, motivated by the basic (anti)commutativity 
of the bosonic/fermionic field operators in physics.

\subsection{Determinant identities}
The backbone of the supersymmetric method are the determinant identities
\[ 
\begin{split}
    \frac{1}{\ii^N}\frac{\sgn(\Im w)^N}{\det(H-w)}&=\int_{\C^N} \exp\bigl(-\ii\sgn(\Im w)\braket{s,(H-w)s}\bigr)\diff s,\quad \diff s\defeq\frac{\diff\Re{s_1}\diff\Im{s_1}}{\pi}\cdots\frac{\diff\Re{s_N}\diff \Im{s_N}}{\pi}\\
    \ii^N \det(H-w) &= \partial_\chi\exp(\ii\braket{\chi,(H-w)\chi} ),\quad\partial_\chi\defeq \partial_{\chi_1}\partial_{\overline{\chi_1}}\dots\partial_{\chi_N}\partial_{\overline{\chi_N}},
\end{split} 
\]
where the exponential is defined by its (terminating) Taylor series. Consequently we can conveniently express the generating function as
\[ Z(w,w_1) \defeq \E \frac{\det(H-w_1)}{\det(H-w)}=\E \int_{\C^N}\diff s\,\partial_\chi \exp\Bigl(\ii\braket{\chi,(H-w_1)\chi}-\ii\braket{s,(H-w)s}\Bigr),\]
for \(w\in\HC\) and \(w_1\in\C\), where choice of \(w\) with \(\Im w>0\) guarantees the convergence of
the integral. By taking the \(w_1\) derivative and setting \(w=w_1\) it follows that 
\begin{equation} \begin{split}
\Tr(H-w)^{-1}=&{}-\frac{\partial}{\partial w_1}\frac{\det(H-w_1)}{\det(H-w)}\bigg\rvert_{w_1=w} =\ii\int \braket{\chi,\chi} e^{-\ii \Tr (H-w) [s s^\ast+\chi \chi^\ast] }, \\ \int\defeq&{}\int_{\C^N}\diff s\, \partial_\chi.
\end{split} \label{Tr susy} \end{equation}

\subsection{Superbosonization identity}
After taking expectations, i.e.~performing the Gaussian integration for the entries of \(Y=Y^z=(X-z)(X-z)^\ast\), the resolvent identity~\eqref{Tr susy} will depend on the complex vector \(s\) and the Grassmannian vector \(\chi\) only via certain inner products. More specifically, after defining the \(N\times 2\) and \(N\times 4\) matrices \(\Phi\defeq (s,\chi)\) and \(\Psi\defeq (s,\overline{s},\chi,\overline{\chi})\), the expectation
of the  resolvent can be expressed as an integral over the \(2\times2\) or \(4\times 4\) supermatrices \(\Phi^\ast\Phi\) or \(\Psi^\ast\Psi\) 
 in the complex and real case, respectively. \emph{Supermatrices} are \(2\times 2\) block matrices whose diagonal blocks are commonly referred to as the boson-boson and the fermion-fermion block, while the off-diagonal blocks are the boson-fermion and fermion-boson block. For supermatrices \(Q\) the \emph{supertrace} and \emph{superdeterminant},  the natural generalizations of trace and determinant, are given by 
\begin{equation}\label{def:supertrace}
\sTr\begin{pmatrix}x&\sigma\\\tau&y\end{pmatrix} \defeq \Tr (x)-\Tr (y), \quad \sdet\begin{pmatrix}x&\sigma\\\tau&y\end{pmatrix}\defeq\frac{\det(x)}{\det(y-\tau x^{-1}\sigma)},
\end{equation} 
and the inverse of a supermatrix is   
\begin{equation}\label{def:superinverse}
\begin{pmatrix}x&\sigma\\\tau&y\end{pmatrix}^{-1} = \begin{pmatrix}  (x - \tau y^{-1} \sigma)^{-1} & - x^{-1} \sigma (y - \sigma x^{-1} \tau) \\
  -  y^{-1} \tau (x - \tau y^{-1} \sigma)^{-1} 
&  (y - \sigma x^{-1} \tau)^{-1}\end{pmatrix}.
\end{equation} 
The integral over the remaining degrees of freedom in \(\Phi,\Psi\) other than the inner products in \(\Phi^\ast\Phi,\Psi^\ast\Psi\) can conveniently be performed using the well known \emph{superbosonization formula} which we now recall. It basically identifies the integration volume of the irrelevant degrees of freedom with  the high power of
the superdeterminant of the supermatrix containing the relevant inner products (collected in a \(2\times 2\)
supermatrix \(Q\) in the complex case and a \(4\times 4\) supermatrix \(Q\) in the real case).

\subsubsection{Complex superbosonization}
For any analytic function \(F\) 
 with sufficiently fast decay at \(+\infty\) in the boson-boson sector (in the variable $x$) 
  the complex superbosonization identity from~\cite[Eq.~(1.10)]{MR2430637} implies
\begin{equation}\label{eq cplx superbosonization}
\int F(\Phi^\ast\Phi) = \int_{Q} \sdet^N(Q) F(Q),\quad \int_Q\defeq\frac{1}{2\pi\ii }\int\diff x \oint\diff y\, \partial_\sigma\partial_\tau,\quad Q\defeq\begin{pmatrix}x&\sigma\\\tau&y\end{pmatrix},
\end{equation}
where \(\int\diff x\) denotes the Lebesgue integral on \([0,\infty)\), \(\oint\diff y\) denotes the counterclockwise complex line integral on \(\Set{z\in\C|\abs{z}=1}\) and \(\sigma,\tau\) denote independent scalar Grassmannian variables. 
The key point is that while 
the integral on the left hand side is performed over \(N\) complex numbers and \(2N\) Grassmannians,
the integral on the right is simply over a \(2\times 2\) supermatrix, i.e.~two complex variables and two Grassmannians. 
Note that the identity in~\cite{MR2430637} is more general than~\eqref{eq cplx superbosonization} in the sense that it allows for bosonic and fermionic sectors of unequal sizes. For the case of equal sizes, which concerns us, the formula gets simplified, the measure \(DQ\)  in~\cite[Eq.~(1.8)]{MR2430637} 
becomes the flat Lebesgue measure since  
two determinants  cancel each other as 
\begin{equation}\label{eq simplification dets} \begin{split}
    \det(1-x^{-1}\sigma y^{-1}\tau)\det(1- y^{-1}\tau x^{-1}\sigma)&= e^{\Tr(\log(1-x^{-1}\sigma y^{-1}\tau)+\log(1- y^{-1}\tau x^{-1}\sigma))}\\
    &= e^{-\sum_{k\ge 1} \frac{1}{k} \Tr\bigl(( x^{-1}\sigma y^{-1}\tau)^k+(  y^{-1}\tau x^{-1}\sigma)^k\bigr)}=1,
\end{split} \end{equation}
where the sum is finite and the last equality followed using the commutation rules.

\subsubsection{Real superbosonization}
In the real case we similarly have the real superbosonization identity from~\cite[Eq.~(1.13)]{MR2430637}
\begin{equation} \label{eq real superbosonization}\begin{split}
\int F(\Psi^\ast \Psi) =&{} \int_{Q} \sdet^{N/2}(Q) F(Q), \\ 
\int_Q\defeq&{}\frac{1}{(2\pi)^2\ii} \int \diff x\oint\diff y\, \partial_\sigma \Bigl(\frac{\det(y)}{\det(x)}\Bigr)^{1/2}\det\Bigl(1-\frac{x^{-1}}{y}\sigma \tau\Bigr)^{-1/2} \end{split} \end{equation}
The supermatrix \(Q\) has \(2\times 2\) blocks: \(x\) is non-negative Hermitian, \(y\) is a scalar multiple of the identity matrix. The off-diagonal blocks \(\sigma,\tau\) are related by 
\[\tau \defeq - t_a \sigma^t t_s^{-1}, \quad t_s\defeq\begin{pmatrix}0&1\\1&0\end{pmatrix}, \quad t_a\defeq \begin{pmatrix}0&-1\\1&0\end{pmatrix}.\]
Here the \(\int\diff x\) integral is the Lebesgue measure on non-negative Hermitian \(2\times 2\) matrices \(x\) satisfying \(x_{11}=x_{22}\), i.e.
\[\int\diff x\defeq \int_{0}^\infty\diff_{x_{11}}\int_{\abs{x_{12}}\le x_{11}}\diff \Re x_{12} \diff \Im x_{12},\]
and the fermionic integral is defined as \(\partial_\sigma\defeq \partial_{\sigma_{11}}\partial_{\sigma_{22}}\partial_{\sigma_{21}}\partial_{\sigma_{12}}\). Furthermore, under the slight abuse of notation of identifying the \(2\times 2\) matrix \(y\), which is a scalar multiple of the identity matrix, with the corresponding scalar, \(\oint \diff y\) is the complex line integral over \(\abs{y}=1\) in a counter-clockwise direction. Unlike in the complex case, the matrix elements of the \(4\times 4\) supermatrix \(Q\) are not
independent; there are only 4 (real) bosonic and 4 fermionic degrees of freedom. These identities among
the elements of \(Q\) stem from natural relations between the scalar product of the column vectors of \(\Psi\).
For example the identity \(\braket{ s, s} = \braket{ \overline{s}, \overline{s}}\) 
from the first two diagonal elements of \((\Psi^*\Psi)\) corresponds to \(x_{11}=x_{22}\),
while \(\braket{ \overline{\chi},  \overline{\chi}} =0\) is responsible for \(y_{12}=0\). The relation
\[
   \tau = \begin{pmatrix} \sigma_{22} & \sigma_{12} \cr-\sigma_{21} & -\sigma_{11}\end{pmatrix}
\] 
encoded in the last line of~\eqref{eq real superbosonization} corresponds to relations between scalar products
of bosonic and fermionic vectors and their complex conjugates; for example \(\tau_{21} =-\sigma_{21}\) comes from 
 the identity
\( (\Psi^*\Psi)_{41}
= \braket{ -\bar \chi, s} = - \braket{\bar s, \chi } = - (\Psi^*\Psi)_{23}\), etc. 

\subsection{Application to \texorpdfstring{\(Y^z\)}{Yz} in the complex case }
Our goal is to evaluate \(\E\Tr(H-w)^{-1}\) asymptotically on the scale where \(E\) is comparable with the eigenvalue spacing. We now use the identity 
\begin{equation}\label{eq:yz} \begin{split}
    \Tr (Y-w)^{-1}&=\ii\int\braket{\chi,\chi} e^{-\ii\Tr[(X-z)(X^\ast-\overline{z})-w](ss^\ast+\chi\chi^\ast)}\\
    &=\ii\int\braket{\chi,\chi} e^{\ii w \braket{s,s}-\ii  w \braket{\chi,\chi} -\ii \Tr(X-z)(X^\ast-\overline{z})\Phi\Phi^\ast}
\end{split}  \end{equation}
for \(w=E+\ii\epsilon\) with \(|E|\gg\epsilon>0\). We now compute the ensemble average as (the integral in the second line below is with respect to the flat Lebesgue measure on matrices \(X\))
\begin{equation}\label{susyused} \begin{split}
    &\E e^{-\ii\Tr(X-z)(X^\ast-\overline{z})\Phi\Phi^\ast}\\
    &=\Bigl(\frac{N}{\pi}\Bigr)^{N^2}\int\exp\Bigl[-N\Tr X^\ast\Bigl(1+\ii \frac{\Phi\Phi^\ast}{N}\Bigr) X+\ii \overline{z}\Tr \Phi\Phi^\ast X+\ii z \Tr X^\ast \Phi\Phi^\ast-\ii\abs{z}^2 \Tr \Phi\Phi^\ast \Bigr]\\
    &=\sdet\Bigl(1+\ii\frac{\Phi^\ast\Phi}{N}\Bigr)^{-N}\exp\Bigl(-\ii\abs{z}^2\Bigl[\Tr \Phi\Phi^\ast-\frac{\ii}{N} \Tr \Phi\Phi^\ast\Bigl(1+\ii \frac{\Phi\Phi^\ast}{N}\Bigr)^{-1}\Phi\Phi^\ast\Bigr]\Bigr),\\
    &=\sdet\Bigl(1+\ii\frac{\Phi^\ast\Phi}{N}\Bigr)^{-N}\exp\Bigl(-\ii\abs{z}^2\Tr \Phi\Bigl(1+\frac{\ii}{N}\Phi^\ast\Phi\Bigr)^{-1}\Phi^\ast\Bigr),\\
    &=\sdet\Bigl(1+\ii\frac{\Phi^\ast\Phi}{N}\Bigr)^{-N}\exp\Bigl(-N\abs{z}^2\sTr \Bigl(1+\frac{\ii}{N}\Phi^\ast\Phi\Bigr)^{-1}\frac{\ii}{N}\Phi^\ast\Phi\Bigr).
 \end{split} \end{equation}
 To perform the $\int=\int_{\C^N}\diff s\, \partial_\chi$ integration in~\eqref{eq:yz}  we use 
the superbosonization formula~\eqref{eq cplx superbosonization}
for the function
\begin{equation}\label{def:F}
 F(\Phi^\ast \Phi)\defeq \braket{\chi,\chi}
 \sdet\Bigl(1+\ii\frac{\Phi^\ast\Phi}{N}\Bigr)^{-N}\exp\Bigl(-N\abs{z}^2\sTr \Bigl(1+\frac{\ii}{N}\Phi^\ast\Phi\Bigr)^{-1}\frac{\ii}{N}\Phi^\ast\Phi
 +\ii w \sTr\Phi^\ast\Phi \Bigr),
\end{equation}
We view $F$ as a function of the four independent variables collected in the entries of the $2\times 2$ matrix $\Phi^\ast\Phi$.
 Strictly speaking 
the function $F$ is only meromorphic but not entire in these four variables, but since the
integration regimes on both sides of the superbosonisation formula are well separated away from the poles of $F$,
a simple approximation argument outlined in Appendix~\ref{app:reg} justifies its usage. 
Together with the change of variables \(\frac{\ii}{N}Q\mapsto Q\) it now implies that
\[ \begin{split}
    \E \Tr(Y-w)^{-1} = N \int_{Q'} y e^{Nw\sTr(Q)+N\log\sdet(Q)-N\log\sdet(1+Q)-N \abs{z}^2\sTr(1+Q)^{-1}Q }
\end{split} \]
where \(\int_{Q'}\) indicates the changed integration regime due to the change of variables, more specifically under, \(Q'\) the \(x\)-integration is over \([0,\ii\infty)\) and the \(y\)-integration is over a small circle \(\Set{u\in\C|\abs{u}=N^{-1}}\). Note that the change of variables through scaling does not contribute an additional factor, since superdeterminants are scale invariant. It remains to perform the Berezinian integral. To do so we split
\[Q=q+\mu,\quad q=\begin{pmatrix}x&0\\0&y\end{pmatrix},\quad \mu=\begin{pmatrix}0&\sigma\\\tau&0\end{pmatrix}\]
and compute 
\[ \begin{split} \exp\Bigl(-N\log\sdet(1+Q)\Bigr)&=\exp\Bigl(-N\sTr\log(1+q+\mu)\Bigr)\\
&=\exp\Bigl(-N\sTr\log(1+q)+\frac{N}{2}\sTr(1+q)^{-1}\mu(1+q)^{-1}\mu\Bigr) \\
&=\exp\Bigl(-N\log(1+x)+N\log(1+y)\Bigr)\Bigl(1+N \frac{\sigma\tau}{(1+x)(1+y)}\Bigr)\\
\exp\Bigl(N\log\sdet(Q)\Bigr) &= \exp\Bigl(N\log(x)-N\log(y)\Bigr)\Bigl(1-N \frac{\sigma\tau}{xy}\Bigr)
\end{split}\]
and 
\[ \begin{split}
    &\exp\Bigl(-N\abs{z}^2\sTr(1+Q)^{-1}Q\Bigr) \\
    &\quad= \exp\Bigl(-N\abs{z}^2 \sTr\Bigl[(1+q)^{-1}q-(1+q)^{-1}\mu(1+q)^{-1}\mu(1+q)^{-1}\Bigr]\Bigr)\\
    &\quad= \exp\Bigl(-N\abs{z}^2 \frac{x}{1+x}+N\abs{z}^2\frac{y}{1+y} \Bigr) \Bigl(1+N\abs{z}^2\frac{\sigma\tau}{(1+x)(1+y)}\Bigl(\frac{1}{1+x}+\frac{1}{1+y}\Bigr)\Bigr).
\end{split} \]
By combining these identities we arrive at the final result\footnote{Essentially the same formula, obtained by direct computations, was presented by M. Shcherbina in her seminar
talk on Jan 11, 2016 in Bonn. Our derivation of the same formula via superbosonization is merely a pedagogical preparation to the much more involved real case in Section~\ref{sec:realsuperbos}.}
\begin{equation}\label{H-z cplx superbos}
    \begin{split}
        \E\Tr(Y-w)^{-1} ={}& \frac{N^2}{2\pi\ii} \int_0^\infty\diff x\oint\diff y e^{-Nf(x)+Nf(y)} y \cdot G(x,y),\\
        G(x,y)\defeq{}& \frac{1}{xy}-\frac{1}{(1+x)(1+y)}\Bigl[1+\frac{\abs{z}^2}{1+x}+\frac{\abs{z}^2}{1+y}\Bigr] , \\
        f(x)\defeq{}& \log\frac{1+x}{x}-\frac{\abs{z}^2}{1+x}-wx, 
    \end{split}
\end{equation}
where the \(x\)-integration is over \((0,\ii\infty)\) and the \(y\)-integration is over a circle of radius \(N^{-1}\) around the origin.

\subsection{Application to \texorpdfstring{\(Y^z\)}{Yz} in the real case }
\label{sec:realsuperbos}
We now consider the real case and introduce the \(N\times 4\) matrix \(\Psi\defeq (s,\overline{s},\chi,\overline{\chi})\), the \(2\times 2\) matrix \(Z\defeq \bigl(\begin{smallmatrix}z&0\\0&\overline{z}\end{smallmatrix}\bigr)\) and the \(4\times 4\) matrix \(Z_2\defeq\bigl(\begin{smallmatrix}Z&0\\0&Z\end{smallmatrix}\bigr)\), and use that 
    \[ \ii \overline{z}\Tr\Phi\Phi^\ast X+\ii z\Tr X^t \Phi\Phi^\ast=\frac{\ii}{2} \Tr \Psi Z_2^\ast \Psi^\ast X+\frac{\ii}{2} \Tr X^t\Psi Z_2\Psi^\ast, \quad \Bigl(\Psi Z_2\Psi^\ast\Bigr)^t=\Psi Z_2^\ast\Psi^\ast  \]
    to compute
    \begin{equation} \begin{split}
    &\E e^{-\ii \Tr(X-z)(X^t-\overline{z})\Phi\Phi^\ast} \\
    &= \E\Bigl(\frac{N}{2\pi}\Bigr)^{N^2/2}\int\exp\Bigl(-\frac{N}{2}\Tr X^t\Bigl(1+\frac{2\ii}{N}\Phi\Phi^\ast\Bigr) X+\ii \overline{z}\Tr \Phi\Phi^\ast X+\ii z \Tr X^t \Phi\Phi^\ast-\ii \abs{z}^2\Tr \Phi\Phi^\ast\Bigr)\\
    &= \E\Bigl(\frac{N}{2\pi}\Bigr)^{N^2/2}\int\exp\Bigl(-\frac{N}{2}\Tr X^t\Bigl(1+\frac{\ii}{N}\Psi\Psi^\ast\Bigr) X+\frac{\ii}{2}\Tr \Psi Z_2^\ast\Psi^\ast X\\
    &\qquad\qquad\qquad\qquad\qquad\qquad\qquad\qquad\qquad+\frac{\ii}{2} \Tr X^t \Psi Z_2\Psi^\ast-\frac{\ii \abs{z}^2}{2}\Tr \Psi\Psi^\ast\Bigr)\\
    &= \sdet\Bigl(1+\frac{\ii}{N}\Psi^\ast\Psi\Bigr)^{-N/2}\exp\Bigl(-\frac{\ii}{2}\Tr\Psi Z_2^\ast\Bigl(1-\frac{\ii}{N} \Psi^\ast\Bigl(1+\frac{\ii}{N}\Psi\Psi^\ast\Bigr)^{-1}\Psi \Bigr)Z_2\Psi^\ast\Bigr)\\
    &= \sdet\Bigl(1+\frac{\ii}{N}\Psi^\ast\Psi\Bigr)^{-N/2}\exp\Bigl(-\frac{\ii}{2}\Tr\Psi Z_2^\ast\Bigl(1+\frac{\ii}{N}\Psi^\ast\Psi\Bigr)^{-1}Z_2\Psi^\ast\Bigr)\\
    &= \sdet\Bigl(1+\frac{\ii}{N}\Psi^\ast\Psi\Bigr)^{-N/2}\exp\Bigl(-\frac{N}{2}\sTr\Bigl(1+\frac{\ii}{N}\Psi^\ast\Psi\Bigr)^{-1}Z_2 \frac{\ii}{N}\Psi^\ast\Psi Z_2^\ast\Bigr),
    \end{split}\end{equation}
    where we used that \(X\) is real and \(\Psi\Psi^\ast\) is symmetric. The superbosonization formula thus implies 
    \[\begin{split}
        \E \Tr (H-w)^{-1}&= \frac{N}{2(2\pi)^2\ii } \int \frac{\Tr(y)\det(y)^{1/2}}{\det(x)^{1/2}}  \exp\Bigl(-\frac{1}{2}\log\det(1-x^{-1}\sigma y^{-1}\tau)\Bigr)\\
        \times \exp\Bigl(\frac{N}{2}&\Bigl[\sTr(wQ)-\log\sdet(1+Q)+\log\sdet(Q)-\sTr(1+Q)^{-1}Z_2 Q Z_2^\ast\Bigr]\Bigr)\\
        &= \frac{N}{2(2\pi)^2} \int \frac{\Tr(y)\det(y)^{1/2}}{\det(x)^{1/2}}  \exp\Bigl(-\frac{1}{2}\Tr\log(1-x^{-1}\sigma y^{-1}\tau)\Bigr)\\
        \times \exp\Bigl(\frac{N}{2}&\Bigl[\sTr(wQ)-\sTr\log(1+Q)+\sTr\log(Q)-\sTr(1+Q)^{-1}Z_2 Q Z_2^\ast\Bigr]\Bigr)
    \end{split}\]
    where 
    \[Q=\begin{pmatrix}x&\sigma\\\tau&y\end{pmatrix}\equiv\frac{\ii}{N}\Psi^\ast\Psi,\qquad\tau=\begin{pmatrix}0&1\\-1&0\end{pmatrix}\sigma^t\begin{pmatrix}0&1\\1&0\end{pmatrix}.\]
    In order to expand the exponential terms to fourth order in \(\sigma\) we introduce the short-hand notations
    \begin{equation}\label{sigmatau}
    Q=q+\mu,\quad q=\begin{pmatrix}x&0\\0&y\end{pmatrix},\quad \mu=\begin{pmatrix}0&\sigma\\\tau&0\end{pmatrix}.
    \end{equation}
    We compute
    \[\begin{split}
        \sTr(1+Q)^{-1}Z_2 Q Z_2^\ast &= \sTr(1+q)^{-1} Z_2 q Z_2^\ast \\
        &\quad-\sTr (1+q)^{-1}\mu(1+q)^{-1} \Bigl(Z_2\mu Z_2^\ast-\mu (1+q)^{-1} Z_2 q Z_2^\ast\Bigr)\\
        &\quad- \sTr \bigl((1+q)^{-1}\mu\bigr)^3 (1+q)^{-1}\Bigl(Z_2\mu Z_2^\ast-\mu (1+q)^{-1} Z_2 q Z_2^\ast\Bigr),\\
        &=\Tr (1+x)^{-1} Z x Z^\ast-\abs{z}^2\Tr\frac{y}{1+y} - \Tr (\sigma Z \tau Z^\ast A+Z \sigma Z^\ast \tau A ) \\
        &\quad +\Tr\sigma\tau A C'-\Tr\sigma\tau A(\sigma Z \tau Z^\ast A+Z\sigma Z^\ast \tau A) + \Tr \sigma\tau A \sigma \tau A C',
    \end{split}\]
where we introduced matrices \(A,C'\) as in 
\[A \defeq \frac{(1+x)^{-1}}{1+y},\quad B\defeq \frac{x^{-1}}{y},\quad C'\defeq Z x Z^\ast (1+x)^{-1}+\abs{z}^2 \frac{y}{1+y},\]
as well as \(B\), which will be used in the sequel. In deriving these formulas we used 
     that \(q\) and \(Z_2\) have zero off-diagonal blocks and \(\mu\) has zero diagonal blocks,
to eliminate terms with odd powers of \(\mu\) after taking the supertrace, and that \(y\) is a scalar multiple of the identity. Similarly we find for the logarithmic terms  
    \[\begin{split} &\sTr\Bigl(\log(q+\mu)-\log(1+q+\mu)\Bigr)\\
        &= \sTr\log(q(1+q)^{-1})-\frac{1}{2}\sTr (q^{-1}\mu)^2-\frac{1}{4}\sTr (q^{-1}\mu)^4 \\
        &\qquad+\frac{1}{2}\sTr ((1+q)^{-1}\mu)^2+ \frac{1}{4}\sTr ((1+q)^{-1}\mu)^4\\
        &=\log\frac{\det(x)}{\det(1+x)}-\log\frac{\det(y)}{\det(1+y)}+\Tr\sigma\tau(A-B)+\frac{1}{2}\Tr (\sigma\tau A)^2-\frac{1}{2}\Tr (\sigma\tau B)^2 \end{split}\]
    and
    \[ -\frac{1}{2}\Tr\log(1-x^{-1}\sigma y^{-1}\tau)=\frac{1}{2}\Tr\sigma\tau B+  \frac{1}{4}\Tr (\sigma\tau B)^2.\]
    Whence
    \begin{equation} \label{expdersusyrealAr}\begin{split}
        \E \Tr(H-w)^{-1} &= \frac{N}{2(2\pi)^2 \ii} \int\diff x\oint \diff y \frac{\Tr(y)\det(y)^{1/2}}{\det(x)^{1/2}} \exp\Bigl(-\frac{N}{2}f(x)+\frac{N}{2}f(y)\Bigr) G(x,y,z)\\
        f(x)&\defeq -w \Tr x +\log\frac{\det(1+x)}{\det(x)}+\Tr Z x Z^\ast(1+x)^{-1}-2\abs{z}^2,\\
        G(x,y,z)&\defeq \partial_\sigma\exp\Bigl[\frac{N}{2}\Bigl(\Tr\sigma\tau \bigl(A(1-C')-\bigl(1-\frac{1}{N}\bigr)B\bigr)+\Tr(\sigma Z \tau Z^\ast A+ Z\sigma Z^\ast \tau A)  \\
         +\Tr \sigma\tau A&\Bigl(\sigma\tau A\bigl(\frac{1}{2}-C'\bigr)+\sigma Z \tau Z^\ast A+ Z\sigma Z^\ast \tau A\Bigr)-\frac{1}{2}(1- \frac{1}{N})\Tr (\sigma\tau B)^2\Bigr)\Bigr],
    \end{split}\end{equation}
    where \(\int\diff x=\int\diff x_{11}\diff \Re x_{12}\diff \Im x_{12}\) is the integral over matrices of the form  
    \[ x=\begin{pmatrix}
        \ii x_{11}& \ii x_{12}\\ \ii\overline{x_{12}} & \ii x_{11}
    \end{pmatrix} \]
    with \(x_{11}\in[0,\infty)\) and \(x_{12}\in\C\) with \(\abs{x_{12}}\le x_{11}\). The integral \(\oint\diff y=\oint\diff y_{11}\) is the 
    integral over scalar matrices \(y=y_{11} I\) with \(\diff y_{11}\) being the complex line integral over \(\abs{y_{11}}=N^{-1}\) in a counter-clockwise direction.

    To integrate out the Grassmannians we expand the exponential to second order, use the relation~\eqref{sigmatau} between \( \sigma\) and 
    \( \tau\), and use that  for  \( 2\times 2\)
    matrices \(X,Y\) which are constant on the diagonal (\( x_{11}= x_{22}\), \( y_{11}=y_{22}\))  we have the identities 
    \[ \begin{split}\frac{N^2}{8}\partial_\sigma \Tr^2(\sigma Z \tau Z^\ast X+ Z\sigma Z^\ast \tau X)&= - 4N^2\abs{z}^2(\Re z)^2\det(X)+N^2 \abs{z}^2(\Im z)^2 \Tr^2(X)\\
     \frac{N^2}{8} \partial_\sigma\Tr^2(\sigma\tau X)&=-N^2\det(X)\\
     \frac{N^2}{4} \partial_\sigma \Tr(\sigma\tau X)\Tr\Bigl(\sigma Z \tau Z^\ast Y+ Z\sigma Z^\ast \tau Y\Bigr)&= 2N^2(\Re z)^2 \bigl(\Tr(XY)  -\Tr(X)\Tr (Y)\bigr)\\
     &\quad+2\ii N^2 (\Im z) (\Re z) \bigl(X_{12}Y_{21}-X_{21} Y_{12}\bigr).\\
     \frac{N}{2}\partial_\sigma\Tr(\sigma\tau X\sigma \tau Y) &= N (\Tr(X)\Tr(Y)-\Tr(XY))\\
     \frac{N}{2}\partial_\sigma\Tr \sigma\tau X\Bigl(\sigma Z \tau Z^\ast X+ Z\sigma Z^\ast \tau X\Bigr)&=4N(\Re z)^2\det(X).
      \end{split} \]
      Whence we finally have the expression
      \begin{equation}\label{defbigGAr}
        \begin{split}
        G&=-N^2 \Bigl[\det\Bigl(A(1-C')-\bigl(1-\frac{1}{N}\bigr)B\Bigr) + (4 \abs{z}^2 (\Re z)^2 + 2 (\Re z)^2 (2 - \Tr C')) \det A \\
        &\qquad\qquad - 
       \abs{z}^2 (\Im z)^2 \Tr^2 A- 2 (\Re z)^2 \bigl(1 - \frac{1}{n}\bigr) \Bigl(\Tr A\Tr B-\Tr AB\Bigr) \\ 
       &\qquad\qquad - 2 (\Re z)^2(\Im z)^2 \det A^2 \det(1 + y) (4 \det(x) - \Tr^2x)\Bigr]\\
       &\quad + N\Bigl(\det(A)\Bigl(1+4(\Re z)^2-\Tr C'\Bigr)-\bigl(1-\frac{1}{N}\bigr)\det(B)\Bigr).
      \end{split}\end{equation}
We now rewrite~\eqref{expdersusyrealAr} by using the parametrizations
\begin{equation}
      \label{paramAr}
      x=  \begin{pmatrix}
      a & a\sqrt{1-\tau} e^{\ii \varphi} \\
      a\sqrt{1-\tau} e^{-\ii \varphi} & a
      \end{pmatrix} ,
      \qquad
      y= \begin{pmatrix}
      \xi & 0 \\
      0 & \xi
      \end{pmatrix} ,
      \end{equation}
with $a\in \ii \R _+$, $\tau\in [0,1]$, $\varphi\in [0,2\pi]$ and $\abs{\xi}=N^{-1}$. Since the integral over $\varphi\in [0,2\pi]$ is equal to $2\pi$ as a consequence of the fact that the functions $f,g,G_N$ defined below do not depend on $\varphi$, we have that
\begin{equation}
      \label{realsusyexplAAr}
      \E \Tr   [Y-w]^{-1}=\frac{N}{4\pi \ii} \oint \diff \xi \int_0^{\ii \infty} \diff a\int_0^1 \diff \tau \frac{\xi^2a}{\tau^{1/2}} e^{N[f(\xi)-g(a,\tau,\eta)]} G_N(a,\tau,\xi,z),
      \end{equation}
      where, using the notation $\eta\defeq\Im z$, the functions $f$ and $g$ are defined by
      \begin{equation}
        \label{eq:fffr}
        f(\xi)\defeq -w \xi+\log (1+\xi)-\log \xi-\frac{|z|^2}{1+\xi},      
      \end{equation}
      \begin{equation}  
          \begin{split}
            \label{bosonphfAr}
            g(a,\tau,\eta)&\defeq -wa+\frac{1}{2}\log [1+2a+a^2\tau]-\log a-\frac{1}{2}\log \tau\\
            &\quad -\frac{|z|^2(1+a)-2\eta^2 a^2 (1-\tau)}{1+2a+a^2\tau}.    
          \end{split}
      \end{equation}    
Note that $g(a,1,\eta)=f(a)$; in particular, we remark that $g(a,1,\eta)$ is independent of $\eta$ for any $a\in\C $. Furthermore, using the parameterizations in~\eqref{paramAr} the function $G_N\defeq G_{1,N}+G_{2,N}$ is given by
\begin{equation}
\label{eq:newbetG}
\begin{split}
        G_{1,N}={}&\Bigl(
        N^2\frac{p_{2,0,0}}{a^2 \xi ^2 (\xi +1)^2 \tau }-N\frac{ p_{1,0,0}}{a^2 \xi ^2 (\xi +1) \tau  }+\delta  N^2\frac{p_{2,0,1}}{a \xi  (\xi +1)^2 \tau  }- N \delta \frac{p_{1,0,1}}{a \xi  (\xi +1) \tau  } \\
        &\quad + N^2\delta ^2\frac{ p_{2,0,2}}{(\xi +1)^2}\Bigr)\times \Bigl((a^2\tau+2a+1)^2(\xi+1)^2\Bigr)^{-1}, \\
         G_{2,N}={}&\Bigl(
         N^2\eta ^2 \frac{p_{2,2,0}}{a \xi  (\xi +1)^3 \tau }-N\eta ^2 \frac{p_{1,2,0}}{a \xi  \tau  }+  N^2\eta ^2\delta \frac{p_{2,2,1}}{(\xi +1) }\Bigr)\times \Bigl((a^2\tau+2a+1)^2(\xi+1)^2\Bigr)^{-1},
\end{split}
\end{equation}
where \(p_{i,j,k}=p_{i,j,k}(a, \tau, \xi)\) are explicit polynomials in \(a,\tau,\xi\) which we defer to Appendix~\ref{appendix poly}, \(\eta\defeq\Im z\) and \(\delta\defeq1 -\abs{z}^2\). The indices $i,j,k$ in the definition of $p_{i,j,k}$ denote the $N$, $\eta$ and $\delta$ power, respectively. We split $G_N$ as the sum of $G_{1,N}$ and $G_{2,N}$ since $G_{1,N}$ depends only on $|z|$, whilst $G_{2,N}$ depends explicitly by $\eta= \Im [z]$, hence $G_{2,N}=0$ if $z\in \R $. We can thus rewrite \eqref{realsusyexplAAr} as
\begin{equation}
\label{eq:finnewform}
\E \Tr   [Y-w]^{-1}=\frac{N}{4\pi \ii} \oint_\Gamma \diff \xi \int_\Omega \diff \tau \int_{r\Lambda} \diff a \frac{\xi^2a}{\tau^{1/2}} e^{N[f(\xi,w)-g(a,\tau,\eta,w)]} G_N(a,\tau,\xi,\eta).
\end{equation}

\section{Asymptotic analysis in the complex case for the saddle point regime}
\label{sec:saddlepointanalysis}
For the density of states $\varrho_{Y^z}$ on the positive semi-axis \(E>0\) we expect a singular behaviour for \(E\approx 0\) and a 
square-root edge for \(E\approx \ed_+\). The singularity at \(E=0\) exhibits a phase transition in \(\delta\) at 0; for \(\delta>0\) the transition 
is between an \(E^{-1/3}\)--singularity for \(\delta=0\) and a 
\( \delta^{1/2}E^{-1/2}\)--singularity for \(0<\delta\le 1\), while for \(\delta<0\) the transition is between the \(E^{-1/3}\)--singularity for \(\delta=0\) and square-root edge in \(\ed_-\sim\abs{\delta}^3\) of slope \(\abs{\delta}^{-5/2}\). 
We now analyse the location of the critical point(s) \( x_*\), i.e.~the solutions to \(f'(x_*)=0\), as well as the asymptotics of 
the phase function \(f\) around them precisely in all of the above regimes. For the saddle-point approximation the second derivative 
\(f''(x_\ast)\) is of particular importance and we find that it can only vanish in the vicinity of \(E\approx \ed_+\) and \(E\approx \ed_- \vee 0\), and otherwise satisfies \(\abs{f''(x_\ast)}\gtrsim 1\).

The saddle point equation \(f'(x_*)=0\) leads to the simple cubic equation
\[ w x_\ast^3 + 2 w x_\ast^2+w x_\ast+\delta  x_\ast+1=0, \]
which is precisely the MDE equation from~\eqref{MDE Y}, whose explicit solution via Cardano's formula
reveals  that for \(E\in(\ed_-,\ed_+)\) there are two relevant critical points \(x_\ast,\overline{x_\ast}\) with \(\Re f(x_\ast)=\Re f(\overline{x_\ast})\), while for \(E\ge\ed_+\) or \(0\le E\le\ed_-\) there is one relevant critical point \(x_\ast\),  where \(x_\ast\) is given by
\begin{equation} \label{xstar} \begin{split}
    x_\ast ={}& \begin{cases}
        e^{-2\ii\pi/3}\sqrt[3]{q+\sqrt{q^2+p^3}}+e^{2\ii\pi/3}\sqrt[3]{q-\sqrt{q^2+p^3}}-\frac{2}{3}, &E\le \ed_-\\
        e^{2\ii\pi/3}\sqrt[3]{q+\sqrt{q^2+p^3}}+e^{-2\ii\pi/3}\sqrt[3]{q-\sqrt{q^2+p^3}}-\frac{2}{3}, &\ed_-\le E\le \ed_+\\
        \sqrt[3]{q+\sqrt{q^2+p^3}}+\sqrt[3]{q-\sqrt{q^2+p^3}}-\frac{2}{3}, & E\ge \ed_+
    \end{cases}\\
    q\defeq{}& \frac{\delta}{3E}+\frac{1}{27}-\frac{1}{2E},\quad p\defeq \frac{\delta}{3E}-\frac{1}{9},
\end{split} \end{equation}
where \(q^2+p^3> 0\) as long as \(E\in(\ed_-,\ed_+)\) and \(q^2+p^3<0\) for \(E>\ed_+\) or \(E<\ed_-\). Here we chose the branch of the cubic root such that \(\sqrt[3]{\R}=\R\) and that \(\sqrt[3]{z}\) for \(z\in\C\setminus\R\) is the cubic root with the maximal real part. Note that the choice of the cubic root implies \(\Im x_\ast\ge 0\) and \(x_\ast=x_\ast(E)=m^z(E+\ii 0)\), where \(m^z\) has been defined in~\eqref{MDE Y}. 

Before concluding this section with the proof of Proposition~\ref{prop local law}, we collect certain asymptotics of the critical point \(x_\ast\) and the phase function \(f\) in its vicinity, which will also be used in the main estimates of the present paper in Sections~\ref{sec cplx}--\ref{sec real}. In the edges \(\ed_\pm\) the critical points have the simple expressions 
\[x_\ast(\ed_\pm) = -\frac{2}{3\pm \sqrt{9-8 \delta }}\]
and satisfy \(x_\ast(\ed_+)\sim -1 \) and \(x_\ast(\ed_-) = -3/(2\delta) \bigl[1+\landauO{\abs{\delta}}\bigr]\).
Elementary expansions of~\eqref{xstar} for $E$ near the edges
reveal  the following asymptotics of $x_*$ in the various regimes.

\subsubsection*{Regime \(E\approx \ed_+\)}
\begin{figure}
    \centering 
    \begin{subfigure}[b]{.33\linewidth}
        \centering 
        \begin{asy}
            size(4cm,5cm,IgnoreAspect);
            pair xyMin=(-1.51,-1.01);  
            pair xyMax=(1,1);
            var delta = 0.;
            var ee = 5.75; 
            pair xast = (-0.321061, 0.158104 );
            real ff(pair z) {return xpart( log(1+z)-log(z) -(1-delta)/(1+z) - ee*z ) ;}
            var c = 4;
            
            real f(real x, real y) {return ff((x,y));}
            var ffS = ff(xast);
            picture bar;
            bounds range=image(f,Range(-3,3),xyMin,xyMax,N,Palette);
            palette(range,(xyMin.x,xyMax.y+.2),(xyMax.x,xyMax.y+.4),Top,Palette,PaletteTicks(N=2,n=2));
            defaultpen(.5bp);
            draw(contour(ff,xyMin,xyMax,new real[] {ffS},N,operator ..),white);
            draw((0,0)--xast--(-1.5,1));
            draw((-.5,0)..xast..(.5,0)..(xast.x,-xast.y)..(-.5,0),dashed);             
            xaxis(Bottom,xyMin.x,xyMax.x,RightTicks(Step=1,step=0.5),above=true);
            yaxis(Left,xyMin.y,xyMax.y,LeftTicks(Step=1,step=1),above=true);
        \end{asy}  
        \caption{\(E_+<0 \)} 
    \end{subfigure}%
    \begin{subfigure}[b]{.33\linewidth}
        \centering
        \begin{asy}
            size(4cm,5cm,IgnoreAspect);
            pair xyMin=(-1.51,-1.01);  
            pair xyMax=(1,1);
            var delta = 0.;
            var ee = 6.75; 
            pair xast = (-0.3333333, 0. );
            real ff(pair z) {return xpart( log(1+z)-log(z) -(1-delta)/(1+z) - ee*z ) ;}
            var c = 4;
            
            real f(real x, real y) {return ff((x,y));}
            var ffS = ff(xast);
            picture bar;
            bounds range=image(f,Range(-3,3),xyMin,xyMax,N,Palette);
            palette(range,(xyMin.x,xyMax.y+.2),(xyMax.x,xyMax.y+.4),Top,Palette,PaletteTicks(N=2,n=2));
            defaultpen(.5bp);
            draw(contour(ff,xyMin,xyMax,new real[] {ffS},N,operator ..),white);
            draw((0,0)--(-.3,0)--(-1.5,1));
            draw(xast..(.5,.2)..(.5,-.2)..xast,dashed);             
            xaxis(Bottom,xyMin.x,xyMax.x,RightTicks(Step=1,step=0.5),above=true);
        \end{asy}  
        \caption{\(\abs[0]{E_+} = 0 \)} 
    \end{subfigure}%
    \begin{subfigure}[b]{.33\linewidth}
        \centering
        \begin{asy}
            size(4cm,5cm,IgnoreAspect);
            pair xyMin=(-1.51,-1.01);  
            pair xyMax=(1,1);
            var delta = 0.;
            var ee = 7.75; 
            pair xast = (-0.203273, 0 );
            real ff(pair z) {return xpart( log(1+z)-log(z) -(1-delta)/(1+z) - ee*z ) ;}
            var c = 4;
            
            real f(real x, real y) {return ff((x,y));}
            var ffS = ff(xast);
            picture bar;
            bounds range=image(f,Range(-3,3),xyMin,xyMax,N,Palette);
            palette(range,(xyMin.x,xyMax.y+.2),(xyMax.x,xyMax.y+.4),Top,Palette,PaletteTicks(N=2,n=2));
            defaultpen(.5bp);
            draw(contour(ff,xyMin,xyMax,new real[] {ffS},N,operator ..),white);
            draw((0,0)--(-.4,0)--(-1.5,1));
            draw(xast..(.5,.5)..(.5,-.5)..xast,dashed);             
            xaxis(Bottom,xyMin.x,xyMax.x,RightTicks(Step=1,step=0.5),above=true);
            yaxis(Right,xyMin.y,xyMax.y,RightTicks(Step=1,step=1),above=true);
        \end{asy}  
        \caption{\( E_+>0 \)} 
    \end{subfigure}
    \caption{Contour plot of \(\Re f(x)\) in the regime \(E\approx \ed_+\). The solid white lines represent the level set \(\Re f(x)=\Re f(x_\ast)\), while the solid and dashed black lines represent the chosen contours for the \(x\)- and \(y\)-integrations, respectively.}
    \label{fig:cont_ep}
\end{figure}
Close to the spectral edge \(E\approx \ed_+\) we have the asymptotic expansion 
\begin{subequations}
\begin{equation} x_\ast = x_\ast(\ed_+) + \gamma_+ \sqrt{E_+}\Bigl(1+\landauO{\abs[0]{E_+^{1/2}}}\Bigr) \label{xast} \end{equation}
in \(E_+\defeq E-\ed_+\), where \(\gamma_+\) was defined in~\eqref{m asymp}. The location of saddle point(s) in the regime \(E\approx \ed_+\) is depicted in Figure~\ref{fig:cont_ep}. The second derivative of \( f\) is asymptotically given by 
\begin{equation}f''(x_\ast) = \frac{2\sqrt{E_+}}{\gamma_+}\Bigl(1+\landauO{\abs[0]{E_+^{1/2}}}\Bigr).\label{fpp e+}\end{equation}
\end{subequations}

\subsubsection*{Regime \(E\approx 0\) in the case \(\delta\ge 0\)}
\begin{figure}
    \centering
    \begin{subfigure}[b]{.5\linewidth}
        \centering
        \begin{asy}
            size(6cm,5cm,IgnoreAspect);
            pair xyMin=(-3,-6);  
            pair xyMax=(6,6);
            var delta = .0000001;
            var ee = .01; 
            pair xast = (1.66553 , 3.99801 );
            real ff(pair z) {return xpart( log(1+z)-log(z) -(1-delta)/(1+z) - ee*z ) ;}
            var c = 4;
            
            real f(real x, real y) {return ff((x,y));}
            var ffS = ff(xast);
            picture bar;
            bounds range=image(f,Range(-.1,.1),xyMin,xyMax,N,Palette);
            palette(range,(xyMin.x,xyMax.y+1),(xyMax.x,xyMax.y+2),Top,Palette,PaletteTicks(N=2,n=2));
            defaultpen(.5bp);
            draw(contour(ff,xyMin,xyMax,new real[] {ffS},N,operator ..),white);
            draw((0,0)--(c,0)--(c+(xast.x-c)*xyMax.y/xast.y,xyMax.y));
            draw((-.7,0)..(-.5,2)..xast..(5.8,0)..(xast.x,-xast.y)..(-.5,-2)..(-.7,0),dashed);             
            xaxis(Bottom,xyMin.x,xyMax.x,RightTicks(Step=3,step=1.5),above=true);
            yaxis(Left,xyMin.y,xyMax.y,LeftTicks(Step=5,step=2.5),above=true);
        \end{asy}  
        \caption{\(0<\delta\ll E^{1/3}\ll 1\)} 
    \end{subfigure}%
    \begin{subfigure}[b]{.5\linewidth}
        \centering
        \begin{asy}
            size(6cm,5cm,IgnoreAspect);
            pair xyMin=(-3,-6);  
            pair xyMax=(6,6);
            var delta = .5;
            var ee = .04;
            pair xast = (-0.0586703, 3.64358 );
            real ff(pair z) {return xpart( log(1+z)-log(z) -(1-delta)/(1+z) - ee*z ) ;}
            var c = 2;

            real f(real x, real y) {return ff((x,y));}
            var ffS = ff(xast);
            picture bar;
            bounds range=image(f,Range(-.1,.1),xyMin,xyMax,N,Palette);
            palette(range,(xyMin.x,xyMax.y+1),(xyMax.x,xyMax.y+2),Top,Palette,PaletteTicks(N=2,n=2));
            defaultpen(.5bp);
            draw(contour(ff,xyMin,xyMax,new real[] {ffS},N,operator ..),white);
            draw((0,0)--(c,0)--(c+(xast.x-c)*xyMax.y/xast.y,xyMax.y));
            draw((-.7,0)..(-.5,2)..xast..(5,0)..(xast.x,-xast.y)..(-.5,-2)..(-.7,0),dashed);             
            xaxis(Bottom,xyMin.x,xyMax.x,RightTicks(Step=3,step=1.5),above=true);
            yaxis(Right,xyMin.y,xyMax.y,RightTicks(Step=5,step=2.5),above=true);
            \end{asy}  
        \caption{\(0< E^{1/3}\ll \delta\)} 
    \end{subfigure}
    \caption{Contour plot of \(\Re f(x)\) for \(\delta>0\) in the regime \(E\approx 0\). The solid white lines represent the level set \(\Re f(x)=\Re f(x_\ast)\), while the solid and dashed black lines represent the chosen contours for the \(x\)- and \(y\)-integrations, respectively.}\label{fig:cont_min}
    \label{fig:contn}
\end{figure}
For \(E\approx 0\) we have the asymptotic expansions 
\begin{subequations}
\begin{equation} \label{xast >0}
x_\ast = E^{-1/3} \Psi\Bigl( \frac{\delta}{E^{1/3}} \Bigr)\biggl[1+\landauO{\delta+E^{1/3}}\biggr], 
\end{equation}
where \(\Psi(\lambda)\) is the unique solution to the cubic equation 
\[1 + \lambda\Psi(\lambda)+\Psi(\lambda)^3 =0, \quad \Re\Psi(\lambda)> 0,\quad \Im\Psi(\lambda)>0, \quad \lambda\ge 0.\]
The explicit function \(\Psi(\lambda)\) has the asymptotics 
\[\lim_{\lambda\searrow 0}\Psi(\lambda)=\Psi(0)=e^{\ii\pi/3},\quad \lim_{\lambda\to\infty} \frac{\Psi(\lambda)}{\sqrt{\lambda}}= \ii.\]
Thus it follows that
\begin{equation} \label{eq:statpoi} x_\ast = \Bigl(\ii \sqrt{\frac{\delta}{E}}-1+\frac{1}{2\delta}\Bigr)\Bigl(1+\landauO{\frac{E}{\delta^3}}\Bigr)=\Bigl(\frac{e^{\ii\pi/3}}{E^{1/3}}-\frac{2}{3}\Bigr)\Bigl(1+\landauO{E^{2/3}+\frac{\delta}{E^{1/3}}}\Bigr),\end{equation}
where the first expansion is informative in the \(E\ll\delta^3\), and the second one in the \(E \gg \delta^3\) regime. The location of saddle point(s) in the regime \(E\approx 0 \) is depicted in Figure~\ref{fig:cont_min}. For the second derivative we have the expansions 
\begin{equation}f''(x_\ast)=3e^{2\ii\pi/3}E^{4/3}\Bigl(1+\landauO{E+\frac{\delta}{E^{1/3}}}\Bigr)=2\ii\frac{E^{3/2}}{\delta^{1/2}}\Bigl(1+\landauO{\frac{E}{\delta^3}}\Bigr)\label{fpp 0}\end{equation}
and similarly for higher derivatives, \(\abs{f^{(k)}(x_\ast)}\sim E^{(2+k)/3}\wedge E^{(k+1)/2}\delta^{-(k-1)/2}\) for \(k\ge 3\). 
\end{subequations}

\subsubsection*{Regime \(E\approx \ed_-\) in the case \(\delta<0\)}
\begin{figure}
    \centering
    \begin{subfigure}[b]{.33\linewidth}
        \centering
        \begin{asy}
            size(4cm,6cm,IgnoreAspect);
            pair xyMin=(-3,-6);  
            pair xyMax=(6,6);
            var delta = .00000001;
            var ee = .01; 
            pair xast = (1.66553,3.998 );
            real ff(pair z) {return xpart( log(1+z)-log(z) -(1-delta)/(1+z) - ee*z )  ;}
            
            real f(real x, real y) {return ff((x,y));}
            var ffS = ff(xast);
            picture bar;
            bounds range=image(f,Range(-.1,.1),xyMin,xyMax,N,Palette);
            palette(range,(xyMin.x,xyMax.y+1),(xyMax.x,xyMax.y+2),Top,Palette,PaletteTicks(N=2,n=2));
            defaultpen(.5bp);
            draw(contour(ff,xyMin,xyMax,new real[] {ffS},N,operator ..),white);
            draw((0,0)--(sqrt(xast.x^2+xast.y^2),0)--(sqrt(xast.x^2+xast.y^2)*(1-xyMax.y/xast.y)+xast.x*xyMax.y/xast.y,xyMax.y));          
            draw((-.5,0)..xast..(5.8,0)..(xast.x,-xast.y)..(-.5,0),dashed);  
            xaxis(Bottom,xyMin.x,xyMax.x,RightTicks(Step=3,step=0.5),above=true);
            yaxis(Left,xyMin.y,xyMax.y,LeftTicks(Step=5,step=1),above=true);
        \end{asy}  
        \caption{\(\abs{\delta}^3\ll E_-\)} 
    \end{subfigure}%
    \begin{subfigure}[b]{.33\linewidth}
        \centering
        \begin{asy}
            size(4cm,6cm,IgnoreAspect);
            pair xyMin=(-3,-6);  
            pair xyMax=(6,6);
            var delta = -.4;
            var ee = 0.00600295;
            pair xast = (4.058,0.0110461);
            real ff(pair z) {return xpart( log(1+z)-log(z) -(1-delta)/(1+z) - ee*z ) ;}
            
            real f(real x, real y) {return ff((x,y));}
            var ffS = ff(xast);
            picture bar;
            bounds range=image(f,Range(-.2,.2),xyMin,xyMax,N,Palette);
            palette(range,(xyMin.x,xyMax.y+1),(xyMax.x,xyMax.y+2),Top,Palette,PaletteTicks(N=2,n=2));
            defaultpen(.5bp);
            draw(contour(ff,xyMin,xyMax,new real[] {ffS},N,operator ..),white);
            draw((0,0)--xast--(1.1*xast.x,.2*xast.x)--(xast.x-(xyMax.y-xast.y)/3,xyMax.y));
            draw(xast..(1,-1.5)..(-.5,0)..(1,1.5)..xast,dashed);
            xaxis(Bottom,xyMin.x,xyMax.x,RightTicks(Step=3,step=0.5),above=true);
            \end{asy}  
        \caption{\(\abs[0]{E_-}\ll \abs{\delta}^3\)} 
    \end{subfigure}%
    \begin{subfigure}[b]{.33\linewidth}
        \centering
        \begin{asy}
            size(4cm,6cm,IgnoreAspect);
            pair xyMin=(-3,-6);  
            pair xyMax=(6,6);
            var delta = -.4;
            var ee = 0.00500285;
            pair xast = (3.21361 ,0);
            real ff(pair z) {return xpart( log(1+z)-log(z) -(1-delta)/(1+z) - ee*z ) ;}
            
            real f(real x, real y) {return ff((x,y));}
            var ffS = ff(xast);
            picture bar;
            bounds range=image(f,Range(-.2,.2),xyMin,xyMax,N,Palette);
            palette(range,(xyMin.x,xyMax.y+1),(xyMax.x,xyMax.y+2),Top,Palette,PaletteTicks(N=2,n=2));
            defaultpen(.5bp);
            draw(contour(ff,xyMin,xyMax,new real[] {ffS},N,operator ..),white);
            draw((0,0)--(5,0)--(2,xyMax.y));
            draw(xast..(1,-1.5)..(-.5,0)..(1,1.5)..xast,dashed);
            xaxis(Bottom,xyMin.x,xyMax.x,RightTicks(Step=3,step=0.5),above=true);
            yaxis(Right,xyMin.y,xyMax.y,RightTicks(Step=5,step=1),above=true);
            \end{asy}  
        \caption{\( E_-\ll -\abs{\delta}^3\)} 
    \end{subfigure}\\
    \begin{subfigure}[b]{.49\linewidth}
        \centering
        \begin{asy}
            size(6cm,6cm,IgnoreAspect);
            pair xyMin=(-3,-6);  
            pair xyMax=(6,6);
            var delta = -.3;
            var ee = 0.0033048;
            pair xast = (4.89538,1.30331);
            real ff(pair z) {return xpart( log(1+z)-log(z) -(1-delta)/(1+z) - ee*z ) ;}
            
            real f(real x, real y) {return ff((x,y));}
            var ffS = ff(xast);
            picture bar;
            bounds range=image(f,Range(-.2,.2),xyMin,xyMax,N,Palette);
            palette(range,(xyMin.x,xyMax.y+1),(xyMax.x,xyMax.y+2),Top,Palette,PaletteTicks(N=2,n=2));
            defaultpen(.5bp);
            draw(contour(ff,xyMin,xyMax,new real[] {ffS},N,operator ..),white);
            draw((0,0)--(.95*xast.x,0)--xast--(1.05*xast.x,2*xast.y)--(.9*xast.x,xyMax.y));
            xaxis(Bottom,xyMin.x,xyMax.x,RightTicks(Step=3,step=0.5),above=true);
            draw((5.5,0)..(xast.x,-xast.y)..(1,-1.5)..(-.5,0)..(1,1.5)..xast..(5.5,0),dashed);
            yaxis(Left,xyMin.y,xyMax.y,LeftTicks(Step=5,step=1),above=true);
            \end{asy}  
        \caption{\( E_-\sim \abs{\delta}^3\)} 
    \end{subfigure}%
    \begin{subfigure}[b]{.49\linewidth}
        \centering
        \begin{asy}
            size(6cm,6cm,IgnoreAspect);
            pair xyMin=(-3,-6);  
            pair xyMax=(6,6);
            var delta = -.3;
            var ee = 0.0023048;
            pair xast = (4.21283,0);
            real ff(pair z) {return xpart( log(1+z)-log(z) -(1-delta)/(1+z) - ee*z ) ;}
            
            real f(real x, real y) {return ff((x,y));}
            var ffS = ff(xast);
            picture bar;
            bounds range=image(f,Range(-.2,.2),xyMin,xyMax,N,Palette);
            palette(range,(xyMin.x,xyMax.y+1),(xyMax.x,xyMax.y+2),Top,Palette,PaletteTicks(N=2,n=2));
            defaultpen(.5bp);
            draw(contour(ff,xyMin,xyMax,new real[] {ffS},N,operator ..),white);
            draw((0,0)--(5.8,0)--(.9*xast.x,xyMax.y));
            draw(xast..(1.5,-2)..(-.5,0)..(1.5,2)..xast,dashed);
            xaxis(Bottom,xyMin.x,xyMax.x,RightTicks(Step=3,step=0.5),above=true);
            yaxis(Right,xyMin.y,xyMax.y,RightTicks(Step=5,step=1),above=true);
            \end{asy}  
        \caption{\( E_-\sim -\abs{\delta}^3\)}
    \end{subfigure}%
    \caption{Contour plot of \(\Re f(x)\) in the regime \(E\approx \ed_-\). The solid white lines represent the level set \(\Re f(x)=\Re f(x_\ast)\), while the solid and dashed black lines represent the chosen contours for the \(x\)- and \(y\)-integrations, respectively.}
    \label{fig:cont_em}
\end{figure}
Around the spectral edge \(\ed_-\) the critical point admits the asymptotic expansion
\begin{subequations}
\begin{equation} \label{eq:saddlenegdelta}  \begin{split}x_\ast&= x_\ast(\ed_-) + \gamma_- \Bigl(1+\landauO{\frac{\abs{E}^{1/2}}{\delta^{3/2}}}\Bigr) \begin{cases}\ii\sqrt{\abs{E_-}}, & E_-\ge 0\\     -\sqrt{\abs{E_-}}, & E_-\le 0 \end{cases} 
\\
    &= \frac{1}{E^{1/3}}\Bigl(e^{\ii\pi/3}+\frac{\ii}{3}e^{\ii\pi/3}\frac{\delta}{E^{1/3}}\bigl(1+\landauO{\abs{E}^{1/3}}\bigr)+\landauO{\frac{\abs{\delta}^2}{E^{2/3}}}\Bigr),
\end{split}
\end{equation}
where \(E_-\defeq E-\ed_-\), and these separate expansions are relevant in the \(\abs[0]{E}\ll \abs{\delta}^3\) and \(\abs[0]{E}\gg \abs{\delta}^3\) regimes, respectively. The location of saddle point(s) in the regime \(E\approx \ed_-\) is depicted in Figure~\ref{fig:cont_em}. The second derivative around \(x_\ast\) is given by 
\begin{equation} \begin{split}
    f''(x_\ast) &= \frac{2}{\gamma_-}\Bigl(1+\landauO{\abs{E_-}^{1/2} \abs{\delta}^{-3/2}}\Bigr) \times \begin{cases}
        \sqrt{\abs{E_-}},& E_-\le 0\\ -\ii\sqrt{\abs{E_-}}, & E_-\ge 0\end{cases}\\
        &= 3e^{2\ii\pi/3}E^{4/3}\Bigl(1+\landauO{E+\frac{\abs{\delta}}{E^{1/3}}}\Bigr),
\end{split} \label{fpp e-}\end{equation}
with \(\gamma_-\sim\abs{\delta}^{-5/2}\).
\end{subequations}

\begin{proof}[Proof of Proposition~\ref{prop local law}]
As the functions \(f\) and \(G\) in~\eqref{H-z cplx superbos} are meromorphic we are free to deform the contours for the \(x\)- and \(y\)-integrals as long as we are not crossing \(0\) or \(-1\) and the \(x\)-contour goes out from \(0\) in the "right" direction (in the region $\Re[x]<0,\Im[x]>0$ in Figure~\ref{fig:cont_ep}, and in the region $\Re[x]>|\Im[x]|$ in Figures~\ref{fig:contn}--\ref{fig:cont_em}). It is easy to see that the contours can always be deformed in such a way that \(\Re f(x)>\Re f(x_\ast)=\Re f(\overline{x_\ast})\) and \(\Re f(y)<\Re f(x_\ast)\) for all \(x,y \ne x_\ast,\overline{x_\ast}\), see Figures~\ref{fig:cont_ep}--\ref{fig:cont_em} for an illustration of the chosen contours. 

We now compute 
the integral~\eqref{H-z cplx superbos} in the large $N$ limit when $E$ is near the edges. 
In certain regimes of the parameters $N$, $E$ and $\delta$ 
a saddle point analysis is applicable after a suitable contour deformation. In most cases, the result is
a point evaluation of the integrand at the saddle points. In some  transition regimes of the parameters
the saddle point analysis only  allows us  to explicitly scale out some combination of the parameters
and leaving an integral depending only on a reduced set of  rescaled parameters.

We recall the   classical quadratic saddle point approximation for holomorphic functions \(f(z),g(z)\) such that \(f(z)\) has a unique critical point in some \(z_\ast\), and that \(\gamma\) can be deformed to go through \(z_\ast\) in such a way that \(\Re f(z)<\Re f(z_\ast)\) for all \(\gamma\ni z\ne z_\ast\). Then for large \(\lambda\gg1\) the saddle point approximation is given by
\begin{equation} \int_\gamma g(z) e^{\lambda f(z)}\diff z = \pm g(z_\ast) e^{\lambda f(z_\ast)} \sqrt{\frac{2\pi}{\lambda\abs{f''(z_\ast)}}} \ii e^{-\frac{\ii}{2} \arg f''(z_\ast)}\Bigl(1+\landauO{\frac{1}{\lambda}}\Bigr),\label{quad saddle}\end{equation}
where \(\pm\) is determined by the direction of \(\gamma\) through \(z_\ast\) with \(+\) corresponding to the direction parallel to \(\ii \exp(-\frac{\ii}{2}\arg f''(z_\ast))\). 
This formula is applicable, i.e.~we can use
 point evaluation in the \emph{saddle point regime}, whenever the lengthscale \(\ell_f \sim (N |f''(x_\ast)|)^{-1/2}\) of the exponential decay
 from the quadratic approximation of the
 phase function is much smaller than the scale \(\ell_g \sim |g(x_\ast)|/|\nabla g(x_\ast)|\) on which \(g\) is essentially 
 unchanged.
 For our integral~\eqref{H-z cplx superbos} we thus need to check 
  the condition 
\begin{equation}\label{saddlecond}
\frac{1}{\sqrt{N \abs{f''(x_\ast)}}} \ll \abs{\frac{\nabla (y G(x,y))}{y G(x,y)}\Big\rvert_{(x,y)=(x_\ast,x_\ast)} }^{-1}  
\end{equation}
in all regimes separately.

In the regime \(E\approx \ed_+\), using the asymptotics for $x_\ast$ from~\eqref{xast} 
 and~\eqref{fpp e+}, the quadratic saddle point approximation is valid if 
\[ \frac{1}{\sqrt{N \abs{E_+}^{1/2} }} \ll \abs{E_+}^{1/2},\]
i.e.~if \(\abs{E_+}\gg N^{-2/3}\). Here the length-scale \(\abs{E_+}^{1/2}\) represents the length-scale on which \((x,y)\mapsto yG(x,y)\) is essentially constant which can be obtained by explicitly computing  the log-derivative
\[\abs{\frac{\nabla (y G(x,y))}{y G(x,y)}\Big\rvert_{(x,y)=(x_\ast,x_\ast)} } \sim \abs{E_+}^{-1/2}.\]
 Similar calculations yield that for \(E\approx 0\) and \(\delta\ge 0\) the quadratic saddle point approximation is valid if
\[\frac{1}{\sqrt{N\bigl(E^{4/3}\wedge E^{3/2}\delta^{-1/2}\bigr)}}\ll E^{-1/3}\vee \delta^{1/2}E^{-1/2},
\quad \text{i.e.} \quad  E \gg N^{-3/2} \wedge N^{-2}\delta^{-1}, \]
while for \(E\approx\ed_-\) and \(\delta<0\) the condition~\eqref{saddlecond} reads
\[\frac{1}{\sqrt{N\bigl(E^{4/3}\vee\abs{E_-}^{1/2}\abs{\delta}^{5/2}\bigr)}}\ll E^{-1/3}\wedge \abs{E_-}^{1/2}\abs{\delta}^{-5/2},
\quad \text{i.e.} \quad \abs{E_-}\gg N^{-2/3}\abs{\delta}^{5/3}, \]
recalling that \(E= E_- +\ed_-\) and  \(\ed_- \sim \delta^3\) from~\eqref{m asymp}.

In these regimes we can thus  apply~\eqref{quad saddle}
to~\eqref{H-z cplx superbos} and using that 
\[ G(x_\ast,x_\ast)=f''(x_\ast),\qquad G(x_\ast,\overline{x_\ast})=0,\]
as follows from explicit computations, we thus finally conclude~\eqref{eq:fineqsaddle}. Here the error terms in~\eqref{eq:fineqsaddle} follow from~\eqref{quad saddle} by choosing \(\lambda\sim (\ell_g/\ell_f)^2\) according to asymptotics of the second derivatives and log-derivatives above. More precisely, for example in the second case \(E\approx 0\) and \(E^{1/3}\gg\delta>0\), the phase function \(f\) is approximately given by 
\[f(x)\approx E^{2/3}\Bigl[\frac{1}{2(E^{1/3}x)}-E^{1/3}x\Bigr],\]
while \(G\) can asymptotically be written as 
\[ y G(x,y)\approx x_\ast G(x_\ast,x_\ast) +3 E^{4/3} e^{2\ii\pi/3}\Bigl(2 (x-x_\ast)+(y-x_\ast)\Bigr) +\landauO{E^{5/3}\Bigl(\abs{x-x_\ast}^2+\abs{y-x_\ast}^2\Bigr)}.\]
Thus we make the change of variables \(x= x_\ast + E^{-1/3}x'\), \(y=x_\ast+ E^{-1/3}y'\) 
to find 
\[ \begin{split}&\frac{N^{2}}{2\pi \ii} E^{-2/3} \int\diff x' \int\diff y' e^{-N E^{-2/3} f''(x_\ast)\Bigl( \frac{x'^2}{2} -\frac{y'^2}{2}\Bigr)-N E^{-1} f'''(x_\ast)\Bigl( \frac{x'^2}{6} -\frac{y'^2}{6}\Bigr)+ \landauO{N E^{2/3} (\abs{x'}^4+\abs{y'}^4) }} \\&\qquad\qquad \qquad\qquad\qquad \times\Bigl( x_\ast G(x_\ast,x_\ast) + 3 E e^{2\ii\pi/3} (y'+2x') + \landauO{E (\abs{x'}^2+\abs{y'}^2)}\Bigr)\\
&\quad= x_\ast\Bigl( 1+ \landauO{\frac{1}{N E^{2/3}}} \Bigr),\end{split}\]
where we used that \(G(x_\ast,x_\ast)\sim E^{4/3}\). The other cases in~\eqref{eq:fineqsaddle} can be checked similarly. 
\end{proof}

\section{Derivation of the \(1\)-point function in the critical regime for the complex case}\label{sec cplx}
In this section we prove Theorem~\ref{thm cplx}, i.e.~we study $\E\Tr[Y-w]^{-1}$, with $w=E+\ii \epsilon$, $1\gg \abs{E}\gg \epsilon>0$, for $E$ so close to $0$ such that $\abs{E}$ is smaller or comparable with the eigenvalues scaling around $0$. We will first consider the case \(\delta\ge 0\) and afterwards explain the necessary changes in the regime \(-C N^{-1/2}\le \delta<0\). 

\subsection{Case $0\le \delta \le 1$.}
\label{sec:posdec}
In the following of this section we assume that $E>0$, since we are interested in the computations of~\eqref{H-z cplx superbos} for $E=\Re[w]$ inside the spectrum of $Y$. In order to study the transition between the local law regime, that is considered in Section \ref{sec:saddlepointanalysis}, and the regime when the main contribution to~\eqref{H-z cplx superbos} comes from the smallest eigenvalue of $Y$, we define the parameter
\begin{equation}
\label{eq:thres}
c(N)=c(N,\delta)\defeq \frac{1}{N^{3/2}}\wedge \frac{1}{\delta N^2}.
\end{equation}

In particular, in the regime $E\gg c(N)$ the double integral in~\eqref{H-z cplx superbos} is computed by saddle point analysis in ~\eqref{eq:fineqsaddle}, i.e.~the main contribution comes from the regime around the stationary point $x_*$ of $f$, with $x_*$ defined in~\eqref{eq:statpoi}, whilst for $E\lesssim c(N)$ the main contribution to~\eqref{H-z cplx superbos} comes from a larger regime around the stationary point $x_*$. From now on we assume that $E\lesssim c(N)$. In the following we denote the leading order of the stationary point $x_*$ by
\begin{equation}
\label{eq:impsadd}
z_*= z_*(E, \delta) \defeq E^{-1/3}\Psi(\delta E^{-1/3}),
\end{equation}
where \(\Psi(\lambda)\) was defined in~\eqref{xast >0} and has the asymptotics \(\Psi(0)=e^{\ii\pi/3}\) and \(\Psi(\lambda)/\sqrt{\lambda}\to\ii\) as \(\lambda\to\infty\). Note that $|z_*|\gg 1$ for any $E\ll 1$, $0\le \delta\le 1$. For this reason, we expect that the main contribution to the double integral in~\eqref{H-z cplx superbos} comes from the regime when $|x|$ and $|y|$ are both large, say $|x|, |y|\ge N^\rho$, for some small fixed $0< \rho<1/2$. Later on in this section, see Lemma~\ref{lem:expsmall}, we prove that the contribution to~\eqref{H-z cplx superbos} in the regime when either $|x|$ or $|y|$ are smaller than $N^\rho$ is exponentially small. In order to get the asymptotics in~\eqref{eq:onepointest}, is not affordable to estimate the error terms in the Taylor expansion by absolute value. In particular, it is not affordable to estimate the integral of $e^{Nf(y)}$ over $\Gamma$ by absolute value, hence the improved bound in~\eqref{eq:exactyint} is needed. To make our writing easier, for any $R\in\N  $, $R\ge 2$, we introduce the notation 
\begin{equation}
\label{eq:diffO}
\mathcal{O}^\#\big(x^{-R}\big)\defeq \{ g\in \mathcal{P}_R\}, \quad \mathcal{O}^\#\big((x,y)^{-R}\big)\defeq \{ g\in \mathcal{Q}_R\},
\end{equation}
where $\mathcal{P}_R$ and $\mathcal{Q}_R$ are defined as Laurent series of order at least $R$ around infinity, i.e.
\[
\begin{split}
\mathcal{P}_R &\defeq\Set{ g\colon\C \to \C  | g(x)= \sum_{\alpha\ge R} \frac{\tilde{c}_{\alpha}}{x^\alpha }, \,\, \text{with} \,\, |\tilde{c}_{\alpha}|\le C^{\alpha}, \,\, \text{if}\,\,  |x|\ge 2C }, \\
\mathcal{Q}_R &\defeq\Set{ \widetilde{g}\colon\C \times \C \to \C  | \widetilde{g}(x,y)= \sum_{\alpha, \beta \ge 1, \alpha+ \beta \ge R} \frac{c_{\alpha,\beta}}{x^\alpha y^\beta}, \,\, \text{with} \,\, |c_{\alpha,\beta}|\le C^{\alpha+\beta}, \,\, \text{if}\,\, |x|, |y|\ge 2C},
\end{split}
\]
for some constant $C>0$ that is implicit in the $\mathcal{O}^\#$ notation. Here $\alpha, \beta$ are integer exponents. Note that $\mathcal{O}^\#(\abs{x}^{-R})=\mathcal{O}(\abs{x}^{-R})$ for any \(x\in\C\). Then, we expand the phase function $f$ for large argument as follows
\begin{subequations}
    \begin{equation}
        \label{eq:expansionf}
        f(x)=g(x)+\mathcal{O}^\#\big(x^{-3}+\delta x^{-2}\big),\quad g(x)\defeq -(E+\ii \epsilon) x+\frac{\delta}{x}+\frac{1}{2x^2}.
        \end{equation}
        and for large $x$ and $y$ we expand $G$ as
        \begin{equation}
        \label{eq:expansionG}
        G(x,y)=H(x,y)+ \mathcal{O}^\#\left((x,y)^{-5}+\delta (x,y)^{-4} \right),\quad H(x,y)\defeq \frac{1}{x^3y}+\frac{1}{x^2 y^2}+\frac{1}{xy^3}+\frac{\delta}{xy^2}+\frac{\delta}{x^2y}.
        \end{equation}
\end{subequations}

\begin{figure}
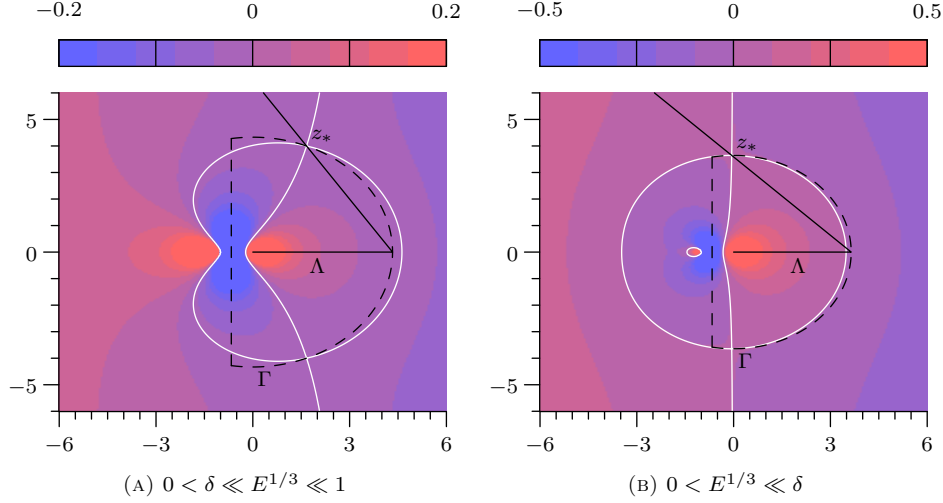

    \centering
    \begin{subfigure}[b]{.5\linewidth}
        \centering
        \begin{asy}
            size(6cm,6cm,IgnoreAspect);
    pair xyMin=(-6.01,-6.01);
    pair xyMax=(6,6);
    var delta = .0000001;
    var ee = .01;
    pair xast = (1.66553 , 3.99801 );
    var ra = sqrt(xast.x^2+xast.y^2);
    real ff(pair z) {return xpart( log(1+z)-log(z) -(1-delta)/(1+z) - ee*z ) ;}
    var c = 4;
    
    real f(real x, real y) {return ff((x,y));}
    var ffS = ff(xast);
    label("$z_\ast$",xast,align=NE);
    picture bar;
    bounds range=image(f,Range(-.2,.2),xyMin,xyMax,N,Palette);
    palette(range,(xyMin.x,xyMax.y+1),(xyMax.x,xyMax.y+2),Top,Palette,PaletteTicks(N=2,n=2));
    defaultpen(.5bp);
    draw(contour(ff,xyMin,xyMax,new real[] {ffS},N,operator ..),white);    
    draw((0,0)--(ra,0)--(ra+(xast.x-ra)*xyMax.y/xast.y,xyMax.y));
    pair gam(real t) { return (ra*cos(t), ra*sin(t)); }
    write(acos(-2/3/ra));
    path gamg = graph(gam, -acos(-2/3/ra), acos(-2/3/ra));
    draw(gamg,dashed);
    //draw((-.7,0)..(-.5,2)..xast..(5.8,0)..(xast.x,-xast.y)..(-.5,-2)..(-.7,0),dotted);
    draw((-2/3,-sqrt(ra^2-4/9))--(-2/3,sqrt(ra^2-4/9)),dashed);
    label("$\Gamma$",(0,-ra),align=SE);
    label("$\Lambda$",(2,0),align=S);
    xaxis(Bottom,xyMin.x,xyMax.x,RightTicks(Step=3,step=0.5),above=true);
    yaxis(Left,xyMin.y,xyMax.y,LeftTicks(Step=5,step=1),above=true);
        \end{asy}   
        \caption{\(0<\delta\ll E^{1/3}\ll 1\)} 
    \end{subfigure}%
    \begin{subfigure}[b]{.5\linewidth}
        \centering
        \begin{asy}
            size(6cm,6cm,IgnoreAspect);
    pair xyMin=(-6.01,-6.01);
    pair xyMax=(6,6);
    var delta = .5;
        var ee = .04;
         pair xast = (-0.0586703, 3.64358 );
        var ra = sqrt(xast.x^2+xast.y^2);
    real ff(pair z) {return xpart( log(1+z)-log(z) -(1-delta)/(1+z) - ee*z ) ;}
    var c = 4;
    
    real f(real x, real y) {return ff((x,y));}
    var ffS = ff(xast);
    //fill(circle(xast,2.));
    label("$z_\ast$",xast,align=NE);
    picture bar;
    bounds range=image(f,Range(-.5,.5),xyMin,xyMax,N,Palette);
    palette(range,(xyMin.x,xyMax.y+1),(xyMax.x,xyMax.y+2),Top,Palette,PaletteTicks(N=2,n=2));
    defaultpen(.5bp);
    draw(contour(ff,xyMin,xyMax,new real[] {ffS},N,operator ..),white);
    //draw((0,0)--(c,0)--(c+(xast.x-c)*xyMax.y/xast.y,xyMax.y));
    draw((0,0)--(ra,0)--(ra+(xast.x-ra)*xyMax.y/xast.y,xyMax.y));
    pair gam(real t) { return (ra*cos(t), ra*sin(t)); }
    path gamg = graph(gam, -acos(-2/3/ra), acos(-2/3/ra));
    draw(gamg,dashed);
    //draw((-.7,0)..(-.5,2)..xast..(5.8,0)..(xast.x,-xast.y)..(-.5,-2)..(-.7,0),dotted);
    draw((-2/3,-sqrt(ra^2-4/9))--(-2/3,sqrt(ra^2-4/9)),dashed);
    label("$\Gamma$",(0,-ra),align=SE);
    label("$\Lambda$",(2,0),align=S);
    xaxis(Bottom,xyMin.x,xyMax.x,RightTicks(Step=3,step=0.5),above=true);
    yaxis(Left,xyMin.y,xyMax.y,LeftTicks(Step=5,step=1),above=true);
        \end{asy}   
        \caption{\(0< E^{1/3}\ll \delta\)} 
    \end{subfigure}
    \caption{Illustration of the contours~\eqref{eq:defycon}--\eqref{eq:defxcon} together with the phase diagram of \(\Re f\), where the white line represents the level set \(\Re f(x)=\Re f(z_\ast)\). Note that the precise choice of the contours is only important close to \(0\) and for very large \(|x|\) as otherwise the phase function is small.}\label{fig:cont_circle}
\end{figure}

In order to compute the integral in~\eqref{H-z cplx superbos} we deform the contours $\Lambda$ and $\Gamma$ through $z_*$, with $z_*$ defined in~\eqref{eq:impsadd}. In particular, we are allowed to deform the contours as long as the $x$-contour goes out from zero in region $\Re[x]>|\Im[x]|$, it ends in the region $\Re[x]<0,\Im[x]>0$, and it does not cross $0$ and $-1$ along the deformation; the $y$-contour, instead, can be freely deformed as long as it does not cross $0$ and $-1$. Hence, we can deform the $y$-contour as $\Gamma=\Gamma_{z_*}\defeq \Gamma_{1,z_*}\cup\Gamma_{2,z_*}$, where
\begin{subequations}
    \begin{equation}
        \label{eq:defycon}
        \Gamma_{1,z_*}\defeq\left\{ -\frac{2}{3}+\ii t: 0\le |t|\le \sqrt{|z_*|^2-\frac{4}{9}}\right\}, \quad \Gamma_{2,z_*}\defeq \left\{|z^*|e^{\ii\psi}: \, \psi \in \left[ -\psi_{z_*},\psi_{z_*}\right]\right\},
        \end{equation}
        with $\psi_{z_*}=\arccos [-2/(3|z_*|)]$, and the $x$-contour as $\Lambda=\Lambda_{z_*}\defeq \Lambda_{1,z_*}\cup \Lambda_{2,z_*}$, with
        \begin{equation}
        \label{eq:defxcon}
        \Lambda_{1,z_*} \defeq [0, |z_*|), \quad \Lambda_{2,z_*}\defeq \left\{|z_*|-qs+\ii s: \, s\in [0,+\infty)\right\},
        \end{equation}
        where
        \begin{equation}
        \label{eq:defq}
        q=q_{z_*}\defeq\Im [z_*]^{-1}(|z_*|-\Re[z_*]).
        \end{equation}
\end{subequations}
Note that $q\sim 1$ uniformly in $N$, $E$ and $\delta$, since $\Re[z_*]\lesssim \Im[z_*]$ for any $E\ll1$, $0\le \delta\le 1$. We assume the convention that the orientation of $\Gamma$ is counter clockwise. See Figure~\ref{fig:cont_circle} for an illustration of \(\Gamma\) and \(\Lambda\).

Before proceeding with the computation of the leading term of~\eqref{H-z cplx superbos}, in the following lemma we state some properties of the function $f$ on the contours $\Gamma$, $\Lambda$. Using that $\epsilon\ll E$,  the proof of the lemma below follows by easy computations.

\begin{lemma}
\label{lem:propf}
Let $f$ be the phase function defined in~\eqref{H-z cplx superbos}, then the following properties hold true:
\begin{enumerate}[(i)]

\item \label{it:reimf} For any $y=-2/3+\ii t\in \Gamma_{1,z_*}$, we have that
\begin{equation}
\label{eq:ref}
\Re[f(-2/3+\ii t)]=\frac{2E}{3}+\epsilon t-\frac{1}{2t^2}+\mathcal{O}(|t|^{-3}+\delta|t|^{-2}),
\end{equation}
and
\begin{equation}
\label{eq:imf}
\Im[f(-2/3+\ii t)]=\frac{2\epsilon}{3}-Et-\frac{\delta}{t}+\mathcal{O}(|t|^{-3}).
\end{equation}

\item \label{it:incf} For $\epsilon=0$, the function $t\mapsto \Re[f(-2/3+\ii t)]$ on $\Gamma_{1,z_*}$ is strictly increasing if $t>0$ and strictly decreasing if $t<0$.

\item \label{it:dec} The function $x\mapsto \Re[f(x)]$ is strictly decreasing on $\Lambda_{1,z_*}$.

\item \label{it:par2} Let $x\in \Lambda_{2,z_*}$ be parametrized as $x=|z_*|-qs+\ii s$, for $s\in [0, +\infty)$, with $q$ defined in~\eqref{eq:defq}, then
\begin{equation}
\label{eq:realimpexp}
\begin{split}
\Re[f(x)]&=-E(|z_*|-qs)+\epsilon s+\frac{\delta (|z_*|-qs)}{s^2+(qs-|z_*|)^2}\\
&\quad+\frac{(1-2\delta)[(|z_*|-qs)^2-s^2]}{(s^2+(qs-|z_*|)^2)^2}+\mathcal{O}\left([s^2+|z_*|^2]^{-3/2}\right).
\end{split}\end{equation}
\end{enumerate}
\end{lemma}

Despite the fact that saddle point analysis is not useful anymore in this regime, we expect that the main contribution to~\eqref{H-z cplx superbos} comes from the regime in the double integral when both $x$ and $y$ are large, i.e.~$|x|, |y|\ge N^\rho$. For this purpose we define
\begin{equation}
\label{eq:smallcont}
\widetilde{\Lambda}\defeq\{ x\in \Lambda_{1,z_*} : |x|\le N^\rho\}, \quad \widetilde{\Gamma}\defeq \{ y\in \Gamma_{1,z_*}: |y|\le N^\rho\}.
\end{equation}
In the following part of this section we will firstly prove that the contribution to~\eqref{H-z cplx superbos} in the regime when either $x\in \widetilde{\Lambda}$ or $y\in \widetilde{\Gamma}$ is exponentially small and then we explicitly compute 
the leading term of~\eqref{H-z cplx superbos} in the regime $(x,y)\in (\Lambda\setminus\widetilde{\Lambda})\times(\Gamma\setminus\widetilde{\Gamma})$. For this purpose, we first prove a bound for the double integral in the regime $y\in \Gamma\setminus\widetilde{\Gamma}$ or $x\in \Lambda\setminus\widetilde{\Lambda}$ in Lemma~\ref{lem:apriorbl} and Lemma~\ref{lem:apriorx}, respectively, and then we conclude the estimate for $x\in \widetilde{\Lambda}$ or $y\in \widetilde{\Gamma}$ in Lemma~\ref{lem:expsmall}. Finally, in Theorem~\ref{thm cplx} we consider the regime $(x,y)\in (\Lambda\setminus\widetilde{\Lambda})\times(\Gamma\setminus\widetilde{\Gamma})$ and compute the leading term of~\eqref{H-z cplx superbos}.

\begin{lemma}
\label{lem:apriorbl}
Let $c(N)$ be defined in~\eqref{eq:thres}, $E\lesssim c(N)$, $b\in \N  $, and let  $f$ be defined in~\eqref{H-z cplx superbos}, then
\begin{equation}
\label{eq:easbony}
\left|\int_{\Gamma\setminus\widetilde{\Gamma}} \frac{e^{Nf(y)}}{y^b}\, \diff y \right| \lesssim |z_*|^{1-b}+
\begin{cases}
\abs{z_*}, &  b=0, \\
1+\abs{\log ( N|z_*|^{-2})},  & b=1,  \, \delta<|z_*|^{-1}, \\
 1+\abs{\log ( N\delta |z_*|^{-1})},  &   b=1, \, \delta\ge |z_*|^{-1}, \\
N^\frac{1-b}{2}\wedge (N\delta)^{1-b},& b\ge 2,
 \end{cases}
\end{equation}
where $\Gamma$, $\widetilde{\Gamma}$ are defined in~\eqref{eq:defycon} and~\eqref{eq:smallcont} respectively. Furthermore, we have that
\begin{equation}
\label{eq:exactyint}
\int_{\Gamma\setminus\widetilde{\Gamma}}e^{Nf(y)}\, \diff y=\mathcal{O}(N^{1/2}\vee (N\delta)), \quad \int_{\Gamma\setminus\widetilde{\Gamma}}\frac{e^{Nf(y)}}{y}\, \diff y=\mathcal{O}(1).
\end{equation}
\begin{proof}
Firstly, we notice that if $y\in \Gamma\setminus\widetilde{\Gamma}$ then $|y|\ge N^\rho$, hence we expand $f$ as in~\eqref{eq:expansionf}, i.e.
\begin{equation}
\label{eq:recexf}
f(y)=-(E+\ii \epsilon) y+\frac{\delta}{y}+\frac{1}{2y^2}+\mathcal{O}(|y|^{-3}+\delta|y|^{-2}).
\end{equation}
Moreover, by~\eqref{eq:impsadd} it follows that $|z_*|\sim E^{-1/3}\vee \sqrt{\delta E^{-1}}$, and so that
\begin{equation}
\label{eq:usbound}
|Nf(y)|\lesssim NE^{2/3}+N\sqrt{\delta E} 
\end{equation}
for any $|y|\sim |z_*|$, $E\lesssim c(N)$. Note that $\Gamma\setminus\widetilde{\Gamma}=(\Gamma_{1,z_*}\setminus \widetilde{\Gamma})\cup \Gamma_{2,z_*}$, with $\Gamma_{1,z_*}$, $\Gamma_{2,z_*}$ defined in~\eqref{eq:defycon}. By~\eqref{eq:usbound} it easily follows that $|Nf(y)|\lesssim 1$ for any $y\in \Gamma_{2,z_*}$, which clearly implies that
\begin{equation}
\label{eq:estcircinty}
\int_{\Gamma_{2,z_*}}\left|\frac{e^{Nf(y)}}{y^b}\right||\diff y|\lesssim |z_*|^{1-b}.
\end{equation}
To conclude the proof of~\eqref{eq:easbony} we bound the integral on $\Gamma_{1,z_*}\setminus\widetilde{\Gamma}$. Let $w=E+\ii \epsilon$, then in this regime, by~\eqref{eq:ref}--\eqref{eq:imf}, we have 
\begin{equation}
\label{eq:bettcomp}
\begin{split}
    \int_{\Gamma_{1,z_*}\setminus\widetilde{\Gamma}}\frac{e^{Nf(y)}}{y^b}\diff y=-\ii \int_{N^\rho}^{\sqrt{|z_*|^2-4/9}} &e^{-N\left[\frac{1}{2t^2}+\mathcal{O}(|t|^{-3}+\delta |t|^{-2})\right]}\\
    &\times\left(\frac{e^{-N\left[\ii wt+\ii \frac{\delta}{t}\right]}}{(-2/3+\ii t)^b}+\frac{e^{N\left[\ii w t+\ii \frac{\delta}{t}\right]}}{(-2/3-\ii t)^b} \right) (1+\mathcal{O}(NE)) \diff t.    
\end{split}
\end{equation}
For any $b\in \N  $, we estimate the integral above as follows
\begin{equation}
\label{eq:firstbbo}
\left| \int_{\Gamma_{1,z_*}\setminus\widetilde{\Gamma}}\frac{e^{Nf(y)}}{y^b}\,\diff y\right|  \lesssim  \left|  \int_{N^\rho}^{|z_*|} \frac{e^{-\frac{N}{2t^2}+\frac{\ii\delta N}{t}}}{t^b }\, \diff t\right| \lesssim \begin{cases}
\abs{z_*}, &  b=0, \\
1+\abs{\log ( N|z_*|^{-2})},  & b=1, \, \delta<|z_*|^{-1}, \\
1+\abs{\log ( N\delta |z_*|^{-1})},  &  b=1, \, \delta\ge |z_*|^{-1}, \\
N^\frac{1-b}{2}\wedge (N\delta)^{1-b},& b\ge 2,
 \end{cases}
\end{equation}
Note that in \eqref{eq:firstbbo} for $b=0,1$ we get the bound in the r.h.s. bringing the absolute value inside the integral, whilst this is not affordable to get the bound for $b\ge 2$, since the term $e^{\ii\delta N/t}$ has to be used. Indeed, we would get a bound $N^{(1-b)/2}$ for $b\ge 2$ if we estimate the integral in \eqref{eq:firstbbo} moving the absolute value inside. In the following part of the proof we compute the integral~\eqref{eq:bettcomp} for $b= 0, 1$ without estimating it by absolute value.

In particular, for $b=1$, we prove that the leading term of the r.h.s.~of~\eqref{eq:bettcomp} is $\landauO{1}$, instead of 
the overestimate $1+|\log (N|z_*|^{-2})|$ in~\eqref{eq:firstbbo}, as a consequence of the symmetry of $\Gamma_{1,z_*}$ respect to $0$. For this computation we have to distinguish the cases $\delta \gg N^{-1/2}$ and $\delta \lesssim N^{-1/2}$. If $E\sim c(N)$ and $\delta\lesssim N^{-1/2}$, then $N|z_*|^{-2}\sim 1$, hence the bound in~\eqref{eq:exactyint} directly follows  by~\eqref{eq:estcircinty} and~\eqref{eq:firstbbo}. We are left with the cases $E\ll c(N)$ and $E\sim c(N)$, $\delta\gg N^{-1/2}$. For $\delta\gg N^{-1/2}$, we have $|z_*|\sim \sqrt{\delta/E}$ and using  $|N w t| \le NE|z_*|\ll 1$, if $E\ll c(N)$, and $|Nwt|\sim 1$, if $E\sim c(N)$, we conclude
\[
\int_{N^\rho}^{\sqrt{|z_*|^2-4/9}} \frac{e^{-N\left[\frac{1}{2t^2}\pm\ii\frac{\delta}{t}\pm \ii t w\right]}}{-2/3+\ii t}\, \diff t =\int_{N^\rho}^{|z_*|} \frac{e^{-\frac{N}{2t^2}\pm \ii \frac{\delta}{t}}}{t}\, \diff t +\mathcal{O}(1) = |\log (N\delta|z_*|^{-1})|+\mathcal{O}(1). 
\]
Similarly we prove that the integral in the l.h.s.\ of the above equalities is equal to $|\log(N|z_*|^{-2} )|+\mathcal{O}(1)$ if $\delta\ll N^{-1/2}$. Similar calculation holds if the denominator is $(-2/3 -\ii t)$ instead of $(-2/3+\ii t)$, just an overall sign changes. Thus the leading terms from the two parts of the integral in~\eqref{eq:bettcomp} cancel each other. We thus conclude the second bound in~\eqref{eq:exactyint} combining the above computations with~\eqref{eq:estcircinty} and~\eqref{eq:bettcomp}.

Next, we compute the integral of $e^{Nf(y)}$ on $\Gamma\setminus\widetilde{\Gamma}$, i.e.~we prove the first bound in~\eqref{eq:exactyint}. We consider only the regime $E\ll c(N)$, since in the regime $E\sim c(N)$ the bound in~\eqref{eq:exactyint} follows directly by~\eqref{eq:estcircinty},\eqref{eq:firstbbo}, and the definition of $c(N)$ in~\eqref{eq:thres}, since $|z_*|\sim N^{1/2}\vee N\delta$. On $\Gamma_{2,z_*}$, using the parametrization $y=|z_*|e^{\ii \psi}$, and that by~\eqref{eq:usbound} we have $|N f(y)|\ll 1$ for $E\ll c(N)$, we Taylor expand $e^{Nf(y)}$ and conclude that
\begin{equation}
\label{eq:excomcircint}
\int_{\Gamma_{2,z_*}} e^{Nf(y)}\, \diff y =2|z_*|\ii\cdot \Big[1+\mathcal{O}\left(NE|z_*|+\delta N |_*|^{-1}+N|z_*|^{-2}\right)\Big].
\end{equation}
Furthermore, by~\eqref{eq:bettcomp} for $b=0$, using that $E\ll c(N)$ and so that $|Nwt|\ll 1$ on $\Gamma_{1,z_*}$, we have
\[
\int_{\Gamma_{1,z_*}\setminus\widetilde{\Gamma}} e^{Nf(y)}\,\diff y=- 2\ii |z_*| +\mathcal{O}(N^{1/2}\vee N\delta).
\]
The minus sign is due to the counter clockwise orientation of $\Gamma$, i.e.~the vertical line  $\Gamma_{1, z_*}$ is parametrized from the top to the bottom. Combining this computation with~\eqref{eq:excomcircint} and using that $NE|z_*|^2+ N\delta +N|z_*|^{-1}\ll N^{1/2}+N\delta$, since $|z_*|\sim E^{-1/3}+\sqrt{\delta E^{-1}}$ by~\eqref{eq:impsadd}, we conclude the proof of this lemma.
\end{proof}
\end{lemma}

\begin{lemma}\label{lem:apriorx}
Let $c(N)$ be defined in~\eqref{eq:thres}, $E\lesssim c(N)$, let $f$ be defined in~\eqref{H-z cplx superbos} and $a\in \R $, then the following bound holds true
\begin{equation}
\label{eq:boundexa}
\int_{\Lambda\setminus\widetilde{\Lambda}} \left|\frac{e^{-Nf(x)}}{x^a}\right|\, |\diff x| \lesssim
\begin{cases}
|z_*|^{1-a} +(NE)^{a-1},&  a<1, \\
1+|\log ( N|z_*|^{-2})|  ,&  a=1, \delta < |z_*|^{-1}, \\
1+|\log ( N\delta|z_*|^{-1})|  ,&  a=1, \delta \ge |z_*|^{-1}, \\
N^\frac{1-a}{2}\wedge (N\delta)^{1-a} ,&  a>1,
 \end{cases}
\end{equation}
where $\Lambda$, $\widetilde{\Lambda}$ are defined in~\eqref{eq:defxcon} and~\eqref{eq:smallcont}.
\begin{proof}
We split the computation of the integral of $e^{N f(x)}x^{-a}$ as the sum of the integral over $\Lambda_{1,z_*}\setminus \widetilde{\Lambda}$ and $\Lambda_{2,z_*}$. Using the parametrization $x=|z_*|-qs+\ii s$, with $s\in [0,+\infty)$ and $q$ defined in~\eqref{eq:defq}, by~\eqref{eq:realimpexp}, we estimate the integral over $\Lambda_{2,z_*}$ as follows
\begin{equation}
\label{eq:faraweq}
\int_{\Lambda_{2,z_*}} \left|\frac{e^{-Nf(x)}}{x^a}\right| \, |\diff x|\lesssim \int_0^{+\infty }\frac{e^{-N\left[ -E(|z_*|-qs)+\frac{\delta(|z_*|-qs)}{(|z_*|-qs)^2+s^2}+\frac{(|z_*|-qs)^2-s^2}{2[(|z_*|-qs)^2+s^2]^2}\right]}}{[(|z_*|-qs)^2+s^2]^{a/2}}\, \diff s.
\end{equation}
We split the computation of the integral in the r.h.s.~of~\eqref{eq:faraweq} into two parts: $|s|\in [0, |z_*|)$ and $s\in [|z_*|,+\infty)$. Since $q\sim 1$ and $NE|z_*|\lesssim 1$, in the regime $|s|\in [0, |z_*|)$ we estimate the integral in the r.h.s.~of~\eqref{eq:faraweq} as
\begin{equation}
\label{eq:uslbou}
e^{NE|z_*|}\int_0^{|z_*|} \frac{e^{-\frac{N\delta}{|z_*|}-\frac{N}{|z_*|^2}}}{|z_*|^a}\, \diff s\lesssim |z_*|^{1-a}.
\end{equation}
In the regime $s\in [|z_*|,+\infty)$, instead, we have
\begin{equation}
\label{eq:farb}
e^{NE|z_*|}\int_{|z_*|}^{+\infty} \frac{e^{-NE qs}}{s^a}\diff s\lesssim 
\begin{cases}
(NE)^{a-1} ,&  a<1, \\
1+|\log (NE|z_*|)| ,&  a=1, \\
|z_*|^{1-a} ,&  a>1.
\end{cases}
\end{equation}

We are left with the estimate of the integral over $\Lambda_{1,z_*}\setminus \widetilde{\Lambda}$. Similarly to the bound in~\eqref{eq:firstbbo}, using that $\Lambda_{1,z_*}\setminus \widetilde{\Lambda}=[N^\rho, |z_*|)$ and that $NE|z_*|\lesssim 1$, we have that
\begin{equation}
\label{eq:mainb}
\int_{\Lambda_{1,z_*}\setminus\widetilde{\Lambda}} \left|\frac{e^{-Nf(x)}}{x^a}\right| \, |\diff x|\lesssim \int_{N^\rho}^{|z_*|} \frac{e^{-\frac{N\delta}{s}-\frac{N}{2s^2}}}{s^a}\diff s\lesssim
\begin{cases}
|z_*|^{1-a} ,&  a<1, \\
1+|\log ( N|z_*|^{-2})|  ,&  a=1, \delta < |z_*|^{-1}, \\
1+|\log ( N\delta|z_*|^{-1})|  ,&  a=1, \delta \ge |z_*|^{-1}, \\
N^\frac{1-a}{2}\wedge (N\delta)^{1-a} ,&  a>1,
 \end{cases}
\end{equation}
Combining~\eqref{eq:faraweq}--\eqref{eq:mainb} we conclude the proof of~\eqref{eq:boundexa}.
\end{proof}
\end{lemma}

Using Lemma~\ref{lem:apriorbl},  Lemma~\ref{lem:apriorx} in the following lemma we prove that the contribution to~\eqref{H-z cplx superbos} in the regime where either $x\in \widetilde{\Lambda}$ or $y\in \widetilde{\Gamma}$ is exponentially small.

\begin{lemma}\label{lem:expsmall}
Let $c(N)$ be defined in~\eqref{eq:thres}, and let $f$, $G$ be defined in~\eqref{H-z cplx superbos}, then, as $\epsilon\to 0^+$, for any $E\lesssim c(N)$ we have that
\begin{equation}
\label{eq:boundsmallreg}
\begin{split}
&\left|\left(\int_\Lambda\diff x\int_\Gamma\diff y-\int_{\Lambda\setminus\widetilde{\Lambda}}\diff x\int_{\Gamma\setminus\widetilde{\Gamma}}\diff y\right)\left[ e^{N[f(y)-f(x)]} y G(x,y)\right]\right| \\
&\qquad\qquad\qquad\qquad\lesssim N^{\rho}(N^{1/2}+N\delta+ |\log (NE^{2/3})|)e^{-\frac{1}{2}N^{1-2\rho}}.
\end{split}
\end{equation}
\begin{proof}
We split the estimate of the integral over $(\Lambda\times\Gamma)\setminus[(\Lambda\setminus\widetilde{\Lambda})\times(\Gamma\setminus\widetilde{\Gamma})]$ into three regimes: $(x,y)\in \widetilde{\Lambda}\times \widetilde{\Gamma}$, $(x,y)\in \widetilde{\Lambda}\times (\Gamma\setminus\widetilde{\Gamma})$, $(x,y)\in (\Lambda\setminus\widetilde{\Lambda})\times \widetilde{\Gamma}$.
By \ref{it:incf}, \ref{it:dec} of Lemma~\ref{lem:propf}, in the regime $y\in \widetilde{\Gamma}$ and $x\in \widetilde{\Lambda}$, respectively, it follows that the function $f$ attains its maximum at $y=-2/3\pm \ii N^\rho$ on $\widetilde{\Gamma}$ and $f$ attains its minimum at $x=N^\rho$ on $\widetilde{\Lambda}$. Hence, by the expansion in~\eqref{eq:expansionf} it follows that
\begin{equation}\label{eq:expbo}
\sup_{y\in\widetilde{\Gamma}}\big|e^{Nf(y)}\big|+\sup_{x\in \widetilde{\Gamma}} \big|e^{-Nf(x)}\big|\lesssim e^{-Nf(N^\rho)},
\end{equation}
with
\begin{equation}
\label{eq:fexp}
f(N^\rho)=\frac{\delta}{N^\rho}+\frac{1}{2N^{2\rho}}+\mathcal{O}(N^{-3\rho}+\delta N^{-2\rho}).
\end{equation}
Then, by~\eqref{eq:expbo} and~\eqref{eq:fexp}, it follows that the integral over $(x,y)\in \widetilde{\Lambda}\times \widetilde{\Gamma}$ is bounded by $N^{2\rho} e^{-N^{1-2\rho}}$.
Note that in the regimes $(x,y)\in \widetilde{\Lambda}\times (\Gamma\setminus\widetilde{\Gamma})$ and $(x,y)\in (\Lambda\setminus\widetilde{\Lambda})\times \widetilde{\Gamma}$ one among $|x|$ and $|y|$ is bigger than $N^\rho$. Hence, expanding~\eqref{H-z cplx superbos} for large $x$ or $y$ argument, using Lemma~\ref{lem:apriorbl}, Lemma~\ref{lem:apriorx} to estimate the regime $x\in \widetilde{\Lambda}$ and $y\in \widetilde{\Gamma}$, respectively, by~\eqref{eq:expbo}--\eqref{eq:fexp}, we conclude that the integral over $(x,y)\in \widetilde{\Lambda}\times (\Gamma\setminus\widetilde{\Gamma})$ is bounded by $N^\rho(N^{1/2}+(N\delta)) e^{-N^{1-2\rho}/2}$, and that the one over $(x,y)\in (\Lambda\setminus\widetilde{\Lambda})\times \widetilde{\Gamma}$ is bounded by $N^\rho(1+|\log (NE^{2/3})|)e^{-N^{1-2\rho}/2}$.
\end{proof}
\end{lemma}

Next, we  compute  the leading term of~\eqref{H-z cplx superbos}. 
We define $\widetilde{z}_*$ as
\begin{equation}
\label{eq:newrescvar}
 \widetilde{z}_*(\lambda, \widetilde{\delta} )\defeq N^{-1/2}\big|z_*(\lambda c(N), N^{-1/2}\widetilde{\delta})\big|, \qquad \lambda = \frac{E}{c(N)}, \quad \widetilde{\delta} = N^{1/2}\delta,
\end{equation}
where we also recalled the rescaled parameters $\lambda$ and $\widetilde\delta$ from \eqref{rescale}. 
Note that $\widetilde{z}_*(\lambda,\widetilde{\delta})$ is $N$-independent, indeed all $N$ factors scale out by using the definition of $z_*(E,\delta)$ from ~\eqref{eq:impsadd}.  Since $E\ll 1$ and $\delta\in [0,1]$, by~\eqref{eq:newrescvar} it follows that the range of the new parameters is $\widetilde{\delta}\le N^{1/2}$ and $\lambda\ll c(N)^{-1}$. 

We are now ready to prove our main result on the leading term of~\eqref{H-z cplx superbos}, denoted by $q_{\widetilde{\delta}}(\lambda)$, in the complex case, Theorem~\ref{thm cplx}. Then, the one point function of $Y$ is asymptotically given by $p_{\widetilde{\delta}}(\lambda)\defeq\Im[q_{\widetilde{\delta}}(\lambda)]$. The main inputs for the proof are the bounds in~\eqref{eq:easbony} and~\eqref{eq:boundexa} that will be used to estimate the error terms in the expansions for large arguments of $f$ and $G$ in~\eqref{eq:expansionf} and~\eqref{eq:expansionG}.
\begin{proof}[Proof of Theorem~\ref{thm cplx} in the case \(\delta\ge 0\)]
By~\eqref{H-z cplx superbos} and Lemma~\ref{lem:expsmall} it follows that
\begin{equation}\label{expre}
\begin{split}
\E \Tr  [Y-w]^{-1}&=\frac{N^2}{2\pi \ii}\int_{\Lambda\setminus\widetilde{\Lambda}}\diff x\int_{\Gamma\setminus\widetilde{\Gamma}}\diff y e^{-Nf(x)+Nf(y)} y \cdot G(x,y) \\
&\quad +\mathcal{O}\left( N^{\rho}(N^{1/2}+N\delta + |\log (NE^{2/3})|)e^{-\frac{1}{2}N^{1-2\rho}}\right).
\end{split}
\end{equation}
Note that $|x|, |y|\ge N^\rho$ for any $(x,y)\in (\Lambda\setminus\widetilde{\Lambda})\times(\Gamma\setminus\widetilde{\Gamma})$. In order to prove~\eqref{eq:onepointest} we first estimate the error terms in the expansions of $f$ and $G$ in~\eqref{eq:expansionf}--\eqref{eq:expansionG} and then in order to get an $N$-independent double integral we rescale the phase function by $|z_*|$. By Lemma~\ref{lem:apriorbl} and Lemma~\ref{lem:apriorx}, using that $|\log (N|z_*|^{-2})|+|\log (N\delta |z_*|^{-1})|\lesssim |\log (NE^{2/3})|$ by the definition of $z_*$ in~\eqref{eq:impsadd}, it follows that
\begin{equation}
\label{eq:combb}
\left| \int_{\Lambda\setminus\widetilde{\Lambda}}\diff x\int_{\Gamma\setminus\widetilde{\Gamma}}\diff y \frac{e^{-Nf(x)+Nf(y)}}{x^ay^b}\right|\lesssim 
\begin{cases}
N^\frac{2-d}{2}(1\wedge \widetilde{\delta}^{1-d}) (1\vee \widetilde{\delta}) ,&  b=0, \\
N^\frac{1-b}{2} (1\wedge \widetilde{\delta}^{1-b})(1+|\log (NE^{2/3})|) ,&  a=1,\, b\ge 1, \\
N^\frac{2-d}{2}(1\wedge \widetilde{\delta}^{2-d}) ,&  a>1, \, b\ge 1,
\end{cases}
\end{equation}
for any $a\ge 1$, $b\in \N  $, where $d\defeq a+b$. In order to get the bound in the r.h.s.~of~\eqref{eq:combb} we estimated the terms with $b=0$ and $b=1$ using the improved bound in~\eqref{eq:exactyint}, all the other terms are estimated by absolute value. Note that for $\lambda\ll 1$ the bound in~\eqref{eq:combb} and the definition of $\lambda$ in \eqref{eq:newrescvar} imply that
\[
\lim_{\epsilon\to 0^+}\big|\E  \Tr   [Y-(E+\ii \epsilon)]^{-1}\big|\lesssim N^{3/2} (1\vee \widetilde{\delta})
\begin{cases}
|\log \lambda| ,&   \lambda \ge \widetilde{\delta}^3, \\
|\log \lambda \widetilde{\delta}| ,&  \lambda < \widetilde{\delta}^3.
\end{cases}
\]
if $\lambda\ll 1$, since the leading term in the expansion of $y G(x,y)$ in~\eqref{eq:expansionG} consists of monomials of the form $x^{-a}y^{-b}$, with $a+b=3$, and $\delta x^{-a}y^{-b}$, with $a+b=2$. This concludes the proof of~\eqref{eq:asymp}. 

Now we prove the more precise asymptotics~\eqref{eq:onepointest}. We will replace the functions $f$ and $G$ in~\eqref{expre} by their leading order approximations, denoted by $g$ and $H$ from~\eqref{eq:expansionf} and~\eqref{eq:expansionG}. The error of this replacement in the phase function $f$ is estimated by  the Taylor expanding the exponent $e^{\mathcal{O}^\#(x^{-3}+\delta x^{-2})}$, with $\mathcal{O}^\#(x^{-3}+\delta x^{-2})$ defined in~\eqref{eq:diffO}. Hence, by~\eqref{eq:expansionf} and ~\eqref{eq:expansionG} and the bound in~\eqref{eq:combb}, as $\epsilon \to 0^+$, we conclude that
\begin{equation}
\label{eq:newdouint}\begin{split}
\E \Tr  [Y-w]^{-1}&=\frac{N^2}{2\pi \ii}\int_{\Lambda\setminus\widetilde{\Lambda}}\diff x\int_{\Gamma\setminus\widetilde{\Gamma}}\diff y e^{-Ng(x)+Ng(y)} H(x,y)\\
&\quad+\mathcal{O}\left((N+N^{3/2}\delta)\big[1+|\log (NE^{2/3})|\big]\right).\end{split}
\end{equation}
The error estimates in~\eqref{eq:newdouint} come from terms with $d\ge 4$ or terms with $d\ge 3$ multiplied by $\delta$ in~\eqref{eq:combb}.

We recall that $\lambda=Ec(N)^{-1}$, $\widetilde{\delta}= \delta N^{1/2}$, and that $|z_*|=N^{1/2}\widetilde{z}_*(\lambda,\widetilde{\delta})$. Then, defining the contours
\begin{equation}
\label{eq:newcont}
\widehat{\Gamma}\defeq |z_*|^{-1} \Gamma, \quad \widehat{\Lambda}\defeq |z_*|^{-1}\Lambda,
\end{equation}
and using the change of variables $x\to x |z_*|$, $y\to y|z_*|$ in the leading term of~\eqref{eq:newdouint} we conclude that
\begin{equation}
\label{eq:fineq}\begin{split}
\E \Tr  [Y-w]^{-1}&=\frac{N^{3/2} \widetilde{z}_*(\lambda,\widetilde{\delta})^{-1}}{2\pi \ii}\int_{\widehat{\Gamma}}\diff y \int_{\widehat{\Lambda}} \diff x e^{h_{\lambda,\widetilde{\delta}}(y)-h_{\lambda,\widetilde{\delta}}(x)} \widetilde{H}_{\lambda,\widetilde{\delta}}(x,y)\\
&\quad+ \mathcal{O}\big(N(1\vee \widetilde{\delta}) [1+|\log \lambda|]\big)\end{split}
\end{equation}
with $h_{\lambda,\widetilde{\delta}}(x)$ and $\widetilde{H}_{\lambda,\widetilde{\delta}}(x,y)$ defined in~\eqref{eq:newfunone}. Note that in order to get~\eqref{eq:fineq} we used that the integral in the regime when either $x\in [0, N^\rho|z_*|^{-1}]$ or $y\in [-2|z_*|^{-1}/3,-2|z_*|^{-1}/3+\ii N^\rho |z_*|^{-1}]$ is exponentially small. Moreover, since by holomorphicity we can deform the contour $\widehat{\Lambda}$ to any contour, which does not cross $-1$, from $0$ to $e^\frac{3\ii \pi}{4}\infty $ and we can deform the contour $\widehat{\Gamma}$ as long as it does not cross $0$,~\eqref{eq:fineq} concludes the proof of Theorem~\ref{thm cplx}.
\end{proof}

\subsection{Case $\delta<0$, $\lvert\delta\rvert\lesssim N^{-1/2}$.}
We now explain the necessary changes in the case \(\delta<0\). All along this section we assume that $E\lesssim N^{-3/2}$. Let $x_\ast$ be the stationary point of $f$ defined in~\eqref{eq:saddlenegdelta}, that is the point around where the main contribution to~\eqref{H-z cplx superbos} comes from in the saddle point regime for $\delta<0$. Then, at leading order, $x_*$ is given by $3|\delta|^{-1}/2$ if $E\ll |\delta|^3$, by $e^\frac{\pi \ii}{3} E^{-1/3}$ if $E\gg |\delta|^3$, and by $\mu(c) e^\frac{\pi \ii}{3} E^{-1/3}$ if $E= c|\delta|^3$, for some function $\mu(c)>0$ for any fixed constant $c>0$ independent of $N, E$ and $\delta$.

This regime can be treated similarly to the regime $0\le \delta\le 1$, since for $|\delta|\lesssim N^{-1/2}$ the term $\delta x^{-1}$ in the expansion of $f$, for $|x|\gg 1$, does not play any role in the bounds of Lemma~\ref{lem:apriorbl}, Lemma~\ref{lem:apriorx}. Indeed, instead of the deforming the contours $\Gamma$ and $\Lambda$ through the leading term of the stationary point $x_*$, we deform $\Gamma$ and $\Lambda$ through $z_*\defeq e^\frac{\pi \ii}{3} E^{-1/3}$ as $\Gamma=\Gamma_{z_*}\defeq \Gamma_{1,z_*}\cup\Gamma_{2,z_*}$ and $\Lambda=\Lambda_{z_*}\defeq \Lambda_{1,z_*}\cup\Lambda_{2,z_*}$, where $\Gamma_{1,z_*}, \Gamma_{2,z_*}$ and $\Lambda_{1,z_*}, \Lambda_{2,z_*}$ are defined in~\eqref{eq:defycon} and~\eqref{eq:defxcon}, respectively. We could have done the same choice in the case $0\le \delta\lesssim N^{-1/2}$, but not for the regime $N^{-1/2}\ll \delta\le 1$, hence, to treat both the regimes in the same way, in Section~\ref{sec:posdec} we deformed the contours trough \eqref{eq:impsadd}. Note that $z_*$ defined here is not the analogue of~\eqref{eq:impsadd}, since in all cases $z_*= e^\frac{\pi \ii}{3}E^{-1/3}$. The fact that $E\ll 1$ implies that $|z_*|\gg 1$, hence, similarly to the case $0\le \delta \le 1$, we expect that the main contribution to~\eqref{H-z cplx superbos} comes from the regime when $|x|, |y|\ge N^\rho$, for some small $0< \rho<1/2$. Hence, in order to compute the leading term of~\eqref{H-z cplx superbos} we expand $f$ and $G$ for large arguments as in~\eqref{eq:expansionf} and~\eqref{eq:expansionG}.

The phase function $f$ defined in~\eqref{H-z cplx superbos} satisfies the properties \ref{it:reimf}, \ref{it:incf} and \ref{it:par2} of Lemma~\ref{lem:propf}, but \ref{it:dec} does not hold true for $\delta<0$ if $E\ll |\delta|^3$. Instead, it is easy to see that the following lemma holds true.

\begin{lemma}\label{lem:negdeltafprop}
Let $f$ be the phase function defined in~\eqref{H-z cplx superbos}, then, as $\epsilon\to 0^+$, the function $x\mapsto \Re[f(x)]$ has a unique global minimum on $\Lambda_{1,z_*}$ at $x=3|\delta|^{-1}/2$ if $E\ll |\delta|^3$.
\end{lemma}

Note that, since $|\delta|\lesssim N^{-1/2}$, by Lemma~\ref{lem:negdeltafprop} it follows that the function $x\mapsto \Re[f(x)]$ is strictly decreasing for $0\le x\ll N^{-1/2}$. 

\begin{proof}[Proof of Theorem~\ref{thm cplx} for \(-C N^{-1/2}\le \delta<0\)]
Let $\widetilde{\Gamma}$, $\widetilde{\Lambda}$ be defined in~\eqref{eq:smallcont}, then using that
\[
e^{-N\left[ \frac{\delta}{s}+\frac{1}{2s^2}\right]}\lesssim e^{-\frac{N}{4s^2}},
\]
for $s\in [N^\rho, |z_*|]$ and $|\delta|\lesssim N^{-1/2}$, Lemma~\ref{lem:apriorx} and the improved bounds in~\eqref{eq:exactyint} of Lemma~\ref{lem:apriorbl}, for $b=0,1$, we conclude the bound in the following lemma exactly as in~\eqref{eq:combb} without the improvement involving $\widetilde{\delta}$.

\begin{lemma}
\label{lem:finboundnegdelta}
Let $E\lesssim N^{-3/2}$, and let $f$ be defined in~\eqref{H-z cplx superbos}.Then, the following bound holds true
\[
\left| \int_{\Lambda\setminus\widetilde{\Lambda}}\diff x\int_{\Gamma\setminus\widetilde{\Gamma}}\diff y \frac{e^{-Nf(x)+Nf(y)}}{x^ay^b}\right|\lesssim 
\begin{cases}
N^\frac{2-d}{2} ,&  b=0, \\
N^\frac{1-b}{2} (1+|\log (NE^{2/3})|) ,&  a=1,\, b\ge 1, \\
N^\frac{2-d}{2} ,&  a>1, \, b\ge 1,
\end{cases}
\]
for any $a\ge 1$ and $b\in \N  $, where $d\defeq a+b$.
\end{lemma}

Then, similarly to the case $0\le \delta\le 1$, by Lemma~\ref{lem:finboundnegdelta} we conclude that the contribution  to~\eqref{H-z cplx superbos} of the regime when either $x\in\widetilde{\Lambda}$ or $y\in \widetilde{\Gamma}$ is exponentially small, i.e.~Lemma~\ref{lem:expsmall} holds true. Hence, by~\eqref{H-z cplx superbos}, Lemma~\ref{lem:finboundnegdelta} and the expansion of $G$ in~\eqref{eq:expansionG}, we easily conclude Theorem~\ref{thm cplx} also in the regime \(-CN^{-1/2}\le \delta<0\). 
\end{proof}

\section{The real case below the saddle point regime}\label{sec real}
In this section we prove Theorem~\ref{thm real}. Throughout this section we always assume that $\Re w<0$, hence to make our notation easier we define $w= -E+\ii \epsilon$
with some $E>0$ and $\epsilon>0$.  Moreover, we always assume that $E\lesssim c(N)$, with $c(N)$ defined in~\eqref{eq:thres}. We are interested in estimating~\eqref{realsusyexplAAr} in the transitional regime of $|z|$ around one. For this purpose we introduce the parameter $\delta=\delta_z\defeq 1-|z|^2$. In order to have an optimal estimate of the leading order term of~\eqref{realsusyexplAAr} it is not affordable to estimate the error terms in the expansions, for large $a$ and $\xi$, of $f$, $g(\cdot,1,\eta)$ and $G_{1,N}, G_{2,N}$ by absolute value. For this reason, to compute the error terms in the expansions of $f$, $g(\cdot,0,\eta),G_{1,N}, G_{2,N}$ we use a notation $\mathcal{O}^\#(\cdot)$ similar to the one introduced in~\eqref{eq:diffO}. In order to keep track of the power of $\tau$ in the expansion of $G_{1,N}$ and $G_{2,N}$, we define the set of functions
\[
\begin{split}
&\mathcal{O}^\#\big( (a,\tau,\xi)^{-1}\big)\defeq \Bigg\{h:\C \times[0,1]\times \C \to\C : \\
&\qquad h(a,\tau,\xi)=\sum_{\alpha+\beta\ge 1,\atop 0\le\gamma\le \alpha}\frac{c_{\alpha,\beta,\gamma}}{a^\alpha \tau^\gamma \xi^\beta}\, \text{with},\, |c_{\alpha,\beta,\gamma}|\le C^{\alpha+\beta},\, \text{if}\, |a|, |a\tau|,|\xi|\ge 2C\Bigg\},
\end{split}
\]
for some constant $C>0$ implicit in the $\mathcal{O}^\#(\cdot)$ notation. The exponents $\alpha, \beta$ are non negative integers.

We expand the functions $f,g(\cdot,1,\eta),G_{1,N},G_{2,N}$ for large $a$ and $\xi$ arguments as
\begin{equation}
\label{expfxi1AAr}
\begin{split}
f(\xi)&=(E-\ii \epsilon) \xi+\frac{\delta}{\xi}+\frac{1}{2\xi^2} +\mathcal{O}^\#\big(\xi^{-3}+\delta\xi^{-2}\big),\\
g(a,1,\eta)&=(E-\ii \epsilon) a+\frac{\delta}{a}+\frac{1}{2a^2} +\mathcal{O}\big(a^{-3}+\delta a^{-2}\big),
\end{split}
\end{equation}
 with $\mathcal{O}^\#(\cdot)$ defined~\eqref{eq:diffO}, and
\begin{align}
G_{1,N}(a,\tau,\xi,|z|)&= \Bigg[\sum_{\substack{a,\beta\ge 2, \, \alpha+\beta=8,\\\gamma=\min\{\alpha-1,3\}}} \frac{c_{\alpha,\beta,\gamma}N^2}{a^\alpha \tau^\gamma  \xi^{\beta}}+\sum_{\substack{\alpha,\beta\ge 2, \, \alpha+\beta=7, \\ \gamma=\min\{\alpha-1,3\}}}\frac{c_{\alpha,\beta,\gamma}N^2\delta}{a^\alpha \tau^\gamma\xi^\beta}+\sum_{\substack{a,\beta\ge 2,\, \alpha+\beta=6,\\ \gamma=\min\{\alpha-1,2\}}} \frac{c_{\alpha,\beta,\gamma}N}{a^\alpha \tau^\gamma  \xi^{\beta}}  \nonumber\\\label{expfxi2AAr}
&\quad+\sum_{\alpha, \beta\ge 2, \alpha+\beta=6 \atop \gamma=\min\{\alpha-1,2\}}\frac{c_{\alpha,\beta,\gamma} N^2\delta^2}{a^\alpha\tau^\gamma\xi^\beta}\Bigg]\times\big[1+\mathcal{O}^\#\big((a,\tau,\xi)^{-1}\big)\big], \\
G_{2,N}(a,\tau,\xi,z)&=\Bigg[\sum_{\substack{\alpha,\beta\ge 2, \, \alpha+\beta=6, \\ \gamma=\max\{ \alpha-1,2\}}} \frac{c_{\alpha,\beta,\gamma} N^2\eta^2}{a^\alpha\tau^\gamma\xi^{\beta}} +\sum_{\substack{\alpha,\beta\ge 2, \alpha+\beta=5, \\ \gamma=\max\{\alpha-1,2\}}}\frac{c_{\alpha,\beta,\gamma} N^2\eta^2\delta}{a^\alpha\tau^\gamma\xi^\beta} \nonumber\\ \label{newexpansionA1Ar}
&\quad + \sum_{\substack{\alpha,\beta=2,3,\, \alpha+\beta=5 \\ \gamma=\max\{\alpha-1,2\}}} \frac{c_{\alpha,\beta,\gamma} N\eta^2}{a^\alpha\tau^\gamma\xi^{\beta}} \Bigg]\times\big[1+\mathcal{O}^\#\big( (a,\tau,\xi)^{-1}\big)\big],
\end{align}
where $c_{\alpha,\beta,\gamma}\in \mathbb{R}$ is a constant that may change term by term. To make our notation easier in \eqref{expfxi2AAr}-\eqref{newexpansionA1Ar} we used the convention to write a common multiplicative error for all the terms, even if in principle the constants in the series expansion of the error terms differ term by term.

In the following we deform the integration contours in~\eqref{realsusyexplAAr} with the following constraints: the $\xi$ contour can be freely deformed as long as it does not cross $0$ and $-1$, the $a$-contour can be deformed as long as it goes out from zero in the region $\Re[a]>|\Im[a]|$, it ends in the region $\Re[a]>0$, and it does not cross $0$ and $-1$ along the deformation. The $\tau$-contour will not be deformed.

Specifically, we deform the $a$-contour in~\eqref{realsusyexplAAr} to $\Lambda=[0,+\infty)$, and we can deform the $\xi$-contour to any contour around $0$ not encircling $-1$ (this contour is denoted by $\Gamma$ in \eqref{eq:startrepoir}). Moreover, since for $a\in\R _+$ we have $|e^{-N(E-\ii\epsilon)a}|=e^{-NEa}$ and since the factor $e^{-NEa}$ makes the integral convergent, we may pass to the limit $\epsilon\to 0^+$. Hence, for any $E>0$ we conclude that 
\begin{equation}
\label{eq:startrepoir}
\E \Tr   [Y+E]^{-1}=\frac{N}{4\pi \ii} \int_\Gamma \diff \xi \int_0^{+\infty} \diff a\int_0^1 \diff s \frac{\xi^2a}{\tau^{1/2}} e^{N[f(\xi)-g(a,\tau,\eta)]} G_N(a,\tau,\xi,z).
\end{equation}

We split the computation of the leading order term of~\eqref{eq:startrepoir} into the cases $-CN^{-1/2}\le \delta<0$ and $\delta\ge 0$.

\subsection{Case $0\le \delta\le 1$.}
\label{sec:posredelta}

In order to estimate the leading term of~\eqref{eq:startrepoir} we compute the $\xi$-integral and the $(a,\tau)$-integral separately. In particular, we compute the  $(a,\tau)$-integral firstly performing the $\tau$-integral for any fixed $a$ and then we compute the $a$-integral. Note that, since $E'=-E$ is negative, the relevant stationary point of $f(\xi)$ is real and its leading order $\xi_*$ is given by
\begin{equation}
\label{eq:newstatpoi}
\xi_*\defeq \begin{cases}
\sqrt{\frac{\delta}{E}} ,&  E\ll \delta^3, \\
\mu(c)E^{-1/3} ,&  E=c\delta^3, \\
E^{-1/3} ,&  E\gg \delta^3,
\end{cases}
\end{equation}
for some $\mu(c)>0$ and any fixed constant $c>0$ independent of $N$, $E$ and $\delta$. Note that $\xi_* \gg 1$ for any $E\lesssim c(N)$, $0\le \delta\le 1$. We will show that in the regime $a\in [N^\rho,+\infty)$ the $\tau$-integral is concentrated around $1$ as long as $|e^{-N[g(a,\tau,\eta)-g(a,1,\eta)]}|$ is effective, i.e.~as long as $N|g(a,\tau,\eta)-g(a,1,\eta)|\gg 1$, and that it is concentrated around $0$ if $N|g(a,\tau,\eta)-g(a,1,\eta)|\lesssim 1$. For this purpose, using that $g(a,1,\eta)=f(a)$ for any $a\in\C $, we rewrite~\eqref{eq:startrepoir} as
\begin{equation}
\label{eq:firstchr}
\E \Tr   [Y+E]^{-1}=\frac{N}{4\pi \ii} \int_\Gamma \diff \xi e^{Nf(\xi)} \xi^2 \int_0^{+\infty} \diff a e^{-Nf(a)} a\int_0^1 \diff \tau \frac{e^{-N[g(a,\tau,\eta)-g(a,1,\eta)]}}{\tau^{1/2}}  G_N,
\end{equation}
where we used that by holomorphicity, we can deform the contour $\Gamma$ as $\Gamma=\Gamma_{\xi_*}\defeq\Gamma_{1,\xi_*}\cup \Gamma_{2,\xi_*}$, with $\Gamma_{1,\xi_*}, \Gamma_{2,\xi_*}$ defined in~\eqref{eq:defycon} replacing $z_*$ by $\xi_*$. We will show that the contribution to~\eqref{eq:firstchr} of the integrals in the regime when either $|a|\le N^\rho$ or $|\xi|\le N^\rho$, for some small fixed $0<\rho<1/2$, is exponentially small. Moreover, we will show that also the $\tau$-integral is exponentially small for $\tau$ very close to $0$ because of the term $\log \tau$ in the phase function $g(a,\tau,\eta)$. Hence, we define $\widetilde{\Gamma}$, $\widetilde{\Lambda}$ as in~\eqref{eq:smallcont}, and $I\subset [0,1]$ as $I=I_a\defeq [0, N^{\rho/2} a^{-1} ]$, for any $a\in [0, +\infty)$.

In order to compute the leading term of~\eqref{eq:firstchr} we first bound the integral in the regime $(a,\tau,\xi)\in  (\Lambda\setminus\widetilde{\Lambda})\times ([0,1]\setminus I)\times (\Gamma\setminus\widetilde{\Gamma})$, with $\widetilde{\Lambda}$ and $\widetilde{\Gamma}$ defined in~\eqref{eq:smallcont}, that is the regime where we expect that the main contribution comes from, and then we use these bounds to firstly prove that the integral in the regime when either $|\xi|\le N^\rho$ or $|a|\le N^\rho$ is exponentially small for any $\tau\in [0,1]$, and then prove that also the $\tau$-integral on $I$ is exponentially small if $|a|\ge N^\rho$. The bounds for the $\xi$-integral over $\Gamma\setminus \widetilde{\Gamma}$ are exactly the same as Lemma~\ref{lem:apriorbl}, since the phase function $f(\xi)$ and the $\Gamma$-contour are exactly the same as the complex case. In order to estimate the integral over $(a,\tau)\in (\Lambda\setminus\widetilde{\Lambda})\times ([0,1]\setminus I)$, we start with the estimate of the $\tau$-integral over $[0,1]\setminus I$ in Lemma~\ref{lem:tauintr} and then we will conclude the computation of the $a$-integral over $[N^\rho,+\infty)$ in Lemma~\ref{lem:doubleatauintr}.

Before proceeding with the bounds for large $|a|, |\xi|$, in the following lemma we state some properties of the functions $f$ and $g$. The proof of this lemma follows by elementary computations. From now on, for simplicity, we assume that $\eta\ge 0$; the case $\eta<0$ is completely analogous since the functions $g$ and $G_N$ in \eqref{eq:startrepoir} depends only on $\eta^2$ and $|z|^2$.

\begin{lemma}
\label{lem:propfgrealr}
Let $f$ and $g$ be the phase functions defined in~\eqref{eq:fffr} and~\eqref{bosonphfAr}, respectively, then the following properties hold true:
\begin{enumerate}[(i)]

\item \label{it:reimfrealr} For any $\xi=-2/3+\ii t\in \Gamma_{1,\xi_*}$, we have that
\begin{equation}
\label{eq:refrealr}
\Re[f(-2/3+\ii t)]=\frac{2E}{3}+\epsilon t-\frac{1}{2t^2}+\mathcal{O}(|t|^{-3}+\delta|t|^{-2}),
\end{equation}
and
\begin{equation}
\label{eq:imfrealr}
\Im[f(-2/3+\ii t)]=\frac{2\epsilon}{3}-Et-\frac{\delta}{t}+\mathcal{O}(|t|^{-3}).
\end{equation}

\item \label{it:incfrealr} For $\epsilon=0$, the function $t\mapsto \Re[f(-2/3+\ii t)]$ on $\Gamma_{1,\xi_*}$ is strictly increasing if $t>0$ and strictly decreasing if $t<0$.

\item \label{it:taudecr} For any $a\in [0,+\infty)$, we have that $g(a,\tau,\eta)\ge g(a,\tau,0)$ and the function $\tau \mapsto g(a,\tau,0)$ is strictly decreasing on $[0,1]$.

\item \label{it:decr} The function $a\mapsto g(a,1,\eta)$ is strictly decreasing on $[0,\xi_*/2]$.

\end{enumerate}
\end{lemma}

\begin{lemma}
\label{lem:tauintr}
Let $\rho>0$ be sufficiently small, $I=I_a=[0, N^{\rho/2} a^{-1} ]$, $\gamma\in \N  $, $\gamma\ge 1$, $c(N)$ be defined in~\eqref{eq:thres}, $E\lesssim c(N)$,  $0\le \delta\le 1$ and let $g$ be defined in~\eqref{bosonphfAr}. Then, for any $a\in [N^\rho,+\infty)$, we have
\begin{equation}
\label{eq:impbtau2r}
\begin{split}
&\int_{[0,1]\setminus I_a}  \frac{\left|e^{-N[g(a,\tau,\eta)-g(a,1,\eta)]}\right|}{\tau^{\gamma+1/2}} \diff \tau \lesssim F(a)\defeq \begin{cases}
a^2N^{-1} \wedge 1, & N^\rho \le a\le \delta^{-1}\wedge \eta^{-1}, \\
a (N\delta)^{-1} \wedge 1, &\delta^{-1}\vee N^\rho\le a\le\delta \eta^{-2}, \\
(N\eta^2)^{-1}\wedge 1, &a \ge (\delta\eta^{-1}\vee 1)\eta^{-1}, 
\end{cases} \\
&\qquad\qquad+ e^{-\frac{1}{2}N\eta^2}\times
\begin{cases}
e^{-\frac{N(\delta\vee N^{-1/2})}{a}} ,   &   N^\rho\le a\le N(\delta\vee N^{-1/2}), \\
(a(N\delta \vee \sqrt{N})^{-1})^{\gamma-1/2}, &  N(\delta\vee N^{-1/2})\le a\le (\delta\vee N^{-1/2})\eta^{-2},\\
(N\eta^2)^{1/2-\gamma},  &    a\ge  (\delta\vee N^{-1/2})\eta^{-2},
\end{cases}
\end{split}
\end{equation}
where some regimes in~\eqref{eq:impbtau2r} might be empty for certain values of $\delta$ and $\eta$.
\begin{proof}
In order to estimate the integral in the l.h.s. of~\eqref{eq:impbtau2r}  we first compute the expansion
\begin{equation}
\label{eq:explargataur}
\begin{split}
g(a,\tau,\eta)-g(a,1,\eta)&= \frac{(1-2\delta)(1-\tau)-(1-\tau)^2}{a^2 \tau^2}+\frac{2\eta^2(1-\tau)}{\tau} + \frac{(1-\tau) \delta}{a\tau} \\
&\quad+\mathcal{O}\left(\frac{1-\tau}{a^2 \tau}+\frac{(1-\tau)(\delta+a^{-1})}{a^2\tau^2} +\frac{\eta^2(1-\tau)}{a\tau^2},\right),
\end{split}
\end{equation}
which holds true for any $\tau\in [0,1]\setminus I=[N^{\rho/2}a^{-1},1]$. Note that by~\eqref{eq:explargataur} it follows that for any $(a,\tau)\in [N^\rho,+\infty)\times [N^{\rho/2}a^{-1},1]$ it holds
\[
g(a,\tau,\eta)-g(a,1,\eta)\ge \frac{1-\tau}{2}\left[ \frac{1}{a^2\tau^2}+\frac{\delta}{a\tau}+\frac{\eta^2}{\tau}\right].
\]
Then, by~\eqref{eq:explargataur} it follows that
\begin{equation}
\label{eq:firstapproxtaur}
\int_{[0,1]\setminus I}  \frac{\left|e^{-N[g(a,\tau,\eta)-g(a,1,\eta)]}\right|}{\tau^{\gamma+1/2}} \diff \tau\lesssim  \int_{N^{\rho/2} a^{-1}}^1 \frac{e^{-(1-\tau)\frac{N}{2} \left[ \frac{1}{a^2\tau^2}+\frac{\delta}{a\tau}+\frac{\eta^2}{\tau}\right]}}{\tau^{\gamma+1/2}}\, \diff \tau.
\end{equation}
In order to bound the r.h.s.~of~\eqref{eq:firstapproxtaur} we split the computations into two cases: $\delta\le N^{-1/2}$ and $\delta > N^{-1/2}$. We firstly consider the case $\delta> N^{-1/2}$. In order to prove the bound in the r.h.s.~of~\eqref{eq:impbtau2r} we further split the computation of the $\tau$-integral into the regimes $\tau\in [N^{\rho/2}a^{-1},1/2]$ and $\tau\in [1/2,1]$. We start estimating the integral over $[1/2,1]$ as follows
\begin{equation}
\label{eq:firsttaub1r}
\begin{split}
\int_{1/2}^1 \frac{e^{-(1-\tau)\frac{N}{2}\left[ \frac{1}{a^2\tau^2}+\frac{\delta}{a\tau}+\frac{\eta^2}{\tau}\right]}}{\tau^{\gamma+1/2}}\, \diff \tau&\lesssim \int_{1/2}^1 e^{-(1-\tau)\frac{N}{2} \left[ \frac{1}{a^2}+\frac{\delta}{a}+\eta^2\right]}\, \diff \tau \\
&\lesssim
\begin{cases}
a^2N^{-1} \wedge 1, & N^\rho \le a \le \delta^{-1}\wedge \eta^{-1}, \\
a (N\delta)^{-1} \wedge 1 ,&  N^\rho \vee \delta^{-1}\le a\le\delta \eta^{-2}, \\
(N\eta^2)^{-1} \wedge 1 ,&  a \ge \delta\eta^{-2}\vee \eta^{-1}.
\end{cases}
\end{split}
\end{equation}
For the integral over $\tau\in [N^{\rho/2}a^{-1},1/2]$, instead, we bound the r.h.s.~of~\eqref{eq:firstapproxtaur} as
\begin{equation}
\label{eq:firsttaub2r}
\int_{N^{\rho/2} a^{-1}}^{1/2} \frac{e^{-\frac{N}{2} \left[\frac{\delta}{a\tau}+\frac{\eta^2}{\tau}\right]}}{\tau^{\gamma+1/2}}\, \diff \tau\lesssim e^{-\frac{1}{2}N\eta^2}\times
\begin{cases}
e^{-\frac{1}{2}N\delta a^{-1}}(a(N\delta)^{-1})^{\gamma-1/2} ,&   N^\rho \le a\le \delta\eta^{-2},\\
(N\eta^2)^{1/2-\gamma} ,&   a\ge \delta\eta^{-2}.
\end{cases}
\end{equation}
Then, combining~\eqref{eq:firsttaub1r}--\eqref{eq:firsttaub2r} we conclude the bound in \eqref{eq:impbtau2r} for $\delta>N^{-1/2}$. 
Using similar computations for $ \delta\le N^{-1/2}$, we conclude the bound in~\eqref{eq:impbtau2r}.
\end{proof}
\end{lemma}

In the following lemma we conclude the bound for the double integral~\eqref{eq:firstchr} in the regime $(a,\tau)\in (
[N^\rho,+\infty)\times ([0,1]\setminus I)$ using the bound in~\eqref{eq:impbtau2r} as an input.

\begin{lemma}
\label{lem:doubleatauintr}
Let $\rho>0$ be sufficiently small, $I=[0, N^{\rho/2}a^{-1}]$, $0\le \delta\le 1$, let $c(N)$ be defined in~\eqref{eq:thres}, $E\lesssim c(N)$, and let $g$ be defined in~\eqref{bosonphfAr}. Then, for any integers $\alpha \ge 2$, $1\le \gamma\le \alpha$,  we have
\begin{equation}
\label{eq:boddoubintr}
\int_{\Lambda\setminus \widetilde{\Lambda}} \int_{[0,1]\setminus I}^1 \abs{ \frac{e^{-Ng(a,\tau,\eta)}}{a^{\alpha-1} \tau^{\gamma+1/2}}  }  \diff \tau \diff a \lesssim C_1 +  e^{-\frac{1}{2}N\eta^2}(N\eta^2)^{1/2-\gamma} C_2 + e^{-\frac{1}{2}N\eta^2}(N\delta\vee \sqrt{N})^{1/2-\gamma} C_3, 
\end{equation}
where 
\[\begin{split}
&C_1\defeq \begin{cases}
1+\big|\log[N(\delta\vee N^{-1/2})\xi_*^{-1}]\big| ,& \alpha=2, \, [\delta\eta^{-1}\vee 1]\eta^{-1}> \xi_*, \\
((N\eta^2)^{-1}\wedge 1)\big(1+\big|\log[N(\delta\vee N^{-1/2})\xi_*^{-1}]\big|\big), & \alpha=2,\, [\delta\eta^{-1}\vee 1]\eta^{-1}\le \xi_*, \\
[N(\delta\vee N^{-1/2})]^{2-\alpha}, & \alpha \ge 3,  \, [\delta\eta^{-1}\vee 1]\eta^{-1}> \xi_* \\
((N\eta^2)^{-1}\wedge 1) [N(\delta\vee N^{-1/2})]^{2-\alpha}, & \alpha \ge 3,  \, [\delta\eta^{-1}\vee 1]\eta^{-1}\le \xi_*
\end{cases} \\
&C_2\defeq
\begin{cases}
\big(1+\big|\log[N(\delta\vee N^{-1/2}) \xi_*^{-1}]\big|\big) e^{-\frac{NE}{2\eta^2}(\delta\vee N^{-1/2})},&  \alpha=2,\\
[(\delta^{-1}\wedge \sqrt{N})\eta^2]^{\alpha-2} e^{-\frac{NE}{2\eta^2}(\delta\vee N^{-1/2})}, &  \alpha \ge 3, \, (\delta^{-1}\wedge \sqrt{N})\eta^2\le NE,\\
[(\delta^{-1}\wedge \sqrt{N})\eta^2]^{\alpha-2} , &  \alpha\ge 3, \, (\delta^{-1}\wedge \sqrt{N})\eta^2\ge NE, \\
\end{cases} \\
&C_3\defeq 
\begin{cases}
(NE)^{\alpha-\gamma-3/2} ,&  \gamma=\alpha,\alpha-1, \,  (\delta^{-1}\wedge \sqrt{N})\eta^2\le NE, \\
[(\delta^{-1}\wedge N^{1/2})\eta^2]^{\alpha-\gamma-3/2} ,&  \gamma=\alpha,\alpha-1, \, (\delta^{-1}\wedge \sqrt{N})\eta^2\ge NE, \\
(N\delta\vee \sqrt{N})^{3/2+\gamma-\alpha} ,&  \gamma\le \alpha-2.
\end{cases}
\end{split}
\]
\end{lemma}
\begin{proof}
Firstly, we add and subtract $Ng(a,1,\eta)=Nf(a)$ to the phase function in the exponent and conclude, by Lemma~\ref{lem:tauintr}, that
\begin{equation}
\label{eq:aintr}
\int_{N^\rho}^{+\infty} \int_{[0,1]\setminus I}^1 \left| \frac{e^{-Ng(a,\tau,\eta)}}{a^{\alpha-1} \tau^{\gamma+1/2}}  \right| \, \diff \tau \diff a \lesssim  \int_{N^\rho}^{+\infty} \frac{\left|e^{-Ng(a,1,\eta)}\right|\cdot F(a)}{a^{\alpha-1}} \, \diff a,
\end{equation}
with $F(a)$ defined in~\eqref{eq:impbtau2r}. In the following of the proof we often use that $NE\xi_*\lesssim 1$ by the definition of $\xi_*$ in~\eqref{eq:newstatpoi}, that implies $e^{NE\xi_*}\lesssim 1$. We split the computation of the integral in the r.h.s.~of~\eqref{eq:aintr} as the sum of the integrals over $[N^\rho,\xi_*]$ and $[\xi_*,+\infty)$. From now on we consider only the case $\delta> N^{-1/2}$, since the case $\delta\le N^{-1/2}$ is completely analogous. In the regime $\delta> N^{-1/2}$ we have $E\lesssim c(N)=\delta^{-1}N^{-2}\lesssim \delta^3$, therefore $\xi_*\sim\sqrt{\delta E^{-1}}$ from~\eqref{eq:newstatpoi}.

Then, using the expansion for large $a$-argument of $g(a,1,\eta)=f(a)$ in~\eqref{expfxi1AAr}, we start estimating the integral over $[N^\rho, +\infty)$ as follows
\begin{equation}
\label{eq:lema4r}
\begin{split}
& e^{NE\xi_*}\int_{N^\rho}^{\xi_*}  \frac{e^{-N\left[ \frac{\delta}{a}+\frac{1}{2a^2}\right]}F(a)}{a^{\alpha-1}} \, \diff a \\
&\lesssim
\chi(\delta\eta^{-2}\le \xi_*) e^{-\frac{1}{2}N\eta^2}(N\eta^2)^{1/2-\gamma}\times \begin{cases}
1+|\log(N\delta \xi_*^{-1})| ,& \alpha=2, \\
(\delta^{-1}\eta^2)^{\alpha-2}, & \alpha\ge 3.
\end{cases} \\
&\qquad\quad + e^{-\frac{1}{2}N\eta^2} \times
\begin{cases}
N^{1/2-\gamma}\delta^{2-\alpha} \eta^{2\alpha-2\gamma-3}, & \gamma=\alpha, \alpha-1, \delta\eta^{-2}\le \xi_* \\
(N\delta)^{1/2-\gamma}\xi_*^{3/2+\gamma-\alpha}, & \gamma=\alpha, \alpha-1, \delta\eta^{-2}\ge \xi_* \\
(N\delta)^{2-\alpha} ,& \gamma\le \alpha-2,
\end{cases} \\
&\qquad \quad +\begin{cases}
1+|\log(N\delta\xi_*^{-1})| ,& \alpha=2, \, [\delta\eta^{-1}\vee 1]\eta^{-1}> \xi_*, \\
((N\eta^2)^{-1}\wedge 1)(1+|\log(N\delta \xi_*^{-1})|) ,& \alpha=2,\, [\delta\eta^{-1}\vee 1]\eta^{-1}\le \xi_*, \\
(N\delta)^{2-\alpha}, & \alpha \ge 3 \, [\delta\eta^{-1}\vee 1]\eta^{-1}> \xi_*, \\
((N\eta^2)^{-1}\wedge 1) (N\delta)^{2-\alpha} & \alpha \ge 3 \, [\delta\eta^{-1}\vee 1]\eta^{-1}\le \xi_*,
\end{cases}
\end{split}
\end{equation}
To conclude the proof we are left with the estimate of the $a$-integral on $[\xi_*,+\infty)$ . In this regime we bound the r.h.s.~of~\eqref{eq:aintr} as follows
\begin{equation}
\label{eq:lema6r}
\begin{split}
&e^{NE\xi_*}\int_{\xi_*}^{+\infty} \frac{e^{-NE a}F(a)}{a^{\alpha-1}}\diff a \\
&\lesssim
e^{-\frac{1}{2}N\eta^2} (N\eta^2)^{1/2-\gamma}\times
\begin{cases}
\left(1+|\log(NE\xi_*)|\right) e^{-NE\delta\eta^{-2}/2}, & \alpha=2, \\
e^{-NE\delta\eta^{-2}} (\delta^{-1}\eta^2)^{\alpha-2}, & \alpha\ge 3, \, \delta^{-1}\eta^2\le NE\wedge \xi_*^{-1},\\
(\delta^{-1}\eta^2)^{\alpha-2}, &  \alpha\ge 3, \, NE\le \delta^{-1}\eta^2\le \xi_*^{-1}, \\
\xi_*^{2-\alpha}, & \alpha\ge 3, \, \delta^{-1}\eta^2\ge \xi_*^{-1},
\end{cases} \\
&\quad +
\chi(\xi_*\le\delta\eta^{-2}) e^{-\frac{1}{2}N\eta^2}(N\delta)^{1/2-\gamma}\times
\begin{cases}
(NE)^{\alpha-\gamma-3/2}, &\gamma=\alpha,\alpha-1, \,  \delta^{-1}\eta^2\le NE, \\
(\delta^{-1}\eta^2)^{\alpha-\gamma-3/2}, & \gamma=\alpha,\alpha-1, \, \delta^{-1}\eta^2\ge NE, \\
\xi_*^{3/2+\gamma-\alpha}, & \gamma\le \alpha-2,
\end{cases} \\
&\quad  + 
\begin{cases}
1+|\log(N\delta\xi_*^{-1})| ,& \alpha=2, \, [\delta\eta^{-1}\vee 1]\eta^{-1}> \xi_*, \\
((N\eta^2)^{-1}\wedge 1)(1+|\log(N\delta \xi_*^{-1})|) ,& \alpha=2,\, [\delta\eta^{-1}\vee 1]\eta^{-1}\le \xi_*, \\
(\xi_*)^{2-\alpha}, &\alpha \ge 3\, [\delta\eta^{-1}\vee 1]\eta^{-1}> \xi_*, \\
((N\eta^2)^{-1}\wedge 1) (\xi_*)^{2-\alpha}, &\alpha \ge 3\, [\delta\eta^{-1}\vee 1]\eta^{-1}> \xi_*.
\end{cases}
\end{split}
\end{equation}
Finally, combining~\eqref{eq:lema4r} and~\eqref{eq:lema6r} we conclude the bound in~\eqref{eq:boddoubintr}.
\end{proof}

In order to conclude the estimate of the leading order term of~\eqref{eq:firstchr}, in the following lemma, using the bounds in Lemma~\ref{lem:apriorbl} for the $\xi$-integral and the ones in Lemma~\ref{lem:doubleatauintr} for the $(a,\tau)$-integral, we prove that the contribution to~\eqref{eq:firstchr} in the regime when either $a\in [0,N^\rho]$ or $\xi\in\widetilde{\Gamma}$ and in the regime $\tau\in I$ is exponentially small.

\begin{lemma}
\label{lem:expsmallrealr}
Let $c(N)$ be defined in~\eqref{eq:thres}, $0\le \delta \le 1$, $I=I_a=[0,N^{\rho/2}a^{-1}]$,
and let $f$, $g$ and $G_N$ be defined in~\eqref{eq:fffr}--\eqref{eq:newbetG}, then, for any $E\lesssim c(N)$, we have that
\begin{equation}
\label{eq:expsmalltripintr}
\begin{split}
&\left|\left(\int_\Gamma\, \diff \xi\int_\Lambda\, \diff a \int_0^1\diff \tau-\int_{\Gamma\setminus\widetilde{\Gamma}}\, \diff \xi\int_{\Lambda\setminus\widetilde{\Lambda}}\, \diff a \int_{[0,1]\setminus I}\diff \tau\right)\left[ e^{N[f(\xi)-g(a,\tau,\eta)]} \frac{a\xi^2}{\tau ^{1/2}} G_N(a,\tau,\xi,z)\right]\right| \\
&\qquad \lesssim N^{5/2+\rho}(N^{1/2}+N\delta) e^{-\frac{1}{2}N^{1-2\rho}}\left(\frac{e^{-N\eta^2/2}}{\sqrt{E}}+(N^{1/2}+N\delta)[1+|\log(NE^{2/3})|]\right).
\end{split}
\end{equation}
\begin{proof}

We split the proof into three parts, we first prove that the contribution to~\eqref{eq:firstchr} in the regime $a\in\widetilde{\Lambda}=[0,N^\rho]$ is exponentially small uniformly in $\tau\in [0,1]$ and $\xi\in \Gamma$, then we prove that for $a \ge N^\rho$ the contribution to~\eqref{eq:firstchr} in the regime $\tau\in I$ is exponentially small uniformly in $\xi\in \Gamma$, and finally we conclude that also the contribution for $\xi\in\widetilde{\Gamma}$ is negligible.

Note that for any $a\in [0,+\infty)$, $\tau\in [0,1]$ we have that the map $\tau\mapsto g(a,\tau,0)$ is strictly decreasing by \ref{it:taudecr} of Lemma~\ref{lem:propfgrealr}, hence, using that $g(a,\tau,\eta)\ge g(a,\tau,0)$ and \ref{it:incfrealr}-\ref{it:decr} of Lemma~\ref{lem:propfgrealr}, it follows that
\begin{equation}
\label{eq:fgexpsma}
\sup_{\xi\in \widetilde{\Gamma}} |e^{Nf(\xi)}|+\sup_{a\in \widetilde{\Lambda}}| e^{-Ng(a,\tau,\eta)}|\le \sup_{\xi\in \widetilde{\Gamma}} |e^{Nf(\xi)}|+\sup_{a\in \widetilde{\Lambda}}| e^{-Ng(a,1,0)}|\lesssim e^{-Nf(N^\rho)},
\end{equation}
with
\begin{equation}
\label{eq:ovovag}
f(N^\rho)=\frac{\delta}{N^\rho}+\frac{1}{2N^{2\rho}}+\mathcal{O}(N^{-3\rho}+\delta N^{-2\rho}).
\end{equation}
In order to estimate the regime $a\in\widetilde{\Lambda}$, we split the computation into two cases: $(a,\xi)\in\widetilde{\Lambda}\times\widetilde{\Gamma}$ and $(a,\xi)\in\widetilde{\Lambda}\times(\Gamma\setminus\widetilde{\Gamma})$. Then, by~\eqref{eq:fgexpsma}--\eqref{eq:ovovag} it follows that the integral in the regime $(a,\tau,\xi)\in \widetilde{\Lambda}\times [0,1]\times \widetilde{\Gamma}$ is bounded by $N^{10\rho}e^{-N^{1-2\rho}/2}$. Note that in the regime $(a,\tau,\xi)\in \widetilde{\Lambda}\times [0,1]\times (\Gamma\setminus\widetilde{\Gamma})$ we have $|\xi|\ge N^\rho$. Hence, by the explicit form of $G_{1,N}, G_{2,N}$ in \eqref{eq:newbetG}, using the bound in \eqref{eq:fgexpsma} for $e^{-Ng(a,\tau,\eta)}$, that $|\xi|\ge N^\rho$ and so Lemma~\ref{lem:apriorbl} to bound the regime $\Gamma\setminus\widetilde{\Gamma}$, we conclude that the integral over $(a,\tau,\xi)\in \widetilde{\Lambda}\times [0,1]\times (\Gamma\setminus\widetilde{\Gamma})$ is bounded by $N^{2+10\rho}(N^{1/2}+(N\delta))e^{-N^{1-2\rho}/2}$.

Next, we consider the integral over $(a,\tau,\xi)\in (\Lambda\setminus \widetilde{\Lambda})\times [0,N^{\rho/2} a^{-1}]\times \Gamma$. Note that in this regime $a\ge N^\rho$. Since $g(a,\tau,\eta)\ge g(a,\tau,0)$ and $\tau\mapsto g(a,\tau,0)$ is strictly decreasing by \ref{it:taudecr} of Lemma~\ref{lem:propfgrealr}, we have that
\begin{equation}
\label{eq:hoplaes}
e^{-Ng(a,\tau,\eta)}\le e^{-Ng(a,N^{\rho/2}a^{-1},0)},
\end{equation}
where
\begin{equation}
\label{eq:ovtau}
g(a,N^{\rho/2}a^{-1},\eta)=Ea+\frac{\delta}{N^{\rho/2}}+\frac{1}{N^\rho}+\mathcal{O}(N^{-3\rho/2}+\delta N^{-\rho}).
\end{equation}
Additionally, using the explicit expression of $G_N$ in \eqref{eq:newbetG}, the bound \eqref{eq:fgexpsma} on $e^{Nf(\xi)}$ for the regime $\xi\in \widetilde{\Gamma}$, and Lemma~\ref{lem:apriorbl} for $\xi\in \Gamma\setminus \widetilde{\Gamma}$, we get
\begin{equation}
\label{eq:bassus}
\left| \int_\Gamma \diff \xi G_N(a,\tau,\xi,z)\xi^2 e^{Nf(\xi)}\right|\lesssim C(N,\eta,a,\tau),
\end{equation}
where
\[
C(a,\tau)=C(N,\eta,a,\tau)\defeq(N^{1/2}+N\delta)\left(N^2\eta^2+N^2\tau+\frac{N^2}{a}+\frac{N^2}{a^2\tau}\right).
 \]
Thus, using \eqref{eq:bassus} and the explicit form of $g(a,\tau,\eta)$ in \eqref{bosonphfAr}, for $\tau\in [0,N^{\rho/2} a^{-1}]$ we have
\begin{equation}
\label{eq:trick}
\begin{split}
\left| \frac{a}{\tau^{1/2}}e^{-Ng(a,\tau,\eta)}\int_\Gamma \diff \xi  G_N(a,\tau,\xi,z)\xi^2e^{Nf(\xi)}\right|&\lesssim C(a,\tau)e^{-(N-2)g(a,\tau,\eta)}\frac{a^3\tau}{\tau^{1/2}(1+2a+a^2\tau)} \\
&\quad \lesssim C(a,\tau)a^2\tau^{1/2} e^{-(N-2)g(a,\tau,0)} e^{-\frac{(N-2)\eta^2a^2}{1+2a+a^2\tau}} \\
&\quad \lesssim C(a,\tau)a^2\tau^{1/2} e^{-\frac{1}{2}N^{1-2\rho}} e^{-(N-2)Ea-\frac{(N-2)\eta^2a^2}{1+2a+a^2\tau}}.
\end{split}
\end{equation}
 
Hence, integrating~\eqref{eq:trick} with respect to $(a,\tau)$, and using that in the regime $\tau\in [0,N^{\rho/2}a^{-1}]$ it holds
\[
N^{-\rho/2}a\lesssim\frac{a^2}{1+2a+a^2\tau}\lesssim a,
\]
we conclude that  the integral over $(a,\tau,\xi)\in (\Lambda\setminus \widetilde{\Lambda})\times [0,N^{\rho/2} a^{-1}]\times\Gamma$ is bounded by $N^{5/2+5\rho} (N^{1/2}+N\delta) E^{-1/2} e^{-N^{1-2\rho}}e^{-N\eta^2/2}$.

Finally, in order to conclude the bound in~\eqref{eq:expsmalltripintr}, we are left with the estimate of the integral over $(a,\tau,\xi)\in (\Lambda\setminus \widetilde{\Lambda})\times [N^{\rho/2} a^{-1},1]\times\widetilde{\Gamma}$. In this regime, using the bound in \eqref{eq:fgexpsma}  on $e^{Nf(\xi)}$ for $\xi\in\widetilde{\Gamma}$, and Lemma~\ref{lem:doubleatauintr} to estimate the integral over $(a,\tau)\in (\Lambda\setminus \widetilde{\Lambda})\times [N^{\rho/2} a^{-1},1]$, we get the bound $N^{3/2+10\rho}  e^{-N^{1-2\rho}/2}(E^{-1/2}e^{-N\eta^2/2}+(N^{1/2}\vee(N\delta))(1+|\log|NE^{2/3}|))$. This concludes the proof of~\eqref{eq:expsmalltripintr}.
\end{proof}
\end{lemma}

\begin{proof}[Proof of Theorem~\ref{thm real} in the case \(\delta\ge 0\)]
Using Lemma~\ref{lem:expsmallrealr} we remove the regime $a\le N^\rho$, $|\xi|\le N^\rho$ or $\tau\in [0, N^{\rho/2}a^{-1}]$ in \eqref{eq:firstchr}. Then, using the expansion for $G_{1,N}$ and $G_{2,N}$ in \eqref{expfxi2AAr}-\eqref{newexpansionA1Ar} in the remaining regime of \eqref{eq:firstchr}, combining Lemma~\ref{lem:apriorbl} and Lemma~\ref{lem:doubleatauintr} we conclude Theorem~\ref{thm real}.
\end{proof}

\subsection{Case $-CN^{-1/2}\le \delta<0$.}
Now we summarize the necessary changes for the case \(\delta<0\). Similarly to the case $0\le \delta\le 1$, all along this section we assume that $E'<0$ in~\eqref{eq:fffr}--\eqref{bosonphfAr}, i.e.~$E'=-E$ with $0\le E\lesssim N^{-3/2}$.

Let $x_*$ be the real stationary point of $f$, i.e.~$x_*$ at leading order is given by 
\[
x_*\approx \begin{cases}
3|\delta|^{-1}/2 &\text{if}\quad E\ll |\delta|^3, \\
\mu(c)E^{-1/3} &\text{if}\quad E=c|\delta|^3, \\
E^{-1/3} &\text{if}\quad E\gg |\delta|^3,
\end{cases}
\]
for some function $\mu(c)>0$ and any fixed constant $c>0$ independent of $N$, $E$, and $\delta$. As in the complex case, we can treat the regime $0< -\delta \lesssim N^{-1/2}$ similarly to the regime $0\le \delta\lesssim  N^{-1/2}$, since for $|\delta|\lesssim N^{-1/2}$ the only $\delta$-dependent terms, i.e.~the term $\delta a^{-1}$ in the expansion of $g(a,1,\eta)=f(a)$ in~\eqref{expfxi1AAr} and the term $(1-\tau)\delta (a\tau)^{-1}$ in the expansion of $g(a,\tau,\eta)-g(a,1,\eta)$, do not play any role in the estimates of the $(a,\tau)$ integral in the regime $a\ge N^\rho$, $\tau\in [N^{\rho/2}a^{-1},1]$. Note that this is also the case for $0\le \delta\le N^{-1/2}$, when the estimates~\eqref{eq:impbtau2r} and~\eqref{eq:boddoubintr} were derived. For this reason, unlike the case $0\le \delta\le 1$, in the present case, $\delta<0$, $|\delta|\lesssim N^{-1/2}$ and $E\lesssim c(N)$, we do not deform the $\xi$-contour through the leading order of the saddle $x_*$, but we always deform it as $\Gamma=\Gamma_{\xi_*}\defeq\Gamma_{1,\xi_*}\cup \Gamma_{2,\xi_*}$, where $\xi_*\defeq E^{-1/3}$, with $\Gamma_{1,\xi_*}, \Gamma_{2,\xi_*}$ defined in~\eqref{eq:defycon} replacing $z_*$ by $\xi_*$. Note that we could have done the same choice in the case $0\le \delta\lesssim N^{-1/2}$, but not for the regime $N^{-1/2}\ll \delta\le 1$, hence, in order to treat the regime $0\le \delta\le 1$ in the same way for any $\delta$, in Section~\ref{sec:posredelta} we deformed the contour $\Gamma$ through~\eqref{eq:newstatpoi}. For any $E\lesssim c(N)$ and $0< -\delta \lesssim N^{-1/2}$ we have $\xi_* \gg 1$, hence we prove that the main contribution to~\eqref{eq:firstchr} comes from the regime when $a, |\xi|\ge N^{\rho}$. Moreover, similarly to the case $0\le \delta\le 1$, the contribution to~\eqref{eq:firstchr} in the regime $(a,\tau,\xi)\in (\Lambda\setminus \widetilde{\Lambda})\times [0, N^{\rho/2}a^{-1}]\times (\Gamma\setminus \widetilde{\Gamma})$ will be exponentially small. Hence, in order to estimate the leading order of~\eqref{eq:firstchr}, we expand $f$, $g(\cdot,1,\eta)$ and $G_N$ for large $a$ and $|\xi|$ arguments as in~\eqref{expfxi1AAr}--\eqref{newexpansionA1Ar}.

The phase functions $f$ and $g$, defined in~\eqref{eq:fffr} and~\eqref{bosonphfAr}, respectively, satisfy the properties \ref{it:reimfrealr} and \ref{it:incf} of Lemma~\ref{lem:propfgrealr}, but not the ones in \ref{it:taudecr} and \ref{it:decr}. Instead, it is easy to see that the following lemma holds true.

\begin{lemma}
\label{lem:negrealdeltafg}
Let $f$ and $g$ be the phase functions defined in~\eqref{eq:fffr} and~\eqref{bosonphfAr}, respectively. Then, the following properties hold true:

\begin{enumerate}[(i')]
\setcounter{enumi}{2}
\item \label{it:fin2} For any $a\in[0,+\infty)$, we have that $g(a,\tau,\eta)\ge g(a,\tau,0)$ and that
\[
e^{-Ng(a,\tau,0)}\lesssim e^{-Ng(a,\tau_0,0)},
\]
for any fixed $\tau_0\in [0,1]$ and any $\tau\in [0,\tau_0]$.

\item \label{it:fin1} The function $a\mapsto g(a,1, 0)$ is strictly decreasing on $[0, |\delta|^{-1}/2]$.

\end{enumerate}
\end{lemma}

Since $|\delta|\lesssim N^{-1/2}$, by \ref{it:fin1} of Lemma \ref{lem:negrealdeltafg} it clearly follows that the function $a\mapsto g(a,1,\eta)$ is strictly decreasing on $[0,N^\rho]$. Note that for $|\delta|\lesssim N^{-1/2}$ we have
\begin{equation}
\label{eq:convboundnegdel}
e^{-N(1-\tau)\left[\frac{1}{a^2\tau^2}+\frac{\delta}{a\tau} \right]}\lesssim e^{-\frac{(1-\tau)N}{2a^2\tau^2}}, \quad e^{-N[\frac{\delta}{a}+\frac{1}{2a^2}]}\lesssim e^{-\frac{N}{4a^2}},
\end{equation}
for any $a\in [N^\rho,+\infty)$ and $\tau\in [N^{\rho/2}a^{-1},1]$. Using~\eqref{eq:convboundnegdel}, inspecting the proof of~\eqref{eq:impbtau2r},~\eqref{eq:boddoubintr} in Lemma~\ref{lem:tauintr} and Lemma~\ref{lem:doubleatauintr}, respectively, and noticing that in the regime $0\le \delta \le N^{-1/2}$ the sign of $\delta$ did not play any role we conclude that the bounds ~\eqref{eq:impbtau2r},~\eqref{eq:boddoubintr} hold true for the case $\delta<0$, $|\delta|\lesssim N^{-1/2}$ as well. Then, similarly to the case $0\le \delta\le 1$, by \ref{it:reimfrealr}--\ref{it:incf} of Lemma~\ref{lem:propfgrealr} and \ref{it:fin2}--\ref{it:fin1} of Lemma~\ref{lem:negrealdeltafg}, using the bound in~\eqref{eq:boddoubintr} to estimate the $(a,\tau)$-integral in the regime $(a,\tau)\in [N^\rho,+\infty)\times [N^{\rho/2}a^{-1},1]$ and the ones in Lemma~\ref{lem:apriorbl} to estimate the $\xi$-integral in the regime $|\xi|\ge N^\rho$ we conclude that the contribution to~\eqref{eq:firstchr} in the regime when either $a\in [0,N^\rho]$ or $|\xi|\le N^\rho$ and in the regime $[0,N^{\rho/2}a^{-1}]$ is exponentially small, i.e.~Lemma~\ref{lem:expsmallrealr} holds true once $\delta$ is replaced by $|\delta|$ everywhere.
\begin{proof}[Proof of Theorem~\ref{thm real} in the case \(-C N^{-1/2}\le \delta<0\)]
By combining~\eqref{eq:boddoubintr}, Lemma~\ref{lem:apriorbl} and Lemma~\ref{lem:expsmallrealr}, using the expansion of $G_N$ in~\eqref{expfxi2AAr}--\eqref{newexpansionA1Ar}, we conclude the proof of Theorem \ref{thm real} also in the case \(-C N^{-1/2}\le \delta<0\).
\end{proof}

\appendix

\section{Superbosonisation formula for meromorphic functions}\label{app:reg}
The superbosonisation formulas~\cite[Eq.~(1.10) and (1.13)]{MR2430637}
(see also~\cite[Corollary 2.6]{MR3264247} for more precise conditions)
 are stated under the condition that 
  $$
   F( \Phi^*\Phi) = F\begin{pmatrix} \langle s, s\rangle &  \langle s, \chi \rangle \cr 
   \langle \chi, s\rangle &  \langle \chi, \chi \rangle \end{pmatrix}, 
 $$
    viewed as a function of  four independent variables, is holomorphic and decays faster than any 
inverse power at real $+\infty$ in the $\langle s, s\rangle$ variable (for definiteness, we discuss the complex case;
the argument for the real case is analogous). Our function $F$ defined in~\eqref{susyused} 
has a pole at $\langle s, s\rangle = \ii N$ and $\langle \chi, \chi\rangle = \ii N$ using 
the definitions~\eqref{def:supertrace}--\eqref{def:superinverse} after expanding the inverse of the  matrix $1+ \frac{\ii}{N} \Phi^*\Phi$
in the Grassmannian variables but this pole is far away from the integration domain on both sides of~\eqref{eq cplx superbosonization}.
We now outline a standard approximation procedure to verify the superbosonisation formula for such meromorphic functions; for simplicity
we consider only our concrete function from~\eqref{susyused}.
 
In the first step notice that the integration at infinity on the non-compact domain for the boson-boson variable is absolutely convergent on both sides
as guaranteed by the $\exp( \ii w\sTr \Phi^\ast\Phi)$ regularization, since $\Im w>0$.

Second, in the LHS of~\eqref{eq cplx superbosonization} using Taylor expansions, we expand $F$  into a finite polynomial
in the  Grassmannian variables with meromorphic coefficient
functions in the variable  $\langle s, s\rangle$. Algebraically, we perform  exactly the same expansion in the RHS of~\eqref{eq cplx superbosonization}.
For the fermionic variables $\sigma, \tau$ these expansions naturally terminate after finitely many terms. From the 
formulas~\eqref{H-z cplx superbos} it is clear that 
only  the geometric expansion $(1+y)^{-1} = 1- y + y^2 - \ldots$ may result in an  infinite power series instead of a finite polynomial.
However, owing  to the contour integral in $y$ and that the integrand has a pole of at most finite order ($\approx N$) at zero,
we may replace this power series with its  finite truncation without changing the value of the RHS of~\eqref{eq cplx superbosonization}.
We choose the order of truncation sufficiently large that the remaining formula contains all non-zero terms on both sides.
We denote this new truncated function by $\widetilde F$.

Now we are in the situation where on both sides of~\eqref{eq cplx superbosonization}, with $\widetilde F$ replacing $F$, we have the same
 finite polynomial in the variables $\langle s, \chi\rangle, \langle \chi, s\rangle$ and $\langle \chi, \chi\rangle$ in the LHS as 
 in the variables  $\sigma, \tau, y$ in the RHS,
 with coefficients that are meromorphic in $\langle s, s\rangle$, resp. in $x$. All coefficient functions $h_k(x)$
 are analytic in a neighborhood of the positive real axis (their possible pole is at $-1$) and they
 have an exponential decay
 $\approx \exp{(-(\Im w) \langle s, s\rangle)}$ in the LHS, resp. $\exp{(-(\Im w) x)}$ in the RHS, at infinity from the regularization
 observed in the first step.
  
 Finally, in the third step, dropping the $k$ index temporarily, we write each coefficient function  as $h(x)= g(x) e^{- \alpha x}$ with $\alpha =\frac{1}{2}\Im w$.
 For any given $\epsilon>0$
 we approximate $g(x)$ via classical (rescaled) Laguerre polynomials $p_{n}(x)$ of degree $n$ with weight function $e^{- \alpha x}$ such that 
 $\int_0^\infty | g(x)- p_n(x)|^2 e^{- \alpha x} \diff x \le (\epsilon \alpha)^2$, where $n$ depends on $\epsilon$ and $\Im w$. By completeness
 of the Laguerre polynomials in $L^2(\mathbb{R}_+,  e^{- \alpha x} \diff x)$
 and by $\int |g(x)|^2 e^{- \alpha x} \diff x = \int |h(x)|^2 e^{\alpha x} \diff x<\infty$ such approximating polynomial exists.
 Therefore, with a Schwarz inequality, we have
 $$
    \int_0^\infty\big| h(x) - p_n (x) e^{-\alpha x}\big| \diff x=  \int_0^\infty | g(x)- p_n(x)| e^{- \alpha x} \diff x \le \epsilon.
 $$
 Since there are only finitely many coefficient functions $h(x)=h_k(x)$ in $\widetilde F$,
 we can replace each of them with an entire function (namely with a polynomial
 times $e^{-\alpha x}$) with at most an $\epsilon$ error in the RHS of~\eqref{eq cplx superbosonization}. The same estimates
 hold on the LHS.  But for these replacements the superbosonisation formula~\cite[Eq.~(1.10)]{MR2430637} is applicable since 
 the new functions are entire. The error is at most $\epsilon$ on both sides, but this argument is valid for arbitrary $\epsilon>0$.
 This proves the superbosonisation formula for the function~\eqref{susyused}.

\section{Explicit formulas for the real symmetric integral representation}\label{appendix poly}
Here we collect the explicit formulas for the polynomials of \(a,\xi,\tau\) in the definition of \(G_N\) in~\eqref{realsusyexplAAr}.
\[ \begin{split}
p_{2,0,0}&\defeq a^4 \tau ^2+2 a^3 \xi  \tau +4 a^3 \tau -a^2 \xi ^2 \tau +4 a^2 \xi ^2+8 a^2 \xi +2 a^2 \tau \\
&\qquad +4 a^2+2 a \xi ^3+8 a \xi ^2+10 a \xi +4 a+\xi ^4+4 \xi ^3+6 \xi ^2+4 \xi +1   \\
p_{1,0,0}&\defeq - a^4 \xi \tau ^2+a^4 \tau ^2-2 a^3 \xi ^2 \tau -2 a^3 \xi  \tau +4 a^3 \tau -a^2 \xi ^3 \tau -3 a^2 \xi ^2 \tau \\
&\qquad -2 a^2 \xi  \tau +4 a^2 \xi +2 a^2 \tau +4 a^2+2 a \xi ^2+6 a \xi +4 a+\xi ^3+3 \xi ^2+3 \xi +1\\
p_{2,2,0}&\defeq 4 (a+1) \left(a^2 \tau +a \xi  \tau +2 a \tau +\xi ^2+2 \xi +1\right) \\
p_{1,2,0}&\defeq 4 (a+1)  \left(a^2 \tau +a \xi  \tau +2 a \tau +\xi +1\right) \\
p_{2,0,1}&\defeq 2  \bigl(a^3 \tau ^2+2 a^2 \xi  \tau +4 a^2 \tau +2 a \xi ^2+2 a \xi  \tau  \\
&\qquad\qquad+4 a \xi +3 a \tau +2 a+\xi ^3+4 \xi ^2+5 \xi +2\bigr)\\
p_{1,0,1}&\defeq 2 \bigl(a^3 \tau ^2+2 a^2 \xi  \tau +4 a^2 \tau +a \xi ^2 \tau +3 a \xi  \tau \\
&\qquad\qquad+2 a \xi +3 a \tau  +2 a+\xi ^2+3 \xi +2\bigr) \\
 p_{2,2,1}&\defeq 4 (a+1)  (a+\xi +2) \\
 p_{2,0,2}&\defeq a^2 \tau +2 a \xi +4 a+\xi ^2+4 \xi +4 
\end{split} \]

\section{Comparison with the contour-integral derivation}\label{app}
\begin{figure}
    \centering
    \setlength{\figurewidth}{.8\linewidth}
    \setlength{\figureheight}{.4\linewidth}
    \pgfplotsset{xmin=0.001, xmax=15, ymin=0, ymax=.45}
\begin{tikzpicture}

\begin{axis}[%
width=0.951\figurewidth,
height=\figureheight,
at={(0\figurewidth,0\figureheight)},
scale only axis,
xmin=0.001,
xmax=15,
ymin=0.001,
ymax=.45,
axis background/.style={fill=white},
legend style={legend cell align=left, align=left, draw=white!15!black},
samples=100
]

\addplot [color=black, dotted,thick,domain=.1:15] {sqrt(3)*x^(-1/3)/(2*pi)};
\addlegendentry{\(\sqrt{3}\lambda^{-1/3}/2\pi\)};

\addplot [color=black, dashed,thick,domain=0:2] { (1- x^2/4)/sqrt(2*pi) };
\addlegendentry{\((1-\lambda^2/4)/\sqrt{2\pi}\)};

\addplot [color=black,thick]
  table[row sep=crcr]{%
0.00316954644277246	0.398941278937307\\
0.0712327891289343	0.398451109655666\\
0.151515151515152	0.396794311332204\\
0.229171816706073	0.394188520756326\\
0.303030303030303	0.390890793049353\\
0.375806063146358	0.386944076552055\\
0.454545454545455	0.381983806091781\\
0.536387286115301	0.376164224438987\\
0.606060606060606	0.370748711456848\\
0.688727728456924	0.3638551079348\\
0.757575757575758	0.357786382408501\\
0.829132331411569	0.351218389555917\\
0.909090909090909	0.343626086977689\\
0.980394150934863	0.336681560794617\\
1.06060606060606	0.328728891790419\\
1.1433466978198	0.3204272465716\\
1.21212121212121	0.313491474305574\\
1.29167188203165	0.30547074770385\\
1.36363636363636	0.298250242801581\\
1.44170772990087	0.290489276952188\\
1.51515151515152	0.283285641669316\\
1.58863630660007	0.276197615930242\\
1.66666666666667	0.268826610938283\\
1.74085891750134	0.261987626855268\\
1.81818181818182	0.255055080823846\\
1.89530290356778	0.248355737329479\\
1.96969696969697	0.242110450528843\\
2.04981469590853	0.235635770734964\\
2.12121212121212	0.230094015797155\\
2.19411210751799	0.224664234722613\\
2.27272727272727	0.21907327648684\\
2.34738705833692	0.214021559063475\\
2.42424242424242	0.209086075682618\\
2.49847254571465	0.204574942531612\\
2.57575757575758	0.200144578609279\\
2.65534067859418	0.195864432518707\\
2.72727272727273	0.192239029486012\\
2.79794707114122	0.188898476791331\\
2.87878787878788	0.185341279819627\\
2.95662137660303	0.182176146437653\\
3.03030303030303	0.179408074021196\\
3.08628945872237	0.177449008423168\\
3.18181818181818	0.17438408538788\\
3.25461307958797	0.172274964542054\\
3.33333333333333	0.170204690430444\\
3.41507276735304	0.168275823830153\\
3.48484848484848	0.166798487091054\\
3.56186167916055	0.16533932452182\\
3.63636363636364	0.164089547545964\\
3.70840974501459	0.163023381011652\\
3.78787878787879	0.161999435290446\\
3.8700015060841	0.161097119828106\\
3.93939393939394	0.160448957127268\\
4.01006658699306	0.159888195628167\\
4.09090909090909	0.159359693331392\\
4.16633616421699	0.158965578634425\\
4.24242424242424	0.15865530197137\\
4.32450007058867	0.158408398931878\\
4.39393939393939	0.158262601745023\\
4.46561926103384	0.158165398717483\\
4.54545454545455	0.158112463974901\\
4.62176771571574	0.158108151247447\\
4.6969696969697	0.158140516673681\\
4.76367692628445	0.158193924887049\\
4.84848484848485	0.158287678628666\\
4.92654327766151	0.158392272189544\\
5	0.158500536458296\\
5.07767575564152	0.158619467308257\\
5.15151515151515	0.158731582666969\\
5.22119388375511	0.158832297216841\\
5.3030303030303	0.158939325832757\\
5.3892618372432	0.159033343922643\\
5.45454545454545	0.159088294561074\\
5.52736959969308	0.159130063734969\\
5.60606060606061	0.159148934566333\\
5.68332019653026	0.159137962540558\\
5.75757575757576	0.159097431459465\\
5.81099920989544	0.159048943029498\\
5.90909090909091	0.158915458439833\\
5.97554797805416	0.15879092851942\\
6.06060606060606	0.158589866730421\\
6.11926897547302	0.158423309259289\\
6.21212121212121	0.15811232738878\\
6.29218396051879	0.157797137503193\\
6.36363636363636	0.157478942920267\\
6.44157299053722	0.157092330974256\\
6.51515151515152	0.156689823478182\\
6.58990910186298	0.156244103232718\\
6.66666666666667	0.155748656072247\\
6.73420850091908	0.155281797371933\\
6.81818181818182	0.154662257676725\\
6.89348082544772	0.154071173092174\\
6.96969696969697	0.153440131030522\\
7.03933523743575	0.152836057926373\\
7.12121212121212	0.152094025017038\\
7.19545064115951	0.151393262686087\\
7.27272727272727	0.150637504602817\\
7.34339227133163	0.149924608884336\\
7.42424242424242	0.149085539059313\\
7.50108120084363	0.148267080480128\\
7.57575757575758	0.147454108720346\\
7.65071204886719	0.146622757004281\\
7.72727272727273	0.145759835993522\\
7.80939240518357	0.144821285255218\\
7.87878787878788	0.144019640993714\\
7.95543721323693	0.143127197370219\\
8.03030303030303	0.142250430612226\\
8.11294423821129	0.141279176098976\\
8.18181818181818	0.140468811767713\\
8.26081035043517	0.13954043182483\\
8.33333333333333	0.138690839040409\\
8.41689753256178	0.137717402830575\\
8.48484848484848	0.13693179338009\\
8.55330427860404	0.136147126360418\\
8.63636363636364	0.135205995676804\\
8.71267533112345	0.13435358877737\\
8.78787878787879	0.133526626211791\\
8.86192057589164	0.132726506303954\\
8.93939393939394	0.131905639529627\\
9.0035593293496	0.131239389573767\\
9.09090909090909	0.130353622333005\\
9.16185749874208	0.129653232848459\\
9.24242424242424	0.128879743421511\\
9.30650316440612	0.128281843800002\\
9.39393939393939	0.127491703448038\\
9.46747152562203	0.126850921280846\\
9.54545454545454	0.12619571569279\\
9.62693925906048	0.12553851310968\\
9.6969696969697	0.124996508819436\\
9.76222563941117	0.124510707076459\\
9.84848484848485	0.123897349238282\\
9.9106642581695	0.123475702060891\\
10	0.122900083687193\\
10.0798773341579	0.122415570547981\\
10.1515151515152	0.122005197641113\\
10.2270909393546	0.121596895910065\\
10.3030303030303	0.121211888310146\\
10.3768301258959	0.120861712987609\\
10.4545454545455	0.120518149954577\\
10.5270167682438	0.120220660643031\\
10.6060606060606	0.11992086942707\\
10.6778773962474	0.11967030865688\\
10.7575757575758	0.119415930527681\\
10.8419649882792	0.119172914737907\\
10.9090909090909	0.118998321423357\\
10.9798833581984	0.118831529958095\\
11.0606060606061	0.118662253110464\\
11.1356550447285	0.118524022767733\\
11.2121212121212	0.118401272930652\\
11.2840220541474	0.118301695217276\\
11.3636363636364	0.118208383906229\\
11.4402515124343	0.118134421402639\\
11.5151515151515	0.118076198336847\\
11.5979388056712	0.118026771317847\\
11.6666666666667	0.117996870731738\\
11.7423583154329	0.117974613922126\\
11.8181818181818	0.117962574352456\\
11.8947727225492	0.117959807599018\\
11.969696969697	0.117965243249949\\
12.031526201574	0.117975100932561\\
12.1212121212121	0.117996859657218\\
12.2015864310618	0.118022665581779\\
12.2727272727273	0.118049510202769\\
12.3532828107084	0.118083405051691\\
12.4242424242424	0.118115477026102\\
12.5026289554529	0.118152367697521\\
12.5757575757576	0.118187318718556\\
12.6488367044864	0.118221956964656\\
12.7272727272727	0.118257943936587\\
12.7972487546594	0.118288284422388\\
12.8787878787879	0.118320676267848\\
12.9537058684846	0.118346852865196\\
13.030303030303	0.118369311490737\\
13.1049810786162	0.118386337245938\\
13.1818181818182	0.118398153245947\\
13.2586908469643	0.118403576622013\\
13.3333333333333	0.11840211025426\\
13.4070874158504	0.118393632109291\\
13.4848484848485	0.118376614527203\\
13.5590910758182	0.118352160954669\\
13.6363636363636	0.118317753650871\\
13.7110753804269	0.118275414171149\\
13.7878787878788	0.118222230248301\\
13.866281099138	0.118157505707488\\
13.9393939393939	0.118087383647547\\
14.008118774177	0.11801268530589\\
14.0909090909091	0.117911182755464\\
14.1681731642486	0.117804931262157\\
14.2424242424242	0.117692219857971\\
14.3270300728278	0.117551025654636\\
14.3939393939394	0.117429699145836\\
14.4696352776963	0.117282142611324\\
14.5454545454545	0.117123414438317\\
14.6295440598239	0.116934654882497\\
14.6969696969697	0.116773723599187\\
14.7718205248304	0.11658520707094\\
14.8484848484848	0.116381519872966\\
14.939248476034	0.116126770297237\\
15	0.115948194145057\\
};
\addlegendentry{\(K(\lambda,\lambda)\)}

\end{axis}
\end{tikzpicture}%
\caption{Plot of the \(1\)-point function \(K(\lambda,\lambda)=\pi^{-1}\Im q_0(\lambda)\) in the complex case with \(|z|=1\). The dotted and dashed lines show the large \(\lambda\) and small \(\lambda\) asymptotes, respectively.\label{fig 1pt}}
\end{figure}
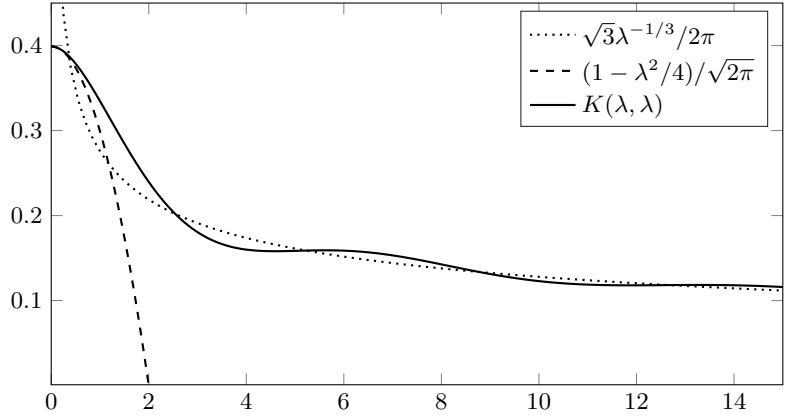
In~\cite{MR2162782} the correlation kernel of \((X-z)(X-z)^\ast\) for complex Ginibre matrices \(X\) has been derived using contour-integral methods. Earlier, the joint eigenvalue density for the general \emph{Laguerre ensemble} had been obtained in the physics literature~\cite{MR1419177,MR1413913} via supersymmetric methods, see also~\cite{10053982} with orthogonal polynomials. Adapting~\cite{MR2162782} to our scaling, and choosing \(y_i=\pm1\), it follows from~\cite[Theorem~7.1]{MR2162782} that for \(\abs{z}=1\) the rescaled kernel \(K_N(N^{-3/2}\lambda,N^{-3/2}\mu)\) is given by 
\begin{equation}\label{peche kernel} 
\begin{split} 
    \frac{N^3}{\ii \pi}\int_\Gamma \diff x\int_\gamma \diff y K_B(2N^{1/4} x\sqrt{\lambda},2N^{1/4} y\sqrt{\mu}) e^{N(h(y)-h(x))}\Bigl(1-\frac{1}{1-x^2}\frac{1}{1-y^2}\Bigr)xy,\\
    K_B(x,y)\defeq\frac{x I_1(x)I_0(y)-y I_0(x)I_1(y)}{x^2-y^2},\quad h(x)=x^2+\log(1-x^2)=-\frac{x^4}{2}+\landauO{x^5},
\end{split}
\end{equation}
where \(I_0,I_1\) are the modified Bessel function of \(0\)-th and \(1\)-st kind. The contour \(\Gamma\) is any contour encircling \([-1,1]\) in a counter-clockwise direction (in contradiction to the contours depicted in~\cite[Figure 8.1]{MR2162782}) and the contour \(\gamma\) is composed of two straight half-lines \([0,\ii\infty)\) and \([0,-\ii\infty)\). The main contribution in~\eqref{peche kernel} comes from the \(\abs{x}\lesssim N^{-1/4}\) and \(\abs{y}\lesssim N^{-1/4}\) regime which motivates the change of variables \(x\mapsto N^{-1/4}x,y\mapsto N^{-1/4}y\). Together with the expansion of \(1-(1-x^2)^{-1}(1-y^{2})^{-1}=-x^2-y^2+\landauO{x^4+y^4}\) it follows that
\[K_N(N^{-3/2}\lambda,N^{-3/2}\mu)\approx N^{3/2} K(\lambda,\mu), \]
where 
\[K(\lambda,\mu) \defeq \frac{\ii}{\pi} \int_{\Gamma'}\diff x\int_{\gamma} \diff y K_B(2x\sqrt{\lambda},2y\sqrt{\mu}) e^{x^4/2-y^4/2} xy(x^2+y^2) \]  
and \(\Gamma'\) consists of four straight half-lines \((e^{\ii\pi/4}\infty,0]\), \([0,e^{3\ii\pi/4}\infty)\), \((e^{5\ii\pi/4}\infty,0]\), \([0,e^{7\ii\pi/4}\infty)\). 

We now compare the limiting \(1\)-point function \(K(\lambda,\lambda)\) with the asymptotic expansion we derived in Theorem~\ref{thm cplx}, which in the case \(\abs{z}=1\), i.e.~\(\delta=0\), simplifies to 
\[q_0(\lambda)=\frac{\lambda^{1/3}}{2\pi\ii}\int \diff x\oint \diff y e^{\lambda^{2/3}\bigl(- y +1/(2y^2) + x -1/(2x^2)\bigr)  }\Bigl(\frac{1}{x^3}+\frac{1}{x^2 y}+\frac{1}{xy^2}\Bigr).\]
The resulting \(1\)-point function, given by \(\pi^{-1}\Im q_0(\lambda)\), coincides precisely with \(K(\lambda,\lambda)\) and is plotted in Figure~\ref{fig 1pt}.

\printbibliography%
\end{document}